%% file: article.tex
\documentclass[preprint,
3p,
times loads txfonts
]{elsarticle}

\input{commands}  % commands.tex

\journal{ }

\begin{document}
%\linenumbers

\begin{frontmatter}
\title{A positivity-preserving implicit-explicit scheme with high order polynomial basis for compressible Navier--Stokes equations}
\author[label1]{Chen Liu}\ead{liu3373@purdue.edu}
\author[label1]{Xiangxiong Zhang\corref{cor1}}\ead{zhan1966@purdue.edu}
\address[label1]{Department of Mathematics, Purdue University, 150 North University Street, West Lafayette, Indiana 47907.} 

\begin{abstract}

In this paper, we are interested in constructing a scheme solving compressible Navier--Stokes equations, with desired properties including high order spatial accuracy, conservation, and positivity-preserving of density and internal energy under a standard hyperbolic type CFL constraint on the time step size, e.g., 
$\Delta t=\mathcal O(\Delta x)$. Strang splitting is used to approximate convection and diffusion operators separately. For the convection part, i.e., the compressible Euler equation,  the high order accurate postivity-preserving Runge--Kutta discontinuous Galerkin method can be used. For the diffusion part, the equation of internal energy instead of the total energy is considered, and a first order semi-implicit time discretization is used for the ease of achieving positivity.  A suitable interior penalty discontinuous Galerkin method for the stress tensor can ensure the conservation of momentum and total energy for any high order polynomial basis. In particular, positivity can be proven with $\Delta t=\mathcal{O}(\Delta x)$ if the Laplacian operator of internal energy is approximated by the $\IQ^k$ spectral element method with $k=1,2,3$. So the full scheme with $\IQ^k$ ($k=1,2,3$) basis is conservative and positivity-preserving with $\Delta t=\mathcal{O}(\Delta x)$, which is robust for demanding problems such as solutions with low density and low pressure induced by high-speed shock diffraction. 
Even though the full scheme is only first order accurate in time, numerical tests indicate that higher order polynomial basis produces much better numerical solutions, e.g., better resolution for capturing the roll-ups during shock reflection. 
\end{abstract}

\begin{keyword}
compressible Navier--Stokes \sep discontinuous Galerkin \sep spectral element \sep implicit-explicit \sep high-order accuracy  \sep positivity-preserving

\vspace{.5\baselineskip}
\MSC 35L65 \sep 65M12 \sep 65M60 \sep 65N30
\end{keyword}
\end{frontmatter}

%% The major content of the artical.
\input{Content/introduction}

\input{Content/numerical}

\input{Content/experiments}

\section{Concluding remarks}\label{sec-remark}
\par In this paper, we have constructed  an implicit-explicit scheme with high order polynomial basis for solving the compressible NS equations. Our scheme preserves the local conservation of density,  global conservation of momentum and total energy, and positivity of density and internal energy, under a CFL constraint like $\Delta t = \mathcal{O}(\Delta x)$. Even though the time accuracy is at most first order, 
numerical tests suggest that the $\IQ^2$ scheme and $\IQ^3$ scheme are not only robust but also producing better numerical solutions than the low order $\IQ^1$ scheme. 
Numerical experiments also indicate that our $\IQ^3$ scheme with only positivity-preserving limiter produces satisfactory non-oscillatory solutions when physical diffusion is accurately resolved.

%% Acknowledgments.
\section*{Acknowledgments}
Research is supported by NSF DMS-2208518.

%% Appendix
\appendix
\input{Content/appendix}

%\section*{References}
\bibliographystyle{elsarticle-num}
\bibliography{bibliography} % Load bibliography.bib.

\end{document}

%% file: commands.tex
\usepackage{amsmath}
\usepackage{amsthm}
\usepackage{cases}
\usepackage{enumitem}              % Improved itemize & enumerate environments.
\usepackage{mathrsfs}
\usepackage{mathtools}
\usepackage[OMLmathsfit]{isomath}  % Loads also fixmath.sty (upper case italic greek letters, definition of \upDelta etc).
\usepackage{graphicx}
\usepackage{grffile}               % Fixes errors with images whose file names contain more than one dot, e.g., *.0000.png.
\usepackage[T1]{fontenc}

\usepackage{xcolor}
\providecommand*{\red}[1]{{{\leavevmode\color{red}#1}}}

\usepackage[xcolor=dvipdf, authormarkup=none]{changes}

\usepackage[cmintegrals]{newpxmath}  % The fitting math symbols for newpxtext.

\usepackage{bm}                      % Provides \boldsymbol that we use in \vec{}.

\usepackage{dsfont}                                              % Double striked symbols.
  \newcommand{\IQ}{\ensuremath\mathds{Q}}                        % Set of rational numbers.
  \newcommand{\IR}{\ensuremath\mathds{R}}                        % Set of real numbers.
\newcommand*{\setE}{\ensuremath{\mathcal{T}}}                    % Partition of the domain.
\newcommand*{\Gammah}{\Gamma_h}                                  % Interior faces.
\newcommand*{\Nel}{{N_\mathrm{el}}}                              % Number of elements.
\newcommand*{\Nloc}{{N_\mathrm{loc}}}                            % Number of local degrees of freedom.
\renewcommand*{\vec}[1]{{\boldsymbol{#1}}}                       % Vector.
\DeclareMathAlphabet{\mathbfsf}{\encodingdefault}{\sfdefault}{bx}{n}
\newcommand*{\vecc}[1]{\mathbfsf{#1}}                            % Tensor or matrix.
\newcommand*{\transpose}[1]{{#1}^\mathrm{T}}                     % Transposition.
\newcommand*{\normal}{\vec{n}}                                   % Unit outward normal.
\newcommand*{\dd}{\mathrm{d}}                                    % Differential d for ``int dx''.
\newcommand*{\grad}{\vec{\nabla}}                                % Gradient.
\renewcommand*{\div}{\vec{\nabla}\cdot}                          % Divergence.
\newcommand*{\laplace}{\upDelta}                                 % Laplace operator. 
\newcommand*{\strain}{{\boldsymbol{\varepsilon}}}                % Deformation tensor in NS equation.
\newcommand*{\llbrace}{\lbrace\hspace*{-0.18em}\vert}
\newcommand*{\rrbrace}{\vert\hspace*{-0.18em}\rbrace}
\newcommand*{\avg}[1]{\llbrace{#1}\rrbrace}                      % Average operator.
\newcommand*{\jump}[1]{\left\llbracket{#1}\right\rrbracket}      % Jump operator.
\newcommand*{\abs}[1]{\ensuremath{|#1|}}                         % Absolute value or volume.
\newcommand*{\norm}[2]{\|#1\|_{#2}}                              % Norm.
\newcommand*{\on}[2]{\left.#1\right\vert_{#2}}                   % restriction on a set, i.e., eg u|G.
\newcommand*{\Rey}{\mathrm{Re}}                                  % Reynolds number. 

\usepackage{multirow}
\usepackage{tabularx}                                            % Specify total length of tabular.
  \newcolumntype{R}{>{\raggedleft\arraybackslash}X}
  \newcolumntype{L}{>{\raggedright\arraybackslash}X}
  \newcolumntype{C}{>{\centering\arraybackslash}X}
\usepackage{booktabs}                                            % Improve tabulars, enable \top-, \mid-, \bottomrule.
\usepackage{rotating}                                            % Provide command for rotating table contents.
          
\newtheorem{lemma}{Lemma}
\newtheorem{theorem}{Theorem}
\newtheorem{remark}{Remark}
  

%% file: Content/introduction.tex
%%%%%%%%%%%%%%%%%%%%%%%%%%%%%%%%%%%%%%%%%%%%%%%%%%%%%%%%%%%%%%%%%%%%%%%%%%%%%%%%%%%%%%%%%%%%%%%%%%%%%%%%%%%%%%%%%%%%%%%
\section{Introduction}
\subsection{Motivation of positivity}
%%%%%%%%%%%%%%%%%%%%%%%%%%%%%%%%%%%%%%%%%%%%%%%%%%%%%%%%%%%%%%%%%%%%%%%%%%%%%%%%%%%%%%%%%%%%%%%%%%%%%%%%%%%%%%%%%%%%%%%
%\red{Remember to add entropy minimum principle}

%It is desired to construct robust, accurate, and efficient numerical algorithms, which poses significant challenges.

The compressible Navier--Stokes (NS) equations are one of the 
most popular and important models in gas dynamics as well as computational fluid dynamics applications.
The  equations in dimensionless form  on a bounded spatial domain $\Omega \subset \IR^d$ over the time interval $[0,T]$ are given  by:
\begin{subequations}\label{eq:CNS:model}
\begin{align}
\partial_t{\rho} + \div{(\rho\vec{u})} = 0 && \text{in}~[0,T]\times\Omega,\label{eq:CNS:model_1}\\
\partial_t{(\rho\vec{u})} + \div{(\rho\vec{u}\otimes\vec{u})} + \grad{p} - {\textstyle\frac{1}{\Rey}}\div{\vec{\tau}(\vec{u})} = \vec{0} && \text{in}~[0,T]\times\Omega,\\
\partial_t{E} + \div{((E+p) \vec{u})} + {\textstyle\frac{1}{\Rey}}\div{\vec{q}} - {\textstyle\frac{1}{\Rey}}\div{(\vec{\tau}(\vec{u})\vec{u})} = 0 && \text{in}~[0,T]\times\Omega,\label{eq:CNS:model_3}
\end{align}
\end{subequations}
where $\rho$, $\vec{u}$, $p$, and $E$ are the density, velocity, pressure, and total energy respectively, and $\Rey$ denotes the Reynolds number.
Let $\vec{m} = \rho\vec{u}$ denote the momentum, then the conservative variables are $\vec{U}=\transpose{[\rho,\vec{m},E]}$.
Assume the fluid is Newtonian, as well as the Stokes hypothesis, which states that the bulk viscosity equals to zero. 
Then the shear stress tensor is given by $\vec{\tau}(\vec{u}) = 2\strain(\vec{u}) - \frac{2}{3}(\div{\vec{u}})\vecc{I}$, where $\strain{(\vec{u})} = \frac{1}{2}(\grad{\vec{u}} + \transpose{(\grad{\vec{u}})})$ and $\vecc{I}\in\IR^{d\times d}$ is an identity matrix.
The total energy can be expressed as $E = \rho e + \frac{1}{2}\rho\norm{\vec{u}}{}^2$, where $e$ denotes the internal energy.
For simplicity, we consider the ideal gas equation of state $p = (\gamma-1) \rho e$, with parameter $\gamma > 0$ where $\gamma=1.4$ for air.
With the Fourier's heat conduction law, the heat flux $\vec{q}$ is defined by $\vec{q} = -\lambda\grad{e}$, where parameter $\lambda = \frac{\gamma}{\Pr}>0$ and $\Pr$ denotes the Prandtl number.
\par
Physically meaningful solutions $\vec{U}$ should have positive density and positive internal energy. 
Define the set of admissible states as:
\begin{align*}
G = \{\vec{U}=\transpose{[\rho, \vec{m}, E]}\!:~\rho>0,~ \rho e(\vec{U})=E-\frac{\|\vec m\|^2}{2\rho} > 0\}.
\end{align*}
The set $G$ is convex and the $\rho e$ is a concave function with respect to $\vec{U}$, see \cite{zhang2010positivity}.  
With an initial condition $\vec{U}_0 = \transpose{[\rho_0,\vec{m}_0,E_0]} \in G$, it is a wide open question whether the solution of compressible NS equations \eqref{eq:CNS:model} should have positive density and internal energy for a given positive initial data, though it is partially justified for special systems, e.g., see \cite{LAI2022555, liu2021brinkman} and the references therein.
On the other hand, empirically we would expect a reasonable numerical solution to this initial value problem should belong to the set $G$ for any time $t>0$. 
\par
In general, classical numerical methods for a convection-diffusion system like \eqref{eq:CNS:model} are not positivity-preserving without any limiters.  In practice, one often observes blow-ups once negative density or negative pressure (corresponding to negative internal energy) is generated during numerical simulations.  
The linearized compressible Euler equations with negative density or negative internal energy will no longer be hyperbolic thus its initial value problem becomes ill-posed \cite{zhang2010positivity}.
When negative values emerge, the simple ad-hoc approach of truncating negative values to zero destroys conservation, which is equivalent to adding mass or internal energy into a conservative system, thus the computation will eventually still blow up.
Therefore for the sake of robustness, it is desired to construct a numerical scheme which is both conservative and  positivity-preserving.

\subsection{Existing positivity-preserving schemes for compressible Navier--Stokes equations}\label{sec-approaches}

\par
In the literature there are many different methods to construct positivity-preserving schemes for compressible Euler equations. However, it is much more difficult to construct a conservative and positive scheme for the compressible NS equations in multiple dimensions due to the mixed second order derivatives in the diffusion operator. In the past decade, significant progress of practical conservative and positive schemes has been made for the fully nonlinear compressible NS equations \eqref{eq:CNS:model}. Notable efforts include at least the following three different kinds of schemes.
\par
The first approach proposed by Grapas et al. in \cite{grapsas2016unconditionally} is to solve the internal energy equation directly 
instead of solving the total energy equation \eqref{eq:CNS:model_3}. 
By solving the internal energy equation, preserving positivity of internal energy becomes simpler but conservation of total energy becomes difficult.
The fully implicit pressure correction scheme on staggered grids in \cite{grapsas2016unconditionally} can be proven unconditionally stable, positivity-preserving and conservative. Nonlinear equations must be solved in the implementation. The spatial accuracy of this approach is at most second order accurate and it seems difficult to extend it to higher order spatial accuracy especially for a fully implicit scheme on a staggered grid. 
\par
The second approach is a fully explicit scheme proposed by the second author  in \cite{zhang2017positivity}. By solving the conservative system \eqref{eq:CNS:model}, conservation is straightforward to achieve but positivity of internal energy is difficult to enforce.
With a simple nonlinear diffusion numerical flux, it was proven in \cite{zhang2017positivity} that 
arbitrarily high order Runge--Kutta discontinuous Galerkin (DG) schemes solving \eqref{eq:CNS:model} can be rendered positivity-preserving without losing conservation and accuracy by a simple limiter,
which can be regarded as an easy extension of the Zhang--Shu method  for conservation laws in \cite{zhang2010maximum, zhang2010positivity, zhang2012maximum} to the compressible NS equations. 
The advantages of such a fully explicit approach include easy extensions to  general shear stress models and heat fluxes, and possible extensions to other type of schemes such as high order accurate finite volume schemes \cite {fan2021positivity} and the high order accurate finite difference WENO (weighted essentially nonoscillatory) scheme \cite{fan2022positivity}. 
However, the major drawback of any fully explicit scheme for the convection diffusion system \eqref{eq:CNS:model} in \cite{zhang2017positivity, fan2021positivity, fan2022positivity} is a time step constraint like $\Delta t = \mathcal{O}(\Rey\,\Delta x^2)$, which is suitable and practical only for high Reynolds number problems.  
\par
The third approach proposed by Guermond et al. in \cite{guermond2021second} introduces a semi-implicit continuous finite element scheme with positivity-preserving property under standard hyperbolic CFL condition like $\mathcal{O}(\Delta x)$. 
By applying the Strang splitting to the compressible NS model \cite{demkowicz1990new}, the equations \eqref{eq:CNS:model} are splitted into a hyperbolic subproblem $(\mathrm{H})$ and a parabolic subproblem $(\mathrm{P})$, which represent two asymptotic regimes, namely the vanishing viscosity limit, i.e., the compressible Euler equations, and the dominant of diffusive terms.
The definition of these subproblems is as follows: 
\begin{align}
(\mathrm{H})~\!
\begin{cases}
\partial_t{\rho} + \div{(\rho\vec{u})} = 0\\
\partial_t(\rho\vec{\vec{u}}) + \div{(\rho\vec{u}\otimes\vec{u} + p\vecc{I})} = \vec{0}\\
\partial_t{E} + \div{((E+p)\vec{u})} = 0
\end{cases},
&&
(\mathrm{P})~\!
\begin{cases}
\partial_t{\rho} = 0\\
\partial_t(\rho\vec{u}) - \frac{1}{\Rey}\div{\vec{\tau}(\vec{u})} = \vec{0}\\
\partial_t{E} + \frac{1}{\Rey}\div{(\vec{q} - \vec{\tau}(\vec{u})\vec{u})} = 0
\end{cases}.
\label{strang-splitting}
\end{align}
The first equation in $(\mathrm{P})$ implies variable $\rho$ in parabolic subproblem is time independent. Multiply the second equation in $(\mathrm{P})$ by $\vec{u}$, use the heat flux $\vec{q}=-\lambda\grad{e}$ and the identity 
$\div{(\vec{\tau}(\vec{u})\vec{u})} = (\div{\vec{\tau}(\vec{u})})\cdot\vec{u} + \vec{\tau}(\vec{u}):\grad{\vec{u}}$, 
we obtain the equivalent non-conservative form of equations for $(\mathrm{P})$: 
\begin{subnumcases}{(\mathrm{P})~\label{eqn-P-nonconservative}}
\partial_t{\rho} = 0, \label{eqn-internal-rho}\\
\rho\partial_t\vec{\vec{u}} - {\textstyle\frac{1}{\Rey}}\div{\vec{\tau}(\vec{u})} = \vec{0}, \label{eqn-internal-u}\\
\rho\partial_t{e} -{\textstyle\frac{\lambda}{\Rey}}\laplace{e} = {\textstyle\frac{1}{\Rey}}\vec{\tau}(\vec{u}):\grad{\vec{u}}. \label{eqn-internal-e}
\end{subnumcases}
In \cite{guermond2021second}, a semi-implicit time discretization is used for the internal energy equation \eqref{eqn-internal-e} such that only a linear system 
needs to be solved for implementing the scheme, without affecting the conservation of momentum and total energy.
The positivity of internal energy in piecewise linear finite element method can also be easily proven due to the well-known fact that piecewise linear methods can form an M-matrix for the Laplacian operator. 

\subsection{Motivation and difficulty of high order spatial accuracy in implicit schemes}
\par Even though schemes constructed from high order polynomials are high order accurate on a uniform or quasi-uniform mesh only for smooth solutions,
they produce less artificial viscosity thus resolve small scale structures  better than first order and second order schemes even for the gas dynamics problems involving with strong shocks, see examples in \cite{zhang2017positivity, fan2022positivity}. In other words, less artificial viscosity is the main motivation of pursuing a high order scheme, e.g., DG methods with polynomial basis of degree at least two. 
\par 
To see the key challenge in constructing a positivity-preserving high order scheme for compressible NS equations, we consider the heat equation $\partial_t e = \partial_{xx}e$ with homogeneious Dirichlet boundary conditions as a simplification of equation \eqref{eqn-internal-e}. The simple second order centered difference $\partial_{xx}e \approx \frac{e_{i-1}-2e_i+
e_{i+1}}{\Delta x^2}$ is monotone with both explicit and implicit time stepping. With forward Euler time stepping, the scheme 
\begin{align*}
e_{i}^{n+1} = e_i^n+\Delta t\frac{e^n_{i-1}-2e^n_i+e^n_{i+1}}{\Delta x^2} 
= \frac{\Delta t}{\Delta x^2} e^n_{i-1}+\Big(1-2\frac{\Delta t}{\Delta x^2} \Big)e^n_i+\frac{\Delta t}{\Delta x^2} e^n_{i+1} 
\end{align*}
is monotone in the sense that $e_{i}^{n+1}$ is a convex combination of $e_{i}^{n}$ and $e_{i\pm 1}^{n}$ if $\frac{\Delta t}{\Delta x^2} \leq \frac12$. Such monotonicity is in general not true for high order schemes, but some explicit high order schemes in \cite{zhang2012maximum-liu, yan2014maximum,chen2016third,   srinivasan2018positivity, sun2018discontinuous, li2018high} were shown to have weak monotonicity for the parabolic equations, which means that the cell averages can still be a monotone function. In principle, all these explicit schemes can be applied to \eqref{eqn-internal-e} for constructing a positivity-preserving scheme for \eqref{eq:CNS:model} but under a small time step constraint $\Delta t = \mathcal{O}(\Rey\,\Delta x^2)$. 
\par
With backward Euler time stepping, the scheme 
\[e_{i}^{n+1}=e_i^n+\Delta t\frac{e^{n+1}_{i-1}-2e^{n+1}_i+e^{n+1}_{i+1}}{\Delta x^2} \]
gives a linear system $\vecc{A} \mathbf e^{n+1}=\mathbf e^{n}$, where $\vecc{A}$ is a tridiagonal matrix with $\lambda=\frac{\Delta t}{\Delta x^2}$,
\begin{align*}
\vecc{A} = 
\begin{pmatrix}
1+2\lambda & -\lambda & & & \\
-\lambda & 1+2\lambda & -\lambda & & \\
 & \ddots & \ddots & \ddots & \\
 & & -\lambda & 1+2\lambda & -\lambda \\
 & & & -\lambda & 1+2\lambda
\end{pmatrix}.
\end{align*}
This implicit scheme is monotone because $\vecc{A}^{-1}$ has nonnegative entries thus one can also show $e^{n+1}_i$ is a convex combination of $e^n_j$ for all $j$ without any time step constraint. The matrix $\vecc{A}$ is diagonally dominant with non-positive off diagonal entries, so $\vecc{A}$ is an M-matrix \cite{plemmons1977m} thus $\vecc{A}^{-1}\geq 0$.  It is well-known that the monotonicity in implicit schemes holds in piecewise linear finite element method, e.g., \cite{guermond2021second}. In general, the monotonicity is not true in implicit high order schemes, e.g., the continuous finite element method with quadratic polynomials cannot be monotone on unstructured meshes \cite{hohn1981some}.
However, it is possible to show that  continuous finite element method with quadratic and cubic polynomial basis can still be monotone on a uniform rectangular mesh under practical time step and mesh constraints
\cite{li2020monotonicity, cross2020monotonicity}. 

\subsection{The main results}

\par In this paper, we are interested in constructing a conservative and positivity-preserving scheme which is high order accurate for spatial variables, without a restrictive time step constraint such as $\Delta t= \mathcal{O}(\Rey\,\Delta x^2)$. 
For problems involved with low density and low pressure, loss of positivity is the main source of instabilities of high order schemes.
In order to avoid small time steps like $\Delta{t} = \mathcal{O}(\Rey\,\Delta x^2)$, we follow the third approach in Section~\ref{sec-approaches} by solving the non-conservative form of diffusion equations \eqref{eqn-P-nonconservative}.
\par
We will mainly consider the high order DG methods, which have a lot of advantages and have been successful in many scientific and industrial applications.
In particular, high order DG methods have been quite popular for the compressible NS equations since the pioneering work in \cite{bassi1997high}. For the sake of easy extensions to arbitrarily high order polynomial basis, we use the positivity-preserving Runge--Kutta DG method for the compressible Euler equations \cite{zhang2010positivity, zhang2012maximum,  zhang2017positivity} for solving the hyperbolic subproblem (H) in \eqref{strang-splitting}. 
\par
For shear stress tensor terms $ \div{\vec{\tau}(\vec{u})}$ and $\vec{\tau}(\vec{u}):\grad{\vec{u}}$ in the parabolic subproblem (P) in \eqref{eqn-P-nonconservative}, we will also use a DG method. In the literature, many different types of DG methods have been developed for solving diffusion equations, including local DG \cite{cockburn1998local,castillo2000priori}, compact DG \cite{peraire2008compact,uranga2009implicit}, direct DG \cite{liu2010direct,zhang2012fourier,liu2015optimal}, hybridizable DG \cite{cockburn2009hybridizable, peraire2010hybridizable,nguyen2011implicit}, interior penalty DG (IPDG) \cite{Rivierebook,riviere1999improved,riviere2001priori,masri2022discontinuous}, weak Galerkin methods \cite{wang2013weak,wang2016weak}, and many others \cite{arnold2002unified,shu2014discontinuous}.  
In particular, we will use the IPDG method since the global conservation of momentum and total energy can be easily achieved via a proper choice of IPDG discretizations for approximating $ \div{\vec{\tau}(\vec{u})}$ 
and $\vec{\tau}(\vec{u}):\grad{\vec{u}}$. 
\par
In order to achieve positivity of internal energy for solving equation \eqref{eqn-internal-e}, we can utilize either IPDG with $\IQ^1$ element or spectral element method with $\IQ^2$ or $\IQ^3$ element on uniform rectangular meshes for the Laplace operator $-\laplace{e}$. The monotonicity of spectral element method with $\IQ^2$ and $\IQ^3$ element for Laplacian has recently been proven in \cite{li2020monotonicity, cross2020monotonicity}. 
\par
To summarize, our numerical scheme for solving \eqref{eq:CNS:model} consists of the following main ingredients:
\begin{enumerate}
\item With Strang splitting, the compressible Euler equations, i.e., the hyperbolic subproblem in \eqref{strang-splitting} and parabolic subproblem \eqref{eqn-P-nonconservative} are solved separately. The compressible Euler equations are solved by the positivity-preserving Runge--Kutta DG method with $\IQ^k$ element on rectangular meshes \cite{zhang2010positivity}.
\item The time stepping for the parabolic subproblem consists of Crank--Nicolson method to \eqref{eqn-internal-u} and a first order semi-implicit time discretization to \eqref{eqn-internal-e}. When a proper IPDG method is used for $\div{\vec{\tau}(\vec{u})}$ and $\vec{\tau}(\vec{u}):\grad{\vec{u}}$, global conservation of momentum and total energy is ensured. 
\item The diffusion term $-\laplace{e}$ is treated implicitly. We will prove positivity of IPDG method with $\IQ^1$ element. For positivity of higher order elements, we use the spectral element method with $\IQ^2$ and $\IQ^3$ element (i.e., continuous finite element method with Gauss--Labotto quadrature), for which monotonicity has been proven in \cite{li2020monotonicity, cross2020monotonicity}.
We emphasize that no limiters are used at all in the fully discretized scheme for solving the parabolic subproblem.
\end{enumerate}
\par
So the overall scheme is at most first order accurate in time for the system \eqref{eq:CNS:model} but fourth order accurate in space when $\IQ^3$ element is used. At first glance, the high order spatial accuracy may not look necessary since the order of time accuracy is low. However, empirically the spatial resolution is more important than the temporal for many fluid dynamics problems. In particular, computational evidence often suggests that a spatially higher order accurate scheme can produce better solutions even if the temporal order of accuracy is low. For instance, see Figure~\ref{fig:shock_reflection_diffraction} for results of our schemes solving a Mach $10$ shock reflection-diffraction problem, which involves strong shock, very low density and pressure, as well as Kelvin--Helmholtz instability. In Figure~\ref{fig:shock_reflection_diffraction}, the  $\IQ^3$ scheme with less degrees of freedom  can better capture the instability roll-ups than the $\IQ^1$ scheme, even though both schemes are first order accurate in time for the internal energy equation \eqref{eqn-internal-e}. See also the numerical examples for the superiority of $\IQ^2$ element over $\IQ^1$ element for scalar convection-diffusion problems in \cite{hu2021positivity, shen2021discrete, liu-2022-monotone}. 

\begin{figure}[ht!]
\begin{center}
\begin{tabularx}{0.8\linewidth}{@{~}C@{~}C@{~}c@{~}}
\includegraphics[width=0.33\textwidth]{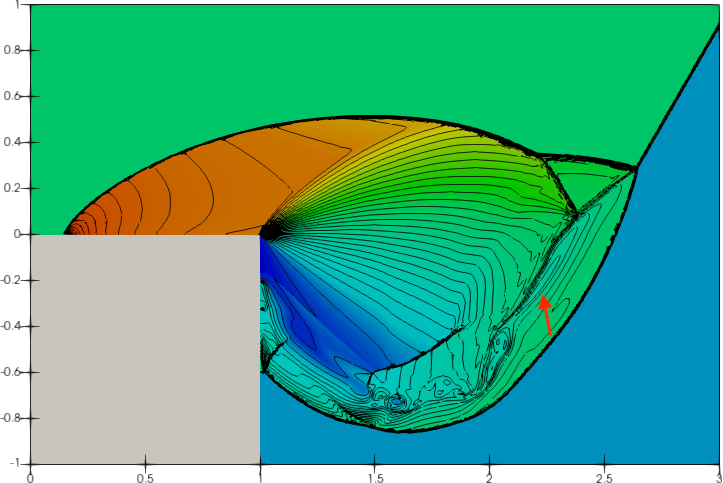} &
\includegraphics[width=0.33\textwidth]{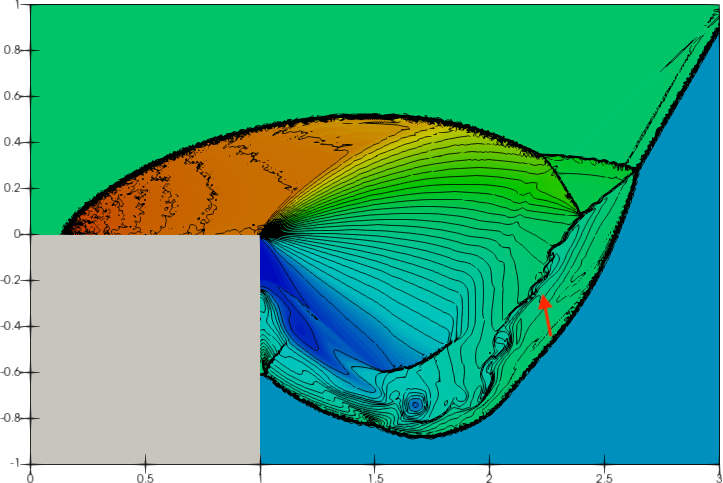} &
\includegraphics[width=0.0435\textwidth]{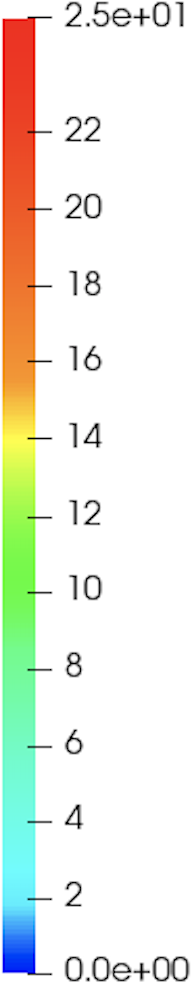}
\end{tabularx}
\caption{Mach $10$ shock reflection and diffraction with Reynolds number $1000$. Plot of density: 50 equally space contour lines from $0$ to $25$. Left snapshot from $\IQ^1$ scheme in this paper on a uniform mesh with mesh resolution $1/480$. Right snapshot from $\IQ^3$ scheme in this paper on a uniform mesh with mesh resolution $1/120$.}
\label{fig:shock_reflection_diffraction}
\end{center}
\end{figure}

\subsection{Contributions and organization of this paper}
\par 
To the best of our knowledge, this is the first time that an implicit conservative positivity-preserving scheme with high order elements like $\IQ^2$ and $\IQ^3$ elements is constructed for the compressible NS equations. Morever, numerical tests suggest that the $\IQ^3$ scheme is indeed robust with much better resolutions. 
\par
It is in general nontrivial to achieve global conservation when solving equations of the non-conservative form \eqref{eqn-P-nonconservative}. 
Even though we only consider rectangular meshes in this paper, the global conservation of IPDG methods for the parabolic subproblem \eqref{eqn-P-nonconservative}  can be easily extended to unstructured meshes. There are many variants of IPDG methods, including the
symmetric version (SIPG), the nonsymmetric version (NIPG), and the incomplete version (IIPG). In particular, we prove that the global conservation can be achieved if the shear stress tensor terms are discretized by the NIPG method.  
\par
We also prove that the second order accurate IIPG method with $\IQ^1$ element for the Laplacian term $-\laplace{e}$ forms an M-matrix. 
Even though it is well known that it is possible to achieve an M-matrix structure when using piecewise linear finite element method,
to the best of our knowledge this is the first time that such an M-matrix structure is proven among the family of IPDG methods beyond one dimension.   
\par
The rest of this paper is organized as follows. In Section~\ref{sec:numerical_scheme}, we introduce the fully discrete numerical scheme and discuss the conservation property. In Section~\ref{sec:positivity}, we discuss the positivity-preserving property. In particular we prove that the IPDG method with $\IQ^1$ element forms an M-matrix thus is monotone in \ref{sec:positivity:parabolic}.
Numerical tests are shown in Section~\ref{sec:experiments}. Concluding remarks are given in Section \ref{sec-remark}.

%% file: Content/numerical.tex
%%%%%%%%%%%%%%%%%%%%%%%%%%%%%%%%%%%%%%%%%%%%%%%%%%%%%%%%%%%%%%%%%%%%%%%%%%%%%%%%%%%%%%%%%%%%%%%%%%%%%%%%%%%%%%%%%%%%%%%
\section{The full numerical scheme}\label{sec:numerical_scheme}
%%%%%%%%%%%%%%%%%%%%%%%%%%%%%%%%%%%%%%%%%%%%%%%%%%%%%%%%%%%%%%%%%%%%%%%%%%%%%%%%%%%%%%%%%%%%%%%%%%%%%%%%%%%%%%%%%%%%%%%
In this section, we describe the fully discretized numerical scheme for solving the compressible NS equations \eqref{eq:CNS:model} that utilizes DG discretization in space within the Strang splitting framework. Then we show that our method preserves the global conservation. 

\subsection{Time discretization}

\par Given the conserved variables $\vec{U}^n$ at time step $t_n$ ($n\geq0$), the Strang splitting for evolving to time step $t_{n+1}=t_n+\Delta t$ for the system \eqref{eq:CNS:model} is to solve $(\mathrm{P})$ and $(\mathrm{H})$ in \eqref{strang-splitting} separately:
\begin{align}
\vec{U}^n
\xrightarrow[\text{step size}~\frac{\Delta t}{2}]{\text{solve}~(\mathrm{H})} \vec{U}^\mathrm{H}
\xrightarrow[\text{step size}~\Delta t]{\text{solve}~(\mathrm{P})} \vec{U}^\mathrm{P}
\xrightarrow[\text{step size}~\frac{\Delta t}{2}]{\text{solve}~(\mathrm{H})} \vec{U}^{n+1}.
\label{algorithm-splitting}
\end{align}
Define the advection flux as 
\begin{align*}
\vec{F}^\mathrm{a} = \transpose{[\rho\vec{u}, \rho\vec{u}\otimes\vec{u} + p\vecc{I}, (E+p)\vec{u}]}.
\end{align*}
For any $n\geq 0$, the time discretization methods in one time step of Strang splitting consists of the following steps: 
\begin{itemize}[leftmargin=0.5cm]
\item[] Step~1. Given $\vec{U}^n=\transpose{[\rho^n,\vec{m}^n,E^n]}$, 
we use the third order strong stability preserving (SSP) Runge--Kutta method \cite{shu1988total}  to obtain
$\vec{U}^{\mathrm{H}}=\transpose{[\rho^{\mathrm{H}}, \vec{m}^{\mathrm{H}}, E^{\mathrm{H}}]}$ in the first step in Strang splitting \eqref{algorithm-splitting},
\begin{subequations}\label{rk3}
\begin{align}
\vec{U}^{\mathrm{(1)}} &= \vec{U}^n - \frac{\Delta t}{2}\div{\vec{F}^\mathrm{a}(\vec{U}^n)},\\
\vec{U}^{\mathrm{(2)}} &= \frac{3}{4} \vec{U}^{n}+\frac{1}{4}\Big[\vec{U}^{(1)} - \frac{\Delta t}{2}\div{\vec{F}^\mathrm{a}(\vec{U}^{(1)})}\Big],\\
\vec{U}^{\mathrm{H}} &= \frac{1}{3} \vec{U}^{n}+\frac{2}{3}\Big[\vec{U}^{(2)} - \frac{\Delta t}{2}\div{\vec{F}^\mathrm{a}(\vec{U}^{(2)})}\Big].
\end{align}
\end{subequations}
\item[] Step~2. Given $\vec{U}^{\mathrm{H}} = \transpose{[\rho^{\mathrm{H}}, \vec{m}^{\mathrm{H}}, E^{\mathrm{H}}]}$, compute $(\vec{u}^\mathrm{H}, e^\mathrm{H})$ by solving
\begin{align*}
\vec{m}^\mathrm{H} = \rho^\mathrm{H} \vec{u}^\mathrm{H} \quad\text{and}\quad
E^\mathrm{H} = \rho^\mathrm{H}e^\mathrm{H} + \frac{1}{2}\rho^\mathrm{H}\norm{\vec{u}^{\mathrm{H}}}{}^2.
\end{align*}
\item[] Step~3.
Notice that equation \eqref{eqn-internal-rho} implies that $\rho^{\mathrm{P}}=\rho^{\mathrm{H}}$ in the second step in Strang splitting 
\eqref{algorithm-splitting}.
Apply the second order Crank--Nicolson method to \eqref{eqn-internal-u} and a first order semi-implicit time discretization to \eqref{eqn-internal-e},
\begin{align*}
&\vec{u}^\ast=\frac{1}{2} \vec{u}^\mathrm{P}+\frac12 \vec{u}^\mathrm{H},\\
&\rho^{\mathrm{P}}\frac{\vec{u}^\mathrm{P}-\vec{u}^\mathrm{H}}{\Delta t} - \frac{1}{\Rey}\div{\vec{\tau}(\vec{u}^\ast)} = \vec{0}, \\
&\rho^{\mathrm{P}} \frac{e^\mathrm{P}-e^\mathrm{H}}{\Delta t} -\frac{1}{\Rey}\vec{\tau}(\vec{u}^\ast):\grad{\vec{u}^\ast}= \frac{\lambda}{\Rey}\laplace{e}^\mathrm{P}, 
\end{align*}
which can be implemented as first solving two decoupled linear systems for $\vec{u}^{\ast}$ and  $e^{\mathrm{P}}$ \begin{subequations}
\begin{align}
\rho^{\mathrm{P}}\vec{u}^\ast - \frac{\Delta t}{2\Rey}\div{\vec{\tau}(\vec{u}^\ast)} &= \rho^{\mathrm{H}}\vec{u}^{\mathrm{H}},\label{eq:time:step3_1}\\
\rho^{\mathrm{P}} e^\mathrm{P} - \frac{{\Delta t}\,\lambda}{\Rey}\laplace{e}^\mathrm{P} &= \rho^{\mathrm{H}}e^{\mathrm{H}} + \frac{\Delta t}{\Rey}\vec{\tau}(\vec{u}^\ast):\grad{\vec{u}^\ast},\label{eq:time:step3_2}
\end{align}
\end{subequations}
then setting $\vec{u}^\mathrm{P} = 2\vec{u}^\ast - \vec{u}^\mathrm{H}$.
\item[] Step~4. Given $(\rho^{\mathrm{P}}, \vec{u}^{\mathrm{P}}, e^{\mathrm{P}})$, compute $(\vec{m}^\mathrm{P}, E^\mathrm{P})$ by
\begin{align*}
\vec{m}^\mathrm{P} = \rho^\mathrm{P}\vec{\vec{u}}^\mathrm{P}
\quad\text{and}\quad
E^\mathrm{P} = \rho^\mathrm{P}e^\mathrm{P} + \frac{1}{2}\rho^\mathrm{P}\norm{\vec{u}^{\mathrm{P}}}{}^2.
\end{align*}
\item[] Step~5. Given $\vec{U}^{\mathrm{P}}=\transpose{[\rho^{\mathrm{P}}, \vec{\vec{m}}^\mathrm{P}, E^{\mathrm{P}}]}$, to obtain $\vec{U}^{n+1}=\transpose{[\rho^{n+1},\vec{m}^{n+1},E^{n+1}]}$ in \eqref{algorithm-splitting}, solve $(\mathrm{H})$ for another $\frac12\Delta t$ by the third order SSP Runge--Kutta.
\end{itemize}

\subsection{Space discretization}
Let $\setE_h$ be a polygonal mesh of the computational domain $\Omega$, where each element $K$ is a square in two dimension and degenerates to an interval in one dimension. 
Let $h$ denote the mesh size, namely the diagonal length of a square element in two dimension and the interval length in one dimension.
\par
Let $\IQ^k(K)$ be the space of tensor product of one-dimensional polynomials of degree $k$ on an element $K$. 
Define the following discontinuous polynomial spaces: 
\begin{align*}
M_h^k &= \big\{\chi_h\in L^2(\Omega):~\forall K \in \setE_h,\, \on{\chi_h}{K} \in \IQ^k(K) \big\}, \\
\mathbf{X}_h^k &= \big\{\vec{\theta}_h \in L^2(\Omega)^d:~\forall K \in \setE_h,\, \on{\vec{\theta}_h}{K} \in \IQ^k(K)^d \big\}.
\end{align*}
We first briefly review the Runge--Kutta DG scheme for Euler equations, then we describe the IPDG scheme for the parabolic subproblem.
\paragraph{\bf Hyperbolic subproblem}
For solving $(\mathrm{H})$, we utilize the same scheme as described in \cite{zhang2010positivity}, in which a simple limiter can preserve positivity without destroying conservation and accuracy in high order DG methods. The positivity-preserving property will be reviewed in Section~\ref{sec:positivity}. Here we briefly review the scheme.
\par The semi-discrete DG scheme on an element $K$  for the compressible Euler equations $\partial_t\vec{U} + \div{\vec{F}^{\mathrm{a}}(\vec{U})} = \vec{0}$ is defined by
finding the piecewise polynomial solution $\vec{U}_h$ satisfying
\begin{align}\label{eq:hyper_space_dis}
\frac{\dd}{\dd t}\int_K \vec{U}_h  \varPsi_h 
=  \int_K \vec{F}^{\mathrm{a}}(\vec{U}_h)\cdot \grad{\varPsi_h} - \int_{\partial{K}} \widehat{\vec{F}^{\mathrm{a}}\cdot\normal_K}(\vec{U}_h^{-},\vec{U}_h^{+})\varPsi_h,
\end{align}
for any piecewise polynomial test function $\varPsi_h$ on any element $K$, 
where $\normal_{K}$ is the unit outward normal of $K$ and the $\widehat{\vec{F}^\mathrm{a}\cdot\normal_K}$ is a Lax--Friedrichs flux for $\vec{F}^{\mathrm{a}}$. On a face or an edge $e\subset \partial K$, the local Lax--Friedrichs flux is defined by
\begin{align*}
\widehat{\vec{F}^\mathrm{a}\cdot\normal_K}(\vec{U}_h^-,\vec{U}_h^+) 
= \frac{\vec{F}^{\mathrm{a}}(\vec{U}_h^-) + \vec{F}^{\mathrm{a}}(\vec{U}_h^+)}{2}\cdot\normal_K 
- \frac{\alpha_e}{2}(\vec{U}_h^+ - \vec{U}_h^-),
\end{align*}
where  the $\vec{U}_h^-$ (resp. $\vec{U}_h^+$) denotes the trace of a function $\vec{U}_h$ on the face $\partial{K}$ coming from the interior (resp. exterior) of $K$. 
Here, $\alpha_e$ denotes the maximum wave speed with maximum taken over all $\vec{U}_h^-$ and $\vec{U}_h^+$ along the face or edge $e$, i.e., the largest magnitude of the eigenvalues of the Jacobian matrix $\frac{\partial \vec{F}^\mathrm{a}}{\partial \vec{U}}$, which equals to the wave speed $\abs{\vec{u}\cdot\normal_K}+\sqrt{\gamma\frac{p}{\rho}}$ for ideal gas equation of state. 

By convention, we replace $\vec{U}_h^+$ by an appropriate boundary function which realizes the boundary conditions when $\partial{K}\cap\partial{\Omega}\neq\emptyset$. For instance, if purely inflow condition $\vec{U} = \vec{U}_\mathrm{D}$ is imposed on $\partial{K}$, then $\vec{U}_h^+$ is replaced by $\vec{U}_\mathrm{D}$; if purely outflow condition is imposed on $\partial{K}$, then set $\vec{U}_h^+=\vec{U}_h^-$; and if reflective boundary condition for fluid--solid interfaces is imposed on $\partial{K}$, then set $\vec{U}_h^+ = \transpose{[\rho_h^-,\vec{m}_h^{-}-2(\vec{m}_h^-\cdot\normal_K)\normal_K,E_h^-]}$.

\paragraph{\bf Parabolic subproblem}
We use the IPDG method for discretizing $\mathrm{(P)}$.
For convenience of introducing discrete forms in parabolic subproblem, we partition the boundary of the domain $\Omega$ into the union of two disjoint sets, namely $\partial{\Omega} = \partial{\Omega}_\mathrm{D} \cup \partial{\Omega}_\mathrm{N}$, where the Dirichlet boundary conditions ($\vec{u}=\vec{u}_\mathrm{D}$ and $e=e_\mathrm{D}$) are applied on $\partial{\Omega}_\mathrm{D}$ and the Neumann-type boundary conditions ($\vec{\tau}(\vec{u})\cdot\normal=\vec{0}$ and $\grad{e}\cdot\normal=0$) are applied on $\partial{\Omega}_\mathrm{N}$. Here, $\normal$ denotes the unit outer normal of domain $\Omega$.
\par
Let $\Gammah$ denote the set of interior faces. For each interior face $e \in \Gammah$ shared by elements $K_{i^-}$ and $K_{i^+}$, with $i^- < i^+$, we define a unit normal vector $\normal_e$ that points from $K_{i^-}$ into $K_{i^+}$. For a boundary face $e$, i.\,e., $e = \partial K_{i^-} \cap \partial\Omega$, the normal $\normal_e$ is taken to be the unit outward vector to $\partial\Omega$. We define the broken Sobolev spaces, for any $r\geq1$,
\begin{align*}
H^r(\setE_h) = \{\omega\in L^2(\Omega):~ \forall K\in\setE_h,\, \on{\omega}{K}\in H^r(K)\}.
\end{align*}
The average and jump operators of any scalar quantity $\omega\in H^1(\setE_h)$ are defined for each interior face $e\in\Gammah$ by
\begin{align*}
\on{\avg{\omega}}{e} = \frac{1}{2}\on{\omega}{K_{i^-}} + \frac{1}{2}\on{\omega}{K_{i^+}}, \quad
\on{\jump{\omega}}{e} = \on{\omega}{K_{i^-}} - \on{\omega}{K_{i^+}}, \quad 
e = \partial K_{i^-} \cap \partial K_{i^+}.
\end{align*}
If a face $e$ belongs to the boundary $\partial\Omega$, the jump and average of $\omega$ coincide with its trace on face $e$. The related definitions of any vector quantity are similar. For more details see \cite{Rivierebook}.
\par
The main focus here is the conservation of momentum and total energy, since we solve the non-conservative form of the parabolic subproblem \eqref{eqn-P-nonconservative}. The fluxes across the element interfaces should be designed such that no extra discrete momentum or discrete total energy is created or eliminated over the whole domain. 
We utilize the NIPG method to discretize \eqref{eq:time:step3_1}. 
The bilinear forms $a_\strain: H^2(\setE_h)^d\times H^2(\setE_h)^d\rightarrow \IR$ and $a_\lambda: H^2(\setE_h)^d\times H^2(\setE_h)^d\rightarrow \IR$ associated with terms $-2\div{\strain{(\vec{u})}}$ and $\div{((\div{\vec{u})\vecc{I}})}$ are defined as follows:
\begin{align*}
a_\strain(\vec{u}, \vec{\theta}) &= 
2\sum_{K\in\setE_h} \int_K \strain{(\vec{u})}:\strain{(\vec{\theta})} 
- 2\sum_{e\in\Gammah\cup\partial\Omega_\mathrm{D}} \int_e\avg{\strain{(\vec{u})}\,\normal_e}\cdot\jump{\vec{\theta}}\nonumber\\
&+2\sum_{e\in\Gammah\cup\partial\Omega_\mathrm{D}} \int_e\avg{\strain{(\vec{\theta})}\,\normal_e}\cdot\jump{\vec{u}}
+ \frac{\sigma}{h}\sum_{e\in\Gammah\cup\partial\Omega_\mathrm{D}} \int_e \jump{\vec{u}}\cdot\jump{\vec{\theta}},\\
a_\lambda(\vec{u}, \vec{\theta}) &= -\sum_{K\in\setE_h}\int_K(\div{\vec{u}})(\div{\vec{\theta}})
+\!\! \sum_{e\in\Gammah\cup\partial\Omega_\mathrm{D}}\!\int_e\avg{\div{\vec{u}}}\jump{\vec{\theta}\cdot\normal_e} 
-\!\! \sum_{e\in\Gammah\cup\partial\Omega_\mathrm{D}}\!\int_e\avg{\div{\vec{\theta}}}\jump{\vec{u}\cdot\normal_e}.
\end{align*}
And the linear form $b_{\vec{\tau}}: H^2(\setE_h)^d\rightarrow \IR$ associated with the term $-\div{\vec{\tau}(\vec{u})}$ for the Dirichlet boundary $\partial \Omega_\mathrm{D}$ in \eqref{eq:time:step3_1} is defined by
\begin{align*}
b_{\vec{\tau}}(\vec{\theta}) 
= 2\sum_{e\in\partial\Omega_\mathrm{D}} \int_e (\strain{(\vec{\theta})}\,\normal)\cdot\vec{u}_\mathrm{D}
+ \frac{\sigma}{h}\sum_{e\in\partial\Omega_\mathrm{D}} \int_e \vec{u}_\mathrm{D}\cdot\vec{\theta}
- \frac{2}{3}\sum_{e\in\partial\Omega_\mathrm{D}}\int_e\div{\vec{\theta}}\,(\vec{u}_\mathrm{D}\cdot\normal).
\end{align*}
\par
In order to achieve monotonicity for at least $\IQ^1$ element, we employ the IIPG method to discretize the term $-\laplace{e}$ in \eqref{eq:time:step3_2}. In \ref{sec:positivity:parabolic}, we will prove that the $\IQ^1$ IIPG discretization enjoys an M-matrix structure unconditionally.
For the IIPG discretization, we define the bilinear form $a_{\mathcal{D}}: H^2(\setE_h)\times H^2(\setE_h)\rightarrow \IR$ and the linear form $b_{\mathcal{D}}: H^2(\setE_h)\rightarrow \IR$ for term $-\laplace{e}$ as follows:
\begin{align*}
a_{\mathcal{D}}(e,\chi) &=
\sum_{K\in\setE_h} \int_K \grad e \cdot \grad \chi
-\sum_{e\in\Gammah\cup\partial{\Omega_{\mathrm{D}}}} \int_e \avg{\grad e \cdot \normal_e} \jump{\chi}
+ \frac{\tilde{\sigma}}{h} \sum_{e\in\Gammah\cup\partial{\Omega_{\mathrm{D}}}}\int_e \jump{e}\jump{\chi},\\
b_{\mathcal{D}}(\chi) &=
\frac{\tilde{\sigma}}{h}\sum_{e\in\partial{\Omega_{\mathrm{D}}}}\int_e e_\mathrm{D}\chi.
\end{align*}
\par
For the sake of global conservation of total energy, to discrete term $\vec{\tau}(\vec{u}):\grad{\vec{u}} = 2\strain{(\vec{u})}:\grad{\vec{u}} - \frac{2}{3}((\div{\vec{u}})\vecc{I}):\grad{\vec{u}}$ in \eqref{eq:time:step3_2}, by using the tensor identity $\strain{(\vec{u})}:\grad{\vec{u}} = \strain{(\vec{u})}:\strain{(\vec{u})}$, the DG forms $b_\strain: H^2(\setE_h)^d\times H^2(\setE_h)\rightarrow \IR$ and $b_\lambda: H^2(\setE_h)^d\times H^2(\setE_h)\rightarrow \IR$ are designed for terms $2\strain{(\vec{u})}:\grad{\vec{u}}$ and $-((\div{\vec{u})\vecc{I}}):\grad{\vec{u}}$, respectively.
\begin{align*}
b_\strain(\vec{u}, \chi) &=
2\sum_{K\in\setE_h} \int_K \strain{(\vec{u})}:\strain{(\vec{u})}\chi
+ \frac{\sigma}{h}\sum_{e\in\Gammah} \int_e \jump{\vec{u}}\cdot\jump{\vec{u}}\avg{\chi}
+ \frac{\sigma}{h}\sum_{e\in\partial\Omega_\mathrm{D}} \int_e (\vec{u}-\vec{u}_\mathrm{D})\cdot(\vec{u}-\vec{u}_\mathrm{D}) \chi,\\
b_\lambda(\vec{u}, \chi) &= -\sum_{K\in\setE_h}\int_K(\div{\vec{u}})(\div{\vec{u}})\chi.
\end{align*}
We note that the DG forms above employ penalty parameters $\sigma$ and $\tilde{\sigma}$. For any $\sigma\geq0$, the bilinear form of the NIPG method is coercive. In particular, NIPG0 refers to the choice $\sigma=0$, namely the penalty term is removed. The NIPG0 method is convergent for polynomial degrees greater than or equal to two in two dimension \cite{Rivierebook}.   For IIPG method, the penalty $\tilde{\sigma}$ needs to be large enough for coercivity. 
The penalty parameters used in our numerical tests will be given in Section~\ref{sec:experiments}.
Next we summarize our fully discrete scheme.

\paragraph{\bf The fully discrete scheme}
Let $(\cdot,\cdot)$ and $\langle\cdot,\cdot\rangle$ denote the $L^2$ inner products associated with the quadrature rules which are employed in hyperbolic and parabolic subproblems, respectively. The quadrature rules should be accurate enough for $\IQ^k$ polynomial basis. On a rectangular cell, the $(k+1)^d$-point Gauss quadrature and $(k+1)^d$-point Gauss--Lobatto quadrature  are accurate for integrating $(2k+1)^\mathrm{th}$-order polynomials and $(2k-1)^\mathrm{th}$-order polynomials, respectively.
The quadrature rule for solving $(\mathrm{H})$ can be the same as in \cite{zhang2010maximum,zhang2010positivity, zhang2017positivity}.
\par
Our fully discrete scheme for solving \eqref{eq:CNS:model} can be stated as follows:
\begin{itemize}[leftmargin=0.5cm]
\item[] Step~1. Given $\vec{U}_h^n \in M_h^k\times \mathbf{X}_h^k\times M_h^k$, compute $\vec{U}_h^\mathrm{H}\in M_h^k\times \mathbf{X}_h^k\times M_h^k$ by the DG method \eqref{eq:hyper_space_dis} with the positivity-preserving SSP Runge--Kutta \eqref{rk3} \cite{zhang2010positivity, zhang2017positivity} using step size $\frac{\Delta t}{2}$.
\item[] Step~2. Given $\vec{U}_h^\mathrm{H} \in M_h^k\times \mathbf{X}_h^k\times M_h^k$, compute $(\vec{u}_h^\mathrm{H}, e_h^\mathrm{H}) \in \mathbf{X}_h^k\times M_h^k$ by $L^2$ projection
\begin{align}\label{eq:CNS:P_full_dis_L2proj1}
\langle\vec{m}_h^\mathrm{H},\vec{\theta}_h \rangle = \langle\rho_h^\mathrm{H} \vec{u}_h^\mathrm{H},\vec{\theta}_h \rangle,~~
\forall \vec{\theta}_h \in \mathbf{X}_h^k \quad\text{and}\quad
\langle E_h^\mathrm{H},\chi_h \rangle = \langle\rho_h^\mathrm{H}e_h^\mathrm{H},\chi_h \rangle + \frac{1}{2} \langle\rho_h^\mathrm{H}\vec{u}_h^{\mathrm{H}},\vec{u}_h^{\mathrm{H}}\chi_h \rangle,~~ \forall \chi_h\in M_h^k.
\end{align}
\item[] Step~3. First given $\rho_h^{\mathrm{H}} \in M_h^k$, and set $\rho^{\mathrm{P}}_h = \rho^{\mathrm{H}}_h$. Given $(\rho^{\mathrm{H}}_h, \rho^{\mathrm{P}}_h, \vec{u}^{\mathrm{H}}_h) \in M_h^k\times M_h^k\times \mathbf{X}_h^k$, solve for $(\vec{u}^{\ast}_h, \vec{u}^{\mathrm{P}}_h)\in \mathbf{X}_h^k\times\mathbf{X}_h^k$, such that for all $\vec{\theta}_h\in \mathbf{X}_h^k$
\begin{subequations}
\begin{align}
\langle\rho_h^\mathrm{P}\vec{u}_h^\ast,\vec{\theta}_h\rangle + \frac{\Delta t}{2\Rey}a_\strain(\vec{u}^{\ast}_h,\vec{\theta}_h) + \frac{\Delta t}{3\Rey}a_\lambda(\vec{u}^{\ast}_h,\vec{\theta}_h)
&= \langle\rho_h^\mathrm{H}\vec{u}_h^\mathrm{H},\vec{\theta}_h\rangle +\frac{\Delta t}{2\Rey} b_{\vec{\tau}}(\vec{\theta}_h),\label{eq:CNS:P_full_dis2_1}\\
\vec{u}^{\mathrm{P}}_h &= 2\vec{u}^{\ast}_h - \vec{u}^{\mathrm{H}}_h.\label{eq:CNS:P_full_dis2_2}
\end{align}
Then given $(\rho^{\mathrm{H}}_h, \rho^{\mathrm{P}}_h, \vec{u}^{\ast}_h, e^{\mathrm{H}}_h) \in M_h^k\times M_h^k\times \mathbf{X}_h^k\times M_h^k$, solve for $e^{\mathrm{P}}_h \in M_h^k$, such that for all $\chi_h\in M_h^k$
\begin{align}\label{eq:CNS:P_full_dis3}
\langle\rho^\mathrm{P}_h e_h^\mathrm{P},\chi_h\rangle + \frac{\Delta t\lambda}{\Rey}a_{\mathcal{D}}(e_h^\mathrm{P},\chi_h) = \langle\rho_h^\mathrm{H} e_h^\mathrm{H},\chi_h\rangle + \frac{\Delta t}{\Rey}b_\strain(\vec{u}_h^{\ast},\chi_h) 
+ \frac{2\Delta t}{3\Rey}b_\lambda(\vec{u}_h^{\ast},\chi_h) + \frac{\Delta t\lambda}{\Rey}b_\mathcal{D}(\chi_h).
\end{align}
\end{subequations}
\item[] Step~4. Given $(\rho_h^{\mathrm{P}}, \vec{u}_h^{\mathrm{P}}, e_h^{\mathrm{P}})\in M_h^k\times \mathbf{X}_h^k\times M_h^k$, compute $(\vec{m}_h^\mathrm{P}, E_h^\mathrm{P})\in \mathbf{X}_h^k\times M_h^k$ by $L^2$ projection
\begin{align}\label{eq:CNS:P_full_dis_L2proj2}
\langle\vec{m}_h^\mathrm{P},\vec{\theta}_h\rangle = \langle\rho_h^\mathrm{P}\vec{\vec{u}}_h^\mathrm{P},\vec{\theta}_h\rangle,~~
\forall \vec{\theta}_h \in \mathbf{X}_h^k
\quad\text{and}\quad
\langle E_h^\mathrm{P},\chi_h\rangle = \langle\rho_h^\mathrm{P}e_h^\mathrm{P},\chi_h\rangle + \frac{1}{2}\langle\rho_h^\mathrm{P}\vec{u}_h^{\mathrm{P}},\vec{u}_h^{\mathrm{P}}\chi_h\rangle,~~ \forall \chi_h\in M_h^k.
\end{align}
Postprocess $\vec{U}_h^\mathrm{P}$ by the positivity-preserving limiter in \cite{zhang2010positivity}.
\item[] Step~5. Given $\vec{U}_h^\mathrm{P}\in M_h^k\times \mathbf{X}_h^k\times M_h^k$, compute $\vec{U}_h^{n+1}\in M_h^k\times \mathbf{X}_h^k\times M_h^k$ by \eqref{eq:hyper_space_dis} with step size $\frac{\Delta t}{2}$.
Postprocess $\vec{U}_h^\mathrm{n+1}$ by
the positivity-preserving limiter in \cite{zhang2010positivity}.
\end{itemize}
The initial value $\vec{U}_h^0$ is obtained via postprocessing the $L^2$ projection of $\vec{U}_0$ by the positivity-preserving limiter \cite{zhang2010positivity}. % on quadrature points for $\IQ^k$ ($k\geq2$) scheme. 
The positivity-preserving limiter will be briefly reviewed in Section~\ref{sec:positivity}.
\par
In Step~3, the two linear systems \eqref{eq:CNS:P_full_dis2_1} and \eqref{eq:CNS:P_full_dis3} are solved sequentially. The unique solvability of the linear systems is a straightforward conclusion due to the coercivity, see Chapter~2 and Chapter~5 in \cite{Rivierebook}.

\subsection{The global conservation}
Next we show that the fully discrete scheme preserves the global conservation of conserved variables. For simplicity, we discuss the conservation only for  periodic boundary conditions. It is straightforward to extend the discussion to many other types of boundary conditions, such as the ones implemented in the numerical tests in this paper.
\par
Both the explicit Runge--Kutta DG scheme for the compressible Euler equations and the positivity-preserving limiter conserve mass, momentum, and total energy \cite{zhang2010positivity,zhang2017positivity}. Thus we have 
\begin{align*}
(\rho_h^{n},1) = (\rho_h^{\mathrm{H}},1), \quad
(\vec{m}_h^{n},\vec{1}) = (\vec{m}_h^{\mathrm{H}},\vec{1}),\quad
(E_h^{n},1) = (E_h^{\mathrm{H}},1),
\end{align*}
and $(\rho_h^{n+1},1) = (\rho_h^{\mathrm{P}},1)$. 
Therefore, $(\rho_h^{n},1) = (\rho_h^{n+1},1)$ holds, since in Step~3, we set $\rho_h^{\mathrm{H}} = \rho_h^{\mathrm{P}}$.
\par
Notice that we have $(\vec{m}_h^{n},\vec{1}) = (\vec{m}_h^{\mathrm{H}},\vec{1})$ and $(\vec{m}_h^{n+1},\vec{1}) = (\vec{m}_h^{\mathrm{P}},\vec{1})$.
Assume that $(\cdot, \cdot)$ and $\langle \cdot, \cdot\rangle$ are accurate enough quadratures, we  also have $(\vec{m}_h^{\mathrm{H}},\vec{1}) = \langle\vec{m}_h^{\mathrm{H}},\vec{1}\rangle$ and $(\vec{m}_h^{\mathrm{P}},\vec{1}) = \langle\vec{m}_h^{\mathrm{P}},\vec{1}\rangle$.
Take $\vec{\theta}_h = \vec{1}$ in \eqref{eq:CNS:P_full_dis_L2proj1} and \eqref{eq:CNS:P_full_dis_L2proj2}, we get $\langle\vec{m}_h^{\mathrm{H}},\vec{1}\rangle = \langle\rho_h^\mathrm{H} \vec{u}_h^\mathrm{H},\vec{1}\rangle$ and $\langle\vec{m}_h^{\mathrm{P}},\vec{1}\rangle = \langle\rho_h^\mathrm{P} \vec{u}_h^\mathrm{P},\vec{1}\rangle$.
Thus, above identities indicate $(\vec{m}_h^{n},\vec{1}) = \langle\rho_h^\mathrm{H} \vec{u}_h^\mathrm{H},\vec{1}\rangle$ and $(\vec{m}_h^{n+1},\vec{1}) = \langle\rho_h^\mathrm{P} \vec{u}_h^\mathrm{P},\vec{1}\rangle$. 
By selecting $\vec{\theta}_h = \vec{1}$ in \eqref{eq:CNS:P_full_dis2_1}, we obtain $\langle\rho_h^\mathrm{H} \vec{u}_h^\mathrm{H},\vec{1}\rangle = \langle\rho_h^\mathrm{P} \vec{u}_h^\mathrm{P},\vec{1}\rangle$, namely the discrete momentum conservation holds.
\par
Similarly, we have $(E_h^{n},1) = \langle\rho_h^\mathrm{H} e_h^\mathrm{H}, 1\rangle + \frac{1}{2}\langle\rho_h^\mathrm{H}\vec{u}_h^{\mathrm{H}},\vec{u}_h^{\mathrm{H}}\rangle$ and $(E_h^{n+1},1) = \langle\rho_h^\mathrm{P} e_h^\mathrm{P},1\rangle + \frac{1}{2}\langle\rho_h^\mathrm{P}\vec{u}_h^{\mathrm{P}},\vec{u}_h^{\mathrm{P}}\rangle$.
Recall that $b_{\vec{\tau}}(\vec{\theta})=0$ and $b_\mathcal{D}(\chi)=0$ for periodic boundary conditions, thus
by \eqref{eq:CNS:P_full_dis2_2} and $\rho_h^\mathrm{H} = \rho_h^\mathrm{P}$, the step \eqref{eq:CNS:P_full_dis2_1} can be written as
\begin{align*}
\langle\rho_h^\mathrm{P}\vec{u}_h^\mathrm{P},\vec{\theta}_h\rangle + \frac{\Delta t}{\Rey}a_\strain(\vec{u}^{\ast}_h,\vec{\theta}_h) + \frac{2\Delta t}{3\Rey}a_\lambda(\vec{u}^{\ast}_h,\vec{\theta}_h) &= \langle\rho_h^\mathrm{H}\vec{u}_h^\mathrm{H},\vec{\theta}_h\rangle.
\end{align*}
Plugging in $\vec{\theta}_h = (\vec{u}^{\mathrm{P}}_h + \vec{u}^{\mathrm{H}}_h)/2 = \vec{u}^{\ast}_h$, we have
\begin{align}
\frac{1}{2}\langle\rho_h^\mathrm{P}\vec{u}_h^\mathrm{P},\vec{u}_h^\mathrm{P}\rangle + \frac{\Delta t}{\Rey}a_\strain(\vec{u}^{\ast}_h,\vec{u}^{\ast}_h) + \frac{2\Delta t}{3\Rey}a_\lambda(\vec{u}^{\ast}_h,\vec{u}^{\ast}_h) = \frac{1}{2}\langle\rho_h^\mathrm{H}\vec{u}_h^\mathrm{H},\vec{u}_h^\mathrm{H}\rangle.
\end{align}
With $\chi_h = 1$ in \eqref{eq:CNS:P_full_dis3}, we have
\begin{align}
\langle\rho^\mathrm{P}_h e_h^\mathrm{P},1\rangle + \frac{\Delta t\lambda}{\Rey}a_{\mathcal{D}}(e_h^\mathrm{P},1) = \langle\rho_h^\mathrm{H} e_h^\mathrm{H},1\rangle + \frac{\Delta t}{\Rey}b_\strain(\vec{u}_h^{\ast},1) + \frac{2\Delta t}{3\Rey}b_\lambda(\vec{u}_h^{\ast},1).
\end{align}
By adding two equations above, with the fact that $a_{\mathcal{D}}(e_h^\mathrm{P},1) = 0$ and the identities $a_\strain(\vec{u}_h^{\ast},\vec{u}_h^{\ast}) = b_\strain(\vec{u}_h^{\ast},1)$ and $a_\lambda(\vec{u}_h^{\ast},\vec{u}_h^{\ast}) = b_\lambda(\vec{u}_h^{\ast},1)$,
we obtain
\begin{align*}
\langle\rho^{\mathrm{H}}_h e^{\mathrm{H}}_h, 1\rangle + \frac{1}{2}\langle\rho^{\mathrm{H}}_h\vec{u}^{\mathrm{H}}_h, \vec{u}_h^{\mathrm{H}}\rangle = \langle\rho^{\mathrm{P}}_h e^{\mathrm{P}}_h, 1\rangle + \frac{1}{2}\langle\rho^{\mathrm{P}}_h\vec{u}^{\mathrm{P}}_h, \vec{u}^{\mathrm{P}}_h\rangle.
\end{align*}

\begin{theorem}\label{thm:dis_momentum_conv}
For $\IQ^k$ scheme, assume the quadrature rules in hyperbolic and parabolic subproblems are both exact for integrating polynomials of degree $k$, then the fully discrete scheme conserves density, momentum, and total energy,
$$(\rho_h^{n},1) = (\rho_h^{n+1},1),\quad (\vec{m}_h^{n},\vec{1}) = (\vec{m}_h^{n+1},\vec{1}), \quad (E_h^{n},1) = (E_h^{n+1},1).$$
\end{theorem}

%%%%%%%%%%%%%%%%%%%%%%%%%%%%%%%%%%%%%%%%%%%%%%%%%%%%%%%%%%%%%%%%%%%%%%%%%%%%%%%%%%%%%%%%%%%%%%%%%%%%%%%%%%%%%%%%%%%%%%%
\section{The positivity-preserving property}\label{sec:positivity}
%%%%%%%%%%%%%%%%%%%%%%%%%%%%%%%%%%%%%%%%%%%%%%%%%%%%%%%%%%%%%%%%%%%%%%%%%%%%%%%%%%%%%%%%%%%%%%%%%%%%%%%%%%%%%%%%%%%%%%%  
\par From Section~\ref{sec:numerical_scheme}, a schematic flowchart of our fully discrete scheme at step $n\geq0$ is as follows:
\begin{align*}
\vec{U}_h^n
\xrightarrow[\text{step size}~\frac{\Delta t}{2}]{\text{solve}~(\mathrm{H})} \vec{U}_h^\mathrm{H}
\xrightarrow[]{\text{$L^2$ proj.}} (\vec{u}_h^\mathrm{H},e_h^\mathrm{H})
\xrightarrow[\text{step size}~\Delta t]{\text{solve}~(\mathrm{P})} (\vec{u}_h^\mathrm{P},e_h^\mathrm{P})
\xrightarrow[]{\text{$L^2$ proj.}} \vec{U}_h^\mathrm{P}
\xrightarrow[\text{step size}~\frac{\Delta t}{2}]{\text{solve}~(\mathrm{H})} \vec{U}_h^{n+1}.
\end{align*}
At a given time step $n$, the numerical solution $\vec{U}_h^{n}$ is a piecewise polynomial. Usually it is impractical to have $\vec{U}_h^{n}(\vec{x})\in G$, for all $\vec{x}\in \Omega$, i.e, positivity holds everywhere. On the other hand, notice that the scheme is implemented with quadrature, thus it suffices to enforce positivity only at quadrature points.

\paragraph{\bf Quadratures and basis}
We utilize different quadrature rules for different integral terms such as volume integrals and surface integrals. For $\IQ^k$ scheme, the quadrature rules employed in hyperbolic and parabolic subproblems are defined as follows: 
\begin{enumerate}[leftmargin=0.5cm]
\item For face integrals in ($\mathrm{H}$), we use the $(k+1)$-point Gauss quadrature. Denote the set of associated quadrature points here by $S_K^{\mathrm{H, face}}$ on a cell $K$.
\item For volume integrals in ($\mathrm{H}$), we use a quadrature rule constructed by the tensor product of Gauss quadrature and request this quadrature is accurate for at least $(2k+1)$-order polynomials. Denote the set of associated quadrature points here by $S_K^{\mathrm{H, vol}}$ on a cell $K$.
\item For all (face and volume) integrals in ($\mathrm{P}$), we use a quadrature rule constructed by the tensor product of $(k+1)$-point Gauss--Lobatto quadrature. Denote the set of associated quadrature points here by $S_K^{\mathrm{P}}$ on a cell $K$.
\end{enumerate}
In addition, we consider the points for weak positivity of the compressible Euler equations \cite{zhang2017positivity}, which are constructed by $(k+1)$-point Gauss quadrature tensor product with $N$-point Gauss--Lobatto quadrature in both $x$ and $y$ directions  and we request $2N-3\geq k$. Let $S_K^{\mathrm{H, aux}}$ denote a collection from these points, where each point in $S_K^{\mathrm{H, aux}}$ is located on the interior of a cell $K$. 
\par
As an example, we illustrate the quadrature points in $\IQ^2$ scheme.  
The red points on the left of Figure~\ref{fig:quadrature_rule} are used for computing the face integrals of numerical fluxes along the cell boundary when solving the hyperbolic subproblem. The black points together with the red points form a special quadrature for weak positivity. Notice, the black points are not used in computing any numerical integrals \cite{zhang2017positivity}.
The red points in the middle of Figure~\ref{fig:quadrature_rule} are used for computing the volume integrals of numerical fluxes when solving the hyperbolic subproblem. 
The blue points on the right of Figure~\ref{fig:quadrature_rule} are used for computing all of the integrals when solving the parabolic subproblem.
\par
Let $\hat{K} = [-\frac{1}{2},\frac{1}{2}]^d$ be the reference element.
For $\IQ^k$ scheme, we use $(k+1)^d$ Gauss--Lobatto points to construct Lagrange interpolation polynomials, which serve as basis functions. For example, the blue points in Figure~\ref{fig:quadrature_rule} are used for constructing the bases of our $\IQ^2$ scheme.
The total number of bases on $\hat{K}$, namely the number of local degrees of freedom is $\Nloc=(k+1)^d$. 
Let $\hat{\vec{q}}_\nu$ denote the $\nu^\mathrm{th}$ Gauss--Lobatto point on $\hat{K}$, where $\nu=0,\cdots,\Nloc-1$. We assign a basis with an index $j$, if it equals to $1$ when evaluated it at $\hat{\vec{q}}_j$. From this construction, we have $\hat{\varphi}_j(\hat{\vec{q}}_\nu) = \delta_{j\nu}$, where $\delta$ denotes the Kronecker delta.
Let $\vec{F}_i:~\hat{K}\rightarrow K_i$ denote the invertible mapping from the reference element $\hat{K}$ to $K_i\in \setE_h$, then the basis functions on element $K_i$ are defined by $\varphi_{ij} = \hat{\varphi}_j\circ\vec{F}_i^{-1}$.
Thus, we have $\varphi_{i_1j}(\vec{q}_{i_2\nu}) = \delta_{i_1i_2}\delta_{j\nu}$, which indicates the points $\vec{q}_{i\nu} = \vec{F}_i(\hat{\vec{q}}_\nu)$ are not only quadrature nodes but also representing all degrees of freedom on cell $K_i$.
It is obvious that these bases are numerically orthogonal with respect to $(k+1)^d$-point Gauss--Lobatto rule.
\begin{figure}[ht!]
\begin{center}
\begin{tabularx}{0.9\linewidth}{@{}C@{~~}C@{~~}C@{}}
\includegraphics[width=0.27\textwidth]{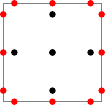} &
\includegraphics[width=0.27\textwidth]{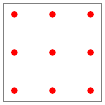} &
\includegraphics[width=0.27\textwidth]{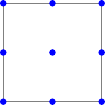} \\
\end{tabularx}
\caption{An illustration of quadratures used in $\IQ^2$ schemes. Left: Gauss quadrature tensor product with Gauss--Lobatto quadrature in both $x$ and $y$ directions. The points along the boundary are exactly $S_K^{\mathrm{H, face}}$, which are marked red. The other points in $S_K^{\mathrm{H, aux}}$ are marked black. Middle: Gauss quadrature tensor product with Gauss quadrature. The points in $S_K^{\mathrm{H, vol}}$ are marked red. Right: Gauss--Lobatto quadrature tensor product with Gauss--Lobatto quadrature. The points in $S_K^{\mathrm{P}}$ are marked blue.}
\label{fig:quadrature_rule}
\end{center}
\end{figure}

\paragraph{\bf The outline of proving positivity}
Let $S_K^\mathrm{H} = S_K^{\mathrm{H, face}}\cup S_K^{\mathrm{H, aux}}\cup S_K^{\mathrm{H, vol}}$ and let $S_h$ be the union of set $S_K^\mathrm{H}$, for all $K\in\setE_h$. 
On each time iteration of the fully discrete scheme, we apply the positivity-preserving limiter on the following quadrature points on each cell $K$.
\begin{enumerate}[leftmargin=0.5cm]
\item In Step~1 and Step~5, on each stage of SSP Runge--Kutta method, all points in set $S_K^{\mathrm{H}}$ need to be limited. 
As an example, for $\IQ^2$ scheme, all of the red and black points in Figure~\ref{fig:quadrature_rule}.
\item In Step~1, on the last stage of SSP Runge--Kutta method, all points in set $S_K^{\mathrm{H}}\cup S_K^{\mathrm{P}}$ need to be limited.
As an example, for $\IQ^2$ scheme, all of the red, black, and blue points in Figure~\ref{fig:quadrature_rule}.
\item In Step~4, all points in set $S_K^{\mathrm{H}}$ need to be limited. 
As an example, for $\IQ^2$ scheme, all of the red and black points in Figure~\ref{fig:quadrature_rule}.
\end{enumerate} 
\par
To prove our fully discrete scheme is positivity-preserving, we need to show $$\vec{U}_h^{n}(\vec{x})\in G,  \forall \vec{x}\in S_h ~~\Rightarrow~~ \vec{U}_h^{n+1}(\vec{x})\in G, \forall \vec{x}\in S_h$$ by the following steps:
\begin{enumerate}[leftmargin=0.5cm]
\item The positivity-preserving property of Runge–Kutta DG scheme for compressible Euler equations will be briefly reviewed in Section~\ref{sec:positivity:hyperbolic}.  
\item  In Section \ref{sec-pos-proj}, we will show that the simple positivity-preserving limiter can ensure positivity in the $L^2$ projection steps.
\item  In Section~\ref{sec:pos_P_high_order} and \ref{sec:positivity:parabolic}, we will show that the system matrix of \eqref{eq:CNS:P_full_dis3} in parabolic subproblem is monotone. Thus, the scheme preserves positivity of internal energy.
\end{enumerate}
% We have a collection of finite points, always exist a smallest one. Therefor, we can always find a sufficiently small $\epsilon > 0$ satisfies above.
\par
We emphasize that the first two steps above can be easily extended to unstructured meshes. But in the third step, the monotonicity of high order schemes only holds on uniform rectangular meshes. 
For the rest of this section, we only consider a uniform rectangular mesh for a computational domain $\Omega\subset\IR^d$.

\subsection{Positivity of hyperbolic subproblem and the positivity-preserving limiter}\label{sec:positivity:hyperbolic}
\par
One of the most popular approaches of constructing a positivity-preserving high order DG method for conservation laws was introduced by Zhang and Shu in \cite{zhang2010maximum, zhang2010positivity}, see also \cite{,zhang2012minimum,zhang2011positivity,xing2010positivity,wang2012robust,xu2017bound, zhang2017positivity}.
A high-order SSP Runge--Kutta method \eqref{rk3} is a convex combination of several forward Euler steps, thus the positivity of forward Euler time discretization of \eqref{eq:hyper_space_dis} also carries over to Runge--Kutta method \eqref{rk3} due to the convex combination.
\par

Define  the numerical admissible state set  $G^\epsilon$ as
\begin{align*}
G^\epsilon = \{\vec{U}=\transpose{[\rho,\vec{m},E]}\!:~\rho\geq\epsilon,~ \rho e(\vec{U}) = E - \frac{\|\vec{m}\|^2}{2\rho} \geq \epsilon\},
\end{align*}
where $\epsilon$ is a small positive number.
Let $\vec{U}_K(\vec{x})$ denote the DG solution polynomial on a cell $K$ and $\overline{\vec{U}}_K$ be its cell average on $K$.
The main results in \cite{zhang2010maximum, zhang2010positivity} include a sufficient condition for positivity of  cell averages $\overline{\vec{U}}_K^{n+1}\in G^\epsilon$
in the forward Euler discretization of high order DG schemes \eqref{eq:hyper_space_dis} and a simple positivity-preserving limiter to enforce the sufficient condition without destroying conservation and high order accuracy. 
To be specific, the sufficient condition for $\overline{\vec{U}}_K^{n+1}\in G^\epsilon$ is to have certain special quadrature point values of ${\vec{U}}_K^{n}$ to be in $G^\epsilon$, as well as a typical hyperbolic type CFL condition. We emphasize that this special quadrature merely serves as a sufficient condition for positivity of $\overline{\vec{U}}_K^{n+1}\in G^\epsilon$ and it should not be used for computing any integrals. We refer to \cite{zhang2017positivity} for a review of these conditions. 
\par
The positivity-preserving limiter modifies the DG polynomial solution 
$\vec{U}_h(\vec{x}) = \transpose{[\rho_h(\vec{x}), \vec{m}_h(\vec{x}), E_h(\vec{x})]}$
with the following two steps under the assumption that the cell average is positive $\overline{\vec{U}}_K \in G^\epsilon$.
\begin{itemize}[leftmargin=0.5cm]
\item[] 1. First enforce positivity of density by
\begin{align*}
\hat{\rho}_K(\vec{x})=\theta_\rho (\rho_K(\vec{x})-\overline{\rho}_K)+\overline{\rho}_K,
~~\text{where}~
\theta_{\rho}=\min\Big\{1, \frac{\overline{\rho}_K-\epsilon}{\overline{\rho}_K-\min\limits_{\vec{x}\in S_K}\rho_K(\vec{x})}\Big\}.
\end{align*}
In above, $\overline{\rho}_K$ denotes the cell average of $\rho_K$ on $K$. Notice that $\hat{\rho}_K$ and $\rho_K$ have the same cell average, and  $\hat{\rho}_K(\vec{x})={\rho}_K(\vec{x})$ if $\min\limits_{\vec{x}\in S_K}\rho_K(\vec{x})\geq \epsilon.$
\item[] 2. Define $\widehat{\vec{U}}(\vec{x})=\transpose{[\hat\rho(\vec{x}), \vec{m}(\vec{x}), E(\vec{x})]}$ and enforce positivity of internal energy by
\begin{align*}
\widetilde{\vec{U}}_K(\vec{x}) = \theta_e (\widehat{\vec{U}}(\vec{x}) - \overline{\vec{U}}_K) + \overline{\vec{U}}_K,
~~\text{where}~ 
\theta_{e}=\min\Big\{1, \frac{\overline {\rho e }_K-\epsilon}{\overline {\rho e}_K-\min\limits_{\vec{x}\in S_K}\rho e_K(\vec{x})}\Big\}.
\end{align*}
In above, $\overline{\rho e}_K=\overline E_K-\frac12\frac{\|\overline{\vec{m}}_K\|^2}{\overline\rho_K}$
and $\rho e(\vec{x})=E(\vec{x})-\frac12 \frac{\|\vec{m}(\vec{x})\|^2}{\rho(\vec{x})}$.
\end{itemize}
We refer to \cite{zhang2017positivity} for details of the sufficient condition of positivity of cell averages, the CFL condition, and the rigorous justification of the high order accuracy of such a simple limiter.  

\subsection{The positivity of the $L^2$ projection steps} \label{sec-pos-proj}
\par For the quadrature rule in the projection steps \eqref{eq:CNS:P_full_dis_L2proj1} and \eqref{eq:CNS:P_full_dis_L2proj2}, we simply use the tensor product of $(k+1)$-point Gauss--Lobatto quadrature. As an example, for $\IQ^2$ scheme, we use the blue points in Figure~\ref{fig:quadrature_rule}. 
It is straightforward to verify that this quadrature satisfy the condition for preserving conservation in Section~\ref{sec:numerical_scheme}, since it is exact for integrating $\IQ^k$ polynomials.
\par
Next we show the $L^2$ projections in \eqref{eq:CNS:P_full_dis_L2proj1} and \eqref{eq:CNS:P_full_dis_L2proj2} preserve positivity. 
Since the $L^2$ projection is local, we only need to consider a cell $K_i$.
Recall that the basis functions are constructed by using Lagrange interpolation polynomials at $(k+1)^d$ Gauss--Lobatto points and they are numerically orthogonal with respect to the employed Gauss--Lobatto rule.
Thus, the coefficients of basis functions also represent the values of DG solution polynomials at associated Gauss--Lobatto points. We use subscript $ij$ to denote the point value on cell $K_i$ at the $j^\mathrm{th}$ Gauss--Lobatto node.
We have the following results.
\begin{lemma}\label{lem:L2_proj_pos1}
If $\vec{U}_h^{\mathrm{H}}(\vec{x})\in G^\epsilon$, for all $\vec{x}\in S^\mathrm{H}_{K_{i}}$, then after applying the positivity-preserving limiter to $\vec{U}_h^{\mathrm{H}}$ on all points in $S^\mathrm{P}_{K_i}$ and taking the $L^2$ projection, we have $\rho_h^\mathrm{H}(\vec{q}_{ij})\geq\epsilon$ and $\rho_h^\mathrm{H}(\vec{q}_{ij}) e_h^\mathrm{H}(\vec{q}_{ij}) \geq \epsilon$, for all Gauss--Lobatto points $\vec{q}_{ij}\in S_{K_i}^{\mathrm{P}}$. 
\end{lemma} 
\begin{proof}
The condition $\vec{U}_h^\mathrm{H}(\vec{x})\in G^\epsilon$, for all $\vec{x}\in S^\mathrm{H}_{K_i}$, implies $\overline{\vec{U}_h^\mathrm{H}}_{K_i} \in G^\epsilon$. 
Applying the positivity-preserving limiter on all Gauss--Lobatto points $\vec{q}_{ij} \in S^\mathrm{P}_{K_i}$, we obtain $\rho_h^\mathrm{H}(\vec{q}_{ij})\geq\epsilon$ and $\on{\rho e(\vec{U}_h^\mathrm{H})}{\vec{q}_{ij}} \geq\epsilon$.
By taking test functions $\vec{\theta}_h = \vec{e}_\ell\varphi_{ij}$ and $\chi_h = \varphi_{ij}$ in \eqref{eq:CNS:P_full_dis_L2proj1}, due to the numerical orthogonality of the Lagrange bases, we get
$\vec{m}_{ij}^\mathrm{H} = \rho_{ij}^\mathrm{H} \vec{u}_{ij}^\mathrm{H}$
and $E^\mathrm{H}_{ij} = \rho^\mathrm{H}_{ij} e^\mathrm{H}_{ij} + \frac{1}{2}\rho^\mathrm{H}_{ij}\norm{\vec{u}^\mathrm{H}_{ij}}{}^2$. 
Therefore, we have
\begin{align*}
\rho^\mathrm{H}_{ij} e^\mathrm{H}_{ij} 
= E^\mathrm{H}_{ij} - \frac{1}{2}\rho^\mathrm{H}_{ij}\norm{\vec{u}^\mathrm{H}_{ij}}{}^2 
= E^\mathrm{H}_{ij} - \frac{\|\vec{m}_{ij}^\mathrm{H}\|^2}{2\rho_{ij}^\mathrm{H}}
= \on{\rho e(\vec{U}_h^\mathrm{H})}{\vec{q}_{ij}} \geq\epsilon.
\end{align*}
\end{proof}
\begin{lemma} \label{lem:L2_proj_pos2}
If $\rho_h^\mathrm{P}(\vec{q}_{ij}) \geq \epsilon$ and $\rho_h^\mathrm{P}(\vec{q}_{ij}) e_h^\mathrm{P}(\vec{q}_{ij}) \geq \epsilon$, for all Gauss--Lobatto points $\vec{q}_{ij} \in S_{K_i}^\mathrm{P}$, then after taking the $L^2$ projection and applying the positivity-preserving limiter to $\vec{U}_h^{\mathrm{P}}$ on all points in $S^\mathrm{H}_{K_i}$ we have $\vec{U}_h^{\mathrm{P}}(\vec{x})\in G^\epsilon$, for all $\vec{x} \in S^{\mathrm{H}}_{K_i}$. 
\end{lemma}
\begin{proof}
The density $\rho_h^\mathrm{P}$ equals to $\rho_h^\mathrm{H}$. Thus, we only need to show the positivity of internal energy. 
By taking test functions $\vec{\theta}_h = \vec{e}_\ell\varphi_{ij}$ and $\chi_h = \varphi_{ij}$ in \eqref{eq:CNS:P_full_dis_L2proj2}, due to the numerical orthogonality of the Lagrange bases, we have $\vec{m}_{ij}^\mathrm{P} = \rho_{ij}^\mathrm{P} \vec{u}_{ij}^\mathrm{P}$ and $E^\mathrm{P}_{ij} = \rho^\mathrm{P}_{ij} e^\mathrm{P}_{ij} + \frac{1}{2}\rho^\mathrm{P}_{ij}\norm{\vec{u}^\mathrm{P}_{ij}}{}^2$. 
By $\rho^\mathrm{P}_{ij} = \rho_h^\mathrm{P}(\vec{q}_{ij})$ and $e^\mathrm{P}_{ij} = e_h^\mathrm{P}(\vec{q}_{ij})$, we have
\begin{align*}
\on{\rho e(\vec{U}^\mathrm{P}_h)}{\vec{q}_{ij}}
= E_{ij}^\mathrm{P} - \frac{\|\vec{m}_{ij}^\mathrm{P}\|^2}{2\rho_{ij}^\mathrm{P}}
= E_{ij}^\mathrm{P} - \frac{1}{2}\rho_{ij}^\mathrm{P}\|\vec{u}_{ij}^\mathrm{P}\|^2 
= \rho^\mathrm{P}_{ij} e^\mathrm{P}_{ij} \geq \epsilon.
\end{align*}
With the ideal gas equation of state, the $\rho e$ is concave with respect to $\vec{U}$, see \cite{zhang2017positivity}. By Jensen's inequality, we have
\begin{align*}
\overline{\rho e}_{K_i} 
= \rho e(\overline{\vec{U}^\mathrm{P}_h}_{K_i}) 
= \rho e\Big(\sum_{j=0}^{\Nloc-1}\hat{\omega}_j \vec{U}^\mathrm{P}_{ij}\Big)
\geq \sum_{j=0}^{\Nloc-1}\hat{\omega}_j \on{\rho e(\vec{U}^\mathrm{P}_h)}{\vec{q}_{ij}} \geq \epsilon,
\end{align*}
where the $\hat{\omega}_j$ denotes the $j^\mathrm{th}$ Gauss--Lobatto quadrature weights on the reference element. Thus, the cell average $\overline{\vec{U}_h^\mathrm{P}}_{K_i}\in G^\epsilon$.
Applying the positivity-preserving limiter on points in $S_{K_i}^{\mathrm{H}}$ gives $\vec{U}_h^\mathrm{P}(\vec{x})\in G^\epsilon$, for all $\vec{x}\in S_{K_i}^{\mathrm{H}}$.
\end{proof}
The Lemma~\ref{lem:L2_proj_pos1} implies: if the positivity-preserving limiter is applied on all Lagrange node points in the last stage of SSP Runge--Kutta method on Step~1, then after taking the $L^2$ projection on Step~2 the internal energy is positive at each Lagrange node point.
The Lemma~\ref{lem:L2_proj_pos2} implies: if the solution of \eqref{eq:CNS:P_full_dis3} is positive on all Lagrange node points, then applying the positivity-preserving limiter on Step~4 guarantees the input of Step~5 is positive on set $S_h$. 

\subsection{Positivity of high-order scheme for parabolic subproblem}\label{sec:pos_P_high_order}
\par A matrix $\vecc{A}$ is monotone if all entries of its inverse are nonnegative, namely $\vecc{A}^{-1}\geq 0$. In the rest of this paper, all inequalities related with matrices are entry-wise inequalities. A matrix $\vecc{A}$ is called an M-matrix if it can be expressed in the form $\vecc{A} = s\vecc{I} - \vecc{B}$, where $\vecc{B}\geq 0$ and $s$ is greater than or equal to the spectral radius of $\vecc{B}$. A non-singular M-matrix is inverse-positive, thus is monotone \cite{plemmons1977m}. 
\par
A convenient way to obtain a sufficient condition on the positivity of internal energy is by proving the monotonicity of a system matrix. 
To be precise,  consider a linear system $(\vecc{M} + \Delta{t}\vecc{L})\vec{x} = \vec{b}$, where the matrix $\vecc{M}$ is diagonal with strictly positive diagonal entries; the matrix $\vecc{L}$ is an approximation of the Laplace operator such that $\vecc{L}\vec{1}=\vec{0}$.
Assume  the right-hand side vector satisfies $\vecc{M}^{-1} \vec{b}\geq\epsilon$,
then $(\vecc{I} + \Delta{t}\vecc{M}^{-1}\vecc{L})\vec{x} = \vecc{M}^{-1}\vec{b} \geq \epsilon$.

Since $(\vecc{I} + \Delta{t}\vecc{M}^{-1}\vecc{L})\vec{1} = \vec{1}$, 
each row of $(\vecc{I} + \Delta{t}\vecc{M}^{-1}\vecc{L})^{-1}$ sums to one.
Notice that if $\vecc{I} + \Delta{t}\vecc{M}^{-1}\vecc{L}$ is monotone, then
each row of $(\vecc{I} + \Delta{t}\vecc{M}^{-1}\vecc{L})^{-1}$ has nonnegative entries thus forms a convex combination coefficients, thus $\vec{x} \geq \epsilon$.

Since  $\vecc{M}^{-1}>0$, $(\vecc{M} + \Delta{t}\vecc{L})^{-1}\geq 0\Leftrightarrow (\vecc{I} + \Delta{t}\vecc{M}^{-1}\vecc{L})^{-1}\geq 0$. Thus, the monotonicity of $\vecc{M} + \Delta{t}\vecc{L}$ is sufficient for positivity of $\vec{x}$. 
 
\par
In order to obtain a monotone system matrix, we use either the IIPG method with $\IQ^1$ element or the spectral element method with $\IQ^k$ ($k=2,3$) element to discretize the Laplace operator $-\laplace{e}$ in \eqref{eq:time:step3_2}.

\subsubsection{Preserve positivity through the IIPG method}
Consider \eqref{eq:CNS:P_full_dis3} in a matrix formulation. The entry of a matrix with row index $\Nloc i' + j'$ and column index $\Nloc i + j$ is denoted by $[\cdot]_{i'j';ij}$. The entry of a vector with index $\Nloc i' + j'$ is denoted by $[\cdot]_{i'j'}$. 
Given $\rho^{\mathrm{H}}_h$, $\rho^{\mathrm{P}}_h$, $\vec{u}^{\ast}_h$, and $e^{\mathrm{H}}_h$, we define the following matrices and vectors:
\begin{align*}
[\vecc{A}_{\mathcal{M}}]_{i'j';ij} &= \langle\rho^\mathrm{P}_h \varphi_{ij},\varphi_{i'j'}\rangle, &
[\vecc{A}_{\mathcal{D}}]_{i'j';ij} &= a_{\mathcal{D}}(\varphi_{ij},\varphi_{i'j'}), & 
[\vec{B}_\strain]_{i'j'} &= b_\strain(\vec{u}_h^{\ast},\varphi_{i'j'}), \\
[\vec{B}_{\mathcal{M}}]_{i'j'} &= \langle\rho_h^\mathrm{H} e_h^\mathrm{H},\varphi_{i'j'}\rangle, &
[\vec{B}_{\mathcal{D}}]_{i'j'} &= b_\mathcal{D}(\varphi_{i'j'}), &
[\vec{B}_\lambda]_{i'j'} &= b_\lambda(\vec{u}_h^{\ast},\varphi_{i'j'}).
\end{align*} 
Then, the matrix formulation of \eqref{eq:CNS:P_full_dis3} reads: find vector $\vec{X}_e^\mathrm{P}$, where $[\vec{X}_e^\mathrm{P}]_{ij} = e_{ij}^\mathrm{P}$, such that:
\begin{align}\label{eq:pos:paraboloc_mat_form}
(\vecc{A}_{\mathcal{M}} + \frac{\Delta t \lambda}{\Rey}\vecc{A}_{\mathcal{D}}) \vec{X}_e^\mathrm{P}
= \vec{B}_{\mathcal{M}} + \frac{\Delta t}{\Rey}\vec{B}_{\strain} + \frac{2\Delta t}{3\Rey}\vec{B}_\lambda + \frac{\Delta t \lambda}{\Rey}\vec{B}_{\mathcal{D}}.
\end{align}
Recall we use $(k+1)^d$-point Gauss--Lobatto quadrature rule to compute all of the numerical integrals in parabolic subproblem and the bases are numerically orthogonal. The matrix $\vecc{A}_\mathcal{M}$ is diagonal with strictly positive diagonal entries. The $e_{ij}^\mathrm{P}$ represents the value of solution polynomial $e_h^\mathrm{P}$ evaluated at Gauss--Lobatto point $\vec{q}_{ij}$. 
The following lemma shows that the right-hand side of system \eqref{eq:pos:paraboloc_mat_form} is positive.
\begin{lemma}\label{lem:pos:parabolic_right_side}
On each cell $K_i\in\setE_h$, if $\rho^\mathrm{H}_h(\vec{q}_{ij})>0$ and $e^\mathrm{H}_h(\vec{q}_{ij})>0$, for all $\vec{q}_{ij}\in S^\mathrm{P}_{K_i}$. Then, under $(k+1)^d$-point Gauss--Lobatto quadrature rule, for any penalty $\sigma\geq0$ and $\tilde{\sigma}\geq0$, each entry of the right-hand side of \eqref{eq:pos:paraboloc_mat_form} is positive.
\end{lemma}
\begin{proof} 
By numerical orthogonality of the Lagrange bases with respect to the $(k+1)^d$-point Gauss--Lobatto quadrature rule, the condition $\rho^\mathrm{H}_h(\vec{q}_{ij})>0$ and $e^\mathrm{H}_h(\vec{q}_{ij})>0$, for all $\vec{q}_{ij}\in S^\mathrm{P}_{K_i}$, implies $[\vec{B}_{\mathcal{M}}]_{ij} = \Delta{x}^2 \hat{\omega}_j \rho_{ij}^\mathrm{H} e_{ij}^\mathrm{H} >0$. Here, $\hat{\omega}_j$ denotes the $j^\mathrm{th}$ Gauss--Lobatto quadrature weight on the reference element.
\par
For the second and third terms on the right-hand side of \eqref{eq:pos:paraboloc_mat_form}, we recall the support of DG basis function $\varphi_{ij}$ is cell $K_i$ and 
\begin{multline}\label{eq:pos:parabolic_rightside}
b_\strain(\vec{u}_h^{\ast},\varphi_{ij}) + \frac{2}{3}b_\lambda(\vec{u}_h^{\ast},\varphi_{ij})
= 2 \int_{K_i} \Big(\strain{(\vec{u}_h^{\ast})}:\strain{(\vec{u}_h^{\ast})} - \frac{1}{3}(\div{\vec{u}_h^{\ast}})^2\Big)\varphi_{ij}\\
+ \frac{\sigma}{h} \int_{\partial{K_i}\subset\Gammah} \jump{\vec{u}_h^{\ast}}\cdot\jump{\vec{u}_h^{\ast}}\avg{\varphi_{ij}}
+ \frac{\sigma}{h} \int_{\partial{K_i}\subset\partial\Omega_\mathrm{D}} (\vec{u}_h^{\ast}-\vec{u}_\mathrm{D})\cdot(\vec{u}_h^{\ast}-\vec{u}_\mathrm{D}) \varphi_{ij}.
\end{multline}
Then, the term $[\vec{B}_{\strain}]_{ij} + \frac{2}{3}[\vec{B}_\lambda]_{ij}$ equals to $(k+1)^d$-point Gauss--Lobatto integral of \eqref{eq:pos:parabolic_rightside}. 
By tensor inequality $\strain(\vec{u}):\strain(\vec{u}) \geq \frac{1}{d}(\div{\vec{u}})^2$, we obtain $\on{(\strain{(\vec{u}_h^{\ast})}:\strain{(\vec{u}_h^{\ast})} - \frac{1}{3}(\div{\vec{u}_h^{\ast}})^2)}{\vec{q}_{ij}} \geq 0$, for all $\vec{q}_{ij} \in S^\mathrm{P}_{K_i}$ when dimension $d\leq3$.
Notice, from the bases construction, we always have $\varphi_{i_1 j}(\vec{q}_{i_2 \nu}) = \delta_{i_1 i_2}\delta_{j \nu} \geq0$. % Although, the $\IQ^1$ Lagrange bases are non-negative and the $\IQ^k$ Lagrange bases for $k\geq2$ can be negative.
Thus, as long as the penalty $\sigma\geq0$, we have $[\vec{B}_{\strain}]_{ij} + \frac{2}{3}[\vec{B}_\lambda]_{ij}\geq 0$.
\par
Finally, it is straightforward to see the last term on the right-hand side of \eqref{eq:pos:paraboloc_mat_form} is always non-negative, since the Dirichlet boundary condition $e_\mathrm{D}>0$ and penalty $\tilde{\sigma}\geq0$.
\end{proof}
\par
In Step~3 of the fully discrete scheme, we have $\rho_h^\mathrm{P} = \rho_h^\mathrm{H}$. Furthermore, the system matrix $\vecc{A}_{\mathcal{M}} + \frac{\Delta t \lambda}{\Rey}\vecc{A}_{\mathcal{D}}$ associated with the $\IQ^1$ IIPG discretization has an M-matrix structure unconditionally. We include the proof in \ref{sec:positivity:parabolic}. Therefore, we obtain $e_h^\mathrm{H}(\vec{q}_{ij})>0 \Rightarrow e_h^\mathrm{P}(\vec{q}_{ij})>0$, for all of the Gauss--Lobatto points $\vec{q}_{ij}\in S^\mathrm{P}_{K_i}$.

\subsubsection{Preserve positivity through the spectral element method}
Except the fourth order compact finite difference scheme \cite{li2018high}, no known high order schemes have an M-matrix structure. On the other hand,
M-matrix structure is only a sufficient but rather than a necessary condition for monotonicity. In particular, a matrix is monotone if it is a product of some M-matrices. For example,
$\vecc{A}=\vecc{M}_1\vecc{M}_2$ where $\vecc{M}_1$ and $\vecc{M}_2$ are both M-matrices, then $\vecc{A}$ is still monotone  since $\vecc{A}^{-1}=\vecc{M}_2^{-1}\vecc{M}_1^{-1}\geq 0$.

The $\IQ^k$ continuous finite element method implemented by $(k+1)^d$-point Gauss--Lobatto quadrature is also called the spectral element method \cite{maday1990optimal}. 
In \cite{li2020monotonicity}, it is proven that $\IQ^2$ spectral element method is a product of two M-matrices thus is monotone for a variable coefficient elliptic operator $-\div{(a\grad{u})}+cu$ under suitable mesh constraints.
In \cite{cross2020monotonicity}, $\IQ^3$ spectral element method  is proven to be  a product of four M-matrices for the Laplacian operator thus monotone.
The monotonicity of $\IQ^2$ spectral element method has been used to construct high order accurate positivity-preserving schemes for Keller--Segel, Allen--Cahn, and Fokker--Planck equations  \cite{hu2021positivity,shen2021discrete,liu-2022-monotone}. 
\par
In this paper, we simply apply the existing monotonicity results in
$\IQ^2$ spectral element method \cite{li2020monotonicity} and 
$\IQ^3$ spectral element method \cite{cross2020monotonicity} 
to the Laplacian operator in \eqref{eq:time:step3_2} and couple it with the DG discretization \eqref{eq:CNS:P_full_dis2_1} in parabolic subproblem. 
For the sake of simplicity, consider the thermally insulating boundary condition $\grad{e}\cdot\normal = 0$ on the entire boundary of domain $\Omega$. 
Define continuous piecewise $\IQ^k$ polynomial space
\begin{align*}
\tilde{M}_h^k = \big\{\chi_h\in C(\Omega):~\forall K \in \setE_h,\, \on{\chi_h}{K} \in \IQ^k(K)\big\}.
\end{align*}
Recall in Step~3 of the fully discrete scheme, when solving \eqref{eq:CNS:P_full_dis3}, the $\rho^{\mathrm{H}}_h$, $\rho^{\mathrm{P}}_h$, $\vec{u}^{\ast}_h$, and $e^{\mathrm{H}}_h$ are given data. We replace \eqref{eq:CNS:P_full_dis3} by introducing the bilinear form $a_\mathrm{CG}: \tilde{M}_h^k\times\tilde{M}_h^k \rightarrow \IR$ and the linear form $b_\mathrm{CG}: \tilde{M}_h^k \rightarrow \IR$, as follows:
\begin{align*}
a_\mathrm{CG}(e_h, \chi_h) &= \int_\Omega \rho^\mathrm{P}_h e_h \chi_h + \frac{\Delta t\lambda}{\Rey}\int_\Omega \grad{e_h}\cdot\grad{\chi_h},\\
b_\mathrm{CG}(\chi_h) &= \int_\Omega \rho^\mathrm{H}_h e^\mathrm{H}_h \chi_h + \frac{\Delta t}{\Rey}b_\strain(\vec{u}_h^{\ast},\chi_h) + \frac{2\Delta t}{3\Rey}b_\lambda(\vec{u}_h^{\ast},\chi_h).
\end{align*}
Then, the variational formulation for solving \eqref{eq:time:step3_2} becomes: find $e^{\mathrm{P}}_h \in \tilde{M}_h^k$, such that for all $\chi_h\in\tilde{M}_h^k$, the $a_\mathrm{CG}(e_h^{\mathrm{P}}, \chi_h) = b_\mathrm{CG}(\chi_h)$ holds. 
For $\IQ^k$ scheme, applying $(k+1)^d$-point Gauss–Lobatto quadrature to compute integrals, the \eqref{eq:CNS:P_full_dis3} is replaced by: 
given $(\rho^{\mathrm{H}}_h, \rho^{\mathrm{P}}_h, \vec{u}^{\ast}_h, e^{\mathrm{H}}_h) \in M_h^k\times M_h^k\times \mathbf{X}_h^k\times M_h^k$, solve for $e^{\mathrm{P}}_h \in \tilde{M}_h^k$, such that for all $\chi_h\in\tilde{M}_h^k$,
\begin{align}\label{eq:CNS:P_full_dis3_CG}
\langle\rho^\mathrm{P}_h e_h^\mathrm{P},\chi_h\rangle + \frac{\Delta t\lambda}{\Rey}\langle\grad{e_h^\mathrm{P}},\grad{\chi_h}\rangle 
= \langle\rho_h^\mathrm{H} e_h^\mathrm{H},\chi_h\rangle + \frac{\Delta t}{\Rey}b_\strain(\vec{u}_h^{\ast},\chi_h) + \frac{2\Delta t}{3\Rey}b_\lambda(\vec{u}_h^{\ast},\chi_h).
\end{align} 
\begin{remark}
For two dimensional problems, if we set penalty $\sigma=0$, namely employ the NIPG0 method in $\IQ^2$ and $\IQ^3$ discretization for $\div{\vec{\tau}(\vec{u})}$ and $\vec{\tau}(\vec{u}):\grad{\vec{u}}$, then \eqref{eq:CNS:P_full_dis3_CG} is further simplified. We have
\begin{align*}
\langle\rho^\mathrm{P}_h e_h^\mathrm{P},\chi_h\rangle + \frac{\Delta t\lambda}{\Rey}\langle\grad{e_h^\mathrm{P}},\grad{\chi_h}\rangle 
= \langle\rho_h^\mathrm{H} e_h^\mathrm{H},\chi_h\rangle + \frac{2\Delta t}{\Rey}\langle\strain{(\vec{u}_h^{\ast})}:\strain{(\vec{u}_h^{\ast})},\chi_h\rangle - \frac{2\Delta t}{3\Rey}\langle(\div{\vec{u}_h^{\ast}})(\div{\vec{u}_h^{\ast}}),\chi_h\rangle.
\end{align*}
The above formula only involves volume integrals, which is convenient for implementation. And more importantly, with the NIPG0 method, we get rid of the face penalties, which minimizes the numerical viscosity.
\end{remark}

The identity $\langle\grad{e_h^\mathrm{P}},\grad{1}\rangle = 0$ acts in the same role as $a_{\mathcal{D}}(e_h^\mathrm{P},1) = 0$ in proving the conservation of total energy. Replacing \eqref{eq:CNS:P_full_dis3} with \eqref{eq:CNS:P_full_dis3_CG} does not affect the proof of Theorem~\ref{thm:dis_momentum_conv}. Therefore, the conservations of density, momentum, and total energy still hold.
Similar to Lemma~3, it is straightforward to verify the right-hand side vector stems from \eqref{eq:CNS:P_full_dis3_CG} is still positive.
For $\IQ^2$ spectral element scheme, by the results in Section~4 in \cite{li2020monotonicity}, we obtain a sufficient condition of monotonicity of the system matrix of \eqref{eq:CNS:P_full_dis3_CG} as follows:
\begin{subequations}\label{eq:parabolic_step_size_condition}
\begin{align}
\label{q2-cfl}
\Delta{t} > \frac{\Rey}{3\lambda} \max_{i,j}{\rho^\mathrm{H}_{ij}}\, \Delta{x}^2.
\end{align}
For $\IQ^3$ spectral element scheme, in principle it is possible to extend the same proof for $-\Delta u$ in Section~6 in \cite{cross2020monotonicity} to an operator like $-\Delta u+c u$ with a variable coefficent $c$. Thus in principle the monotonicity of the system matrix of \eqref{eq:CNS:P_full_dis3_CG} using $\IQ^3$ spectral element holds under a time step contraint like
\begin{align}
\label{q3-cfl}
\Delta t > C(\Rey,\lambda,\rho_h^\mathrm{H})\, \Delta{x}^2,
\end{align}
where $C$ is a constant depending on $\Rey,\lambda,\rho_h^\mathrm{H}$. 
\end{subequations}

We emphasize that the time step constraints \eqref{q2-cfl} and \eqref{q3-cfl} are lower bounds, i.e., the time step cannot be as small as $\Delta x^2$, which is a practical constraint, rather than an impossible one to implement. 

Finally, the unique existence of $e_h^\mathrm{P}$ is a conclusion from the monotonicity of the system matrix, since it is invertible.
Therefore, we obtain $e_h^\mathrm{H}(\vec{q}_{ij}) > 0 \Rightarrow e_h^\mathrm{P}(\vec{q}_{ij}) > 0$, for all of the Gauss--Lobatto points $\vec{q}_{ij}\in S^\mathrm{P}_{K_i}$.

\subsection{Adaptive time-stepping strategy and implementation}
\par
We use SSP Runge--Kutta method in the fully discrete scheme to solve the hyperbolic subproblem. By \cite{zhang2010positivity,zhang2017positivity}, for the compressible Euler equations on a structure mesh, a sufficient condition on preserving positivity in a single forward Euler step with step size $\Delta{t}^\mathrm{H}$ is 
\begin{align}\label{eq:hyper_CFL}
\frac{\Delta{t}^\mathrm{H}}{\Delta{x}} \max_{e}{\alpha_{e}} \leq \frac{1}{2}\hat{\omega} = \frac{1}{2}\frac{1}{N(N-1)},
\end{align}
where N is smallest integer satisfying $2N-3\geq k$ for $\IQ^k$ basis. 
For the parabolic subproblem, the $\IQ^k$ ($k=2,3$) scheme is positivity-preserving under the condition \eqref{eq:parabolic_step_size_condition}, which is a lower bound on the time step size.
These constraints together imply that for a simulation the mesh resolution $\Delta{x}$ should be small enough such that a feasible time step size exist when solving subproblem ($\mathrm{H}$) followed by subproblem ($\mathrm{P}$) in Strang splitting sequentially.
However, we should not simply use a time step suggested by these constraints for the compressible NS equations because of the following reasons.
\begin{enumerate}[leftmargin=0.5cm]
\item Mathematically, the \eqref{eq:parabolic_step_size_condition} and \eqref{eq:hyper_CFL} can be achieved at the same time if $\Delta x$ is small enough. However, \eqref{eq:parabolic_step_size_condition} and \eqref{eq:hyper_CFL} are only sufficient, but not necessary for  preserving positivity in practice.
\item To enforce \eqref{eq:hyper_CFL} in SSP Runge--Kutta method, we need to estimate $\max_{e}{\alpha_{e}}$ for each stage. However, it is difficult to accurately estimate this quantity for the two inner time stages in a third order SSP Runge--Kutta method.
\item  The wave speed contains $\sqrt{\gamma p/\rho}$, which will be  inaccurate for extremely low density problems due to the round-off errors.
\end{enumerate}
Instead, we can apply the following simple adaptive time-stepping strategy. 
At each time step $t^n$, given $\vec{U}_h^n(\vec{x})\in G^\epsilon$ for all $\vec{x}\in S_h$, we start with a trial step size $\Delta{t}^\mathrm{trial}$ by  
\begin{align}\label{eq:hyperbolic_CFL}
\Delta{t}^\mathrm{trial} = a \hat{\omega} \frac{1}{\max_{e}{\alpha_{e}}} \Delta x,
\end{align}
where $a$ is a parameter.  We will specify its value in our experiments, see Section~\ref{sec:experiments}.  
For solving hyperbolic subproblem, the time-stepping strategy is the same as in Section~3.2 in \cite{wang2012robust}, which is listed below for completeness:
\begin{itemize}[leftmargin=0.5cm]
\item[] {\bf Algorithm~H}. At time $t^n$, select a trial hyperbolic step size $\Delta{t}^\mathrm{H}$. The input DG polynomial $\vec{U}_h^n$ satisfies $\vec{U}_h^n(\vec{x})\in G^\epsilon$, for all $\vec{x}\in S_h$. The parameter $\epsilon$ can be set as $\epsilon = \min\{10^{-13}, \overline{\rho}^n_K, \overline{\rho e}^n_K\}$. 
\item[] Step~H1. Given DG polynomial $\vec{U}_h^n$, compute the first stage to obtain $\vec{U}_h^{(1)}$.
\begin{itemize}[leftmargin=1.0cm, label=\labelitemi]
\item If the cell averages $\overline{\vec{U}}_K^{(1)} \in G^\epsilon$, for all $K\in\setE_h$, then apply a positivity-preserving limiter to obtain $\widetilde{\vec{U}}_h^{(1)}$ and go to Step H2.
\item Otherwise, recompute the first stage with halved step size $\Delta{t}^\mathrm{H} \leftarrow \frac{1}{2}\Delta{t}^\mathrm{H}$. Notice, when $\Delta{t}^\mathrm{H}$ satisfies the hyperbolic CFL \eqref{eq:hyper_CFL}, the $\overline{\vec{U}}_K^{(1)} \in G^\epsilon$ is guaranteed.
\end{itemize}
\item[] Step~H2. Given DG polynomial $\widetilde{\vec{U}}_h^{(1)}$, compute the second stage to obtain $\vec{U}_h^{(2)}$.
\begin{itemize}[leftmargin=1.0cm, label=\labelitemi]
\item If the cell averages $\overline{\vec{U}}_K^{(2)} \in G^\epsilon$, for all $K\in\setE_h$, then apply a positivity-preserving limiter to obtain $\widetilde{\vec{U}}_h^{(2)}$ and go to Step H3.
\item Otherwise, return to Step H1 and restart the computation with halved step size $\Delta{t}^\mathrm{H} \leftarrow \frac{1}{2}\Delta{t}^\mathrm{H}$. Notice, even if $\Delta{t}^\mathrm{H}$ satisfies the constraint \eqref{eq:hyper_CFL} in Step H1, the $\overline{\vec{U}}_K^{(2)}$ still may not belong to set $G^\epsilon$, since \eqref{eq:hyper_CFL} is based on $\vec{U}_h^n$ rather than $\widetilde{\vec{U}}_h^{(1)}$.
\end{itemize}
\item[] Step~H3. Given DG polynomial $\widetilde{\vec{U}}_h^{(2)}$, compute the third stage to obtain $\vec{U}_h^{(3)}$.
\begin{itemize}[leftmargin=1.0cm, label=\labelitemi]
\item If the cell averages $\overline{\vec{U}}_K^{(3)} \in G^\epsilon$, for all $K\in\setE_h$, then apply a positivity-preserving limiter to obtain $\vec{U}_h^{\mathrm{H}}$. We finish the current SSP Runge--Kutta.
\item Otherwise, return to Step H1 and restart the computation with halved step size $\Delta{t}^\mathrm{H} \leftarrow \frac{1}{2}\Delta{t}^\mathrm{H}$. Notice, even if $\Delta{t}^\mathrm{H}$ satisfies the constraint \eqref{eq:hyper_CFL} in Step H1, the $\overline{\vec{U}}_K^{(3)}$ still may not belong to set $G^\epsilon$, since \eqref{eq:hyper_CFL} is based on $\vec{U}_h^n$ rather than $\widetilde{\vec{U}}_h^{(2)}$.
\end{itemize}
\end{itemize}
 The adaptive time-stepping strategy for solving the compressible NS equations can be now defined as follows.
At initial, the $\vec{U}_h^0$ is constructed by $L^2$ projection of $\vec{U}_0$ with a positive-preserving limiter on $S_h$, e.g., we have $\vec{U}_h^0(\vec{x})\in G^\epsilon$, for all $\vec{x}\in S_h$.
\begin{itemize}[leftmargin=0.5cm]
\item[] {\bf Algorithm~CNS}. At time $t^n$, select $\Delta{t} = \Delta{t}^\mathrm{trial}$ as a desired time step size. The input DG polynomial $\vec{U}_h^n$ satisfies $\vec{U}_h^n(\vec{x})\in G^\epsilon$, for all $\vec{x}\in S_h$. The parameter  $\epsilon$ is taken as $\epsilon = \min\{10^{-13}, \overline{\rho}^n_K, \overline{\rho e}^n_K\}$.
\item[] Step~CNS1. Given DG polynomial $\vec{U}_h^n$, solve subproblem $(\mathrm{H})$ form time $t^n$ to $t^{n} + \frac{\Delta t}{2}$.
\begin{itemize}[leftmargin=1.0cm, label=\labelitemi]
\item Set $m=0$. Let $t^{n,0} = t^{n}$ and $\vec{U}_h^{n,0} = \vec{U}_h^n$.
\item Given $\vec{U}_h^{n,m}$ at time $t^{n,m}$, solve $(\mathrm{H})$ to compute $\vec{U}_h^{n,m+1}$ by the Algorithm~H. Let $t^{n,m+1} = t^{n,m} + \Delta{t}^\mathrm{H}$. 
If $t^{n,m+1} = t^{n} + \frac{\Delta t}{2}$, then apply a positive-preserving limiter for $\vec{U}_h^{n,m+1}$ on all Gauss--Lobatto points in $S^\mathrm{P}_K$, for all $K\in\setE_h$, we obtain $\vec{U}_h^\mathrm{H}$. Go to Step~CNS2. Otherwise, set $m\leftarrow m+1$ and repeat solving $(\mathrm{H})$ by Algorithm~H until reach $t^{n} + \frac{\Delta t}{2}$. 
Notice, when compute $\vec{U}_h^{n,m+1}$, we can take the minimum of $\Delta{t}^\mathrm{trial}$ and $t^{n} + \frac{\Delta t}{2}-t^{n,m}$ as a trail $\Delta{t}^\mathrm{H}$ to start Algorithm~H.
\end{itemize}
Step~CNS2. Given DG polynomial $\vec{U}_h^\mathrm{H}$, take $L^2$ projection to compute $(\vec{u}_h^\mathrm{H}, e_h^\mathrm{H})$.
\item[] Step~CNS3. Given DG polynomials $(\rho_h^\mathrm{H}, \vec{u}_h^\mathrm{H}, e_h^\mathrm{H})$, solve subproblem $(\mathrm{P})$ form time $t^n$ to $t^{n} + \Delta t$. 
\begin{itemize}[leftmargin=1.0cm, label=\labelitemi]
\item If a negative internal energy $e^\mathrm{P}_h(\vec{q}_{ij})$ emerge, then goto Step~CNS1 and restart the computation with doubled time step size $\Delta t \leftarrow 2\Delta t$. 
\item Otherwise, go to Step~CNS4. Notice, for $\IQ^k$ ($k=2,3$) scheme, when $\Delta{t}$ satisfies \eqref{eq:parabolic_step_size_condition}, the positivity of internal energy is guaranteed.
\end{itemize}
Step~CNS4. Given DG polynomials $(\rho_h^\mathrm{P}, \vec{u}_h^\mathrm{P}, e_h^\mathrm{P})$, take $L^2$ projection follows by applying a positivity-preserving limiter on all points in $S_h$ to compute $\vec{U}_h^\mathrm{P}$.
\item[] Step~CNS5. Given DG polynomial $\vec{U}_h^\mathrm{P}$, use adaptive time-stepping strategy to solve subproblem $(\mathrm{H})$ form time $t^n + \frac{\Delta t}{2}$ to $t^{n} + \Delta{t}$.
\end{itemize}

Notice that the time-stepping strategy above can easily result in an endless loop for a general spatial discretization. However,  since \eqref{eq:parabolic_step_size_condition} and \eqref{eq:hyper_CFL} are sufficient conditions for positivity,  \eqref{eq:parabolic_step_size_condition} and \eqref{eq:hyper_CFL}  ensure that there will be no endless loops when using this  time-stepping strategy with the fully discretized scheme in this paper.
\begin{remark}
Our $\IQ^1$ DG scheme for solving subproblem $(\mathrm{P})$ is unconditional positivity-preserving, since the associated system matrix enjoys an M-matrix structure unconditionally, see \ref{sec:positivity:parabolic}. Therefore, for the $\IQ^1$ DG scheme, we do not need to adapt time step size with respect to the parabolic subproblem, i.e., Step CNS3 always passes without recomputation.
In practice, we can relax the condition for doubling time step size in Step CNS3, since it is not necessary to request the internal energy to be positive at each Gauss--Lobatto point. We can double the time step size only when a negative cell average $\overline{\vec{U}^\mathrm{P}_h}_{K}$ in Step CNS4 emerges.
We only observed Step CNS3 recomputation in the first several time steps of $\IQ^2$ and $\IQ^3$ Sedov blast wave simulations. For all of the rest numerical experiments in Section~\ref{sec:experiments}, Step CNS3 recomputation is not triggered.
\end{remark}

%% file: Content/experiments.tex
%%%%%%%%%%%%%%%%%%%%%%%%%%%%%%%%%%%%%%%%%%%%%%%%%%%%%%%%%%%%%%%%%%%%%%%%%%%%%%%%%%%%%%%%%%%%%%%%%%%%%%%%%%%%%%%%%%%%%%%
\section{Numerical tests}\label{sec:experiments}
%%%%%%%%%%%%%%%%%%%%%%%%%%%%%%%%%%%%%%%%%%%%%%%%%%%%%%%%%%%%%%%%%%%%%%%%%%%%%%%%%%%%%%%%%%%%%%%%%%%%%%%%%%%%%%%%%%%%%%%
We consider some representative tests for validating our numerical scheme in one and two-dimensional spaces, including  the Lax shock tube, the double rarefraction, Sedov blast wave, shock diffraction, shock reflection, and shock reflection-diffraction problems.
\par
The parameters for all the tests are as follows. We use the ideal gas constants $\gamma = 1.4$ and Prandtl number $\Pr = 0.72$.
For the penalty parameters in IPDG method for solving ($\mathrm{P}$),
in the $\IQ^1$ scheme, we set $\sigma = 2$ on $\Gammah$, $\sigma = 4$ on $\partial{\Omega}$, and $\tilde{\sigma} = 2$; 
in the $\IQ^2$ and $\IQ^3$ schemes, we take NIPG0 method, namely set penalty $\sigma = 0$ on all faces. Since we use the continuous finite element to discretize the term $-\laplace{e}$ in $\IQ^2$ and $\IQ^3$ spaces, thus there is no $\tilde{\sigma}$ involved.

We emphasize that only the positivity-preserving limiter is used in the Runge--Kutta DG scheme for the hyperbolic subproblem, and no limiters are used in the parabolic subproblem, even though other limiters, such as TVB limiter \cite{zhang2012maximum} and WENO type limiters \cite{qiu2005runge,zhong2013simple,zhu2013runge}, for reducing oscillations could be used to improve quality of numerical solutions.

\subsection{Spatial order of accuracy for smooth solutions in two dimensions}
We test the accuracy in space for smooth solutions.
We utilize the manufactured solution method on domain $\Omega=[0,1]^2$ and set the end time $T = 0.1024$. The prescribed density, velocity, and internal energy are as follows:
\begin{align*}
\rho &= \exp{(-t)} \sin{2\pi (x+y)} + 2,\\
\vec{u} &= 
\begin{bmatrix}
\exp{(-t)} \cos{(2\pi x)} \sin{(2\pi y)} + 2\\
\exp{(-t)} \sin{(2\pi x)} \cos{(2\pi y)} + 2
\end{bmatrix},\\
e &= \textstyle{\frac{1}{2}}\exp{(-t)} \cos{(2\pi x)}\cos{(2\pi y)} + 1.
\end{align*} 
The total energy and pressure are computed by $E = \rho e + \frac{1}{2}\rho \norm{\vec{u}}{}^2$ and $p = (\gamma-1)\rho e$.
The system right-hand side functions are evaluated from above manufactured solutions, as well as the initial and boundary conditions are imposed by the same prescribed solutions.
\par
We choose $\Rey = 1$ and $\lambda=1$ and use the same IPDG penalties as in the physical simulations for solving ($\mathrm{P}$), e.g., for $\IQ^1$ scheme, we set $\sigma = 2$ on $\Gammah$, $\sigma = 4$ on $\partial{\Omega}$, and $\tilde{\sigma} = 2$; for $\IQ^2$ and $\IQ^3$ schemes, we take NIPG0 method by setting penalty $\sigma = 0$ on all faces. 
Note, there is no parameter $\tilde{\sigma}$ involved in $\IQ^2$ and $\IQ^3$ schemes, since we use the continuous finite element to discrete the term $-\laplace{e}$.
\par 
We obtain spatial convergence rates by computing the solutions on a sequence of uniformly refined meshes with fixed time step size $\Delta t = 2^{-4}\cdot10^{-4}$. 
This time step size is small enough, such that the spatial error dominates and the hyperbolic CFL is satisfied. 
Define the discrete $L^2_h$ error of density by
\begin{align*}
\norm{\rho_h^{n}-\rho(t^n)}{L^2_h}^2 = {\Delta x}^2 \sum_{i=0}^{\Nel-1} \sum_{\nu=0}^{N_\mathrm{q}^\mathrm{H, vol}-1}\omega_\nu \Big|\sum_{j=0}^{\Nloc-1}\rho_{ij}^n\,\hat{\varphi}_j(\hat{\vec{q}}_\nu) - \rho(t^n)\circ \vec{F}_i(\hat{\vec{q}}_\nu)\Big|^2,
\end{align*}
where $\omega_\nu$ and $\hat{\vec{q}}_\nu$ are the Gauss quadrature weights and points used in evaluating volume integrals in ($\mathrm{H}$).
The discrete $L^2_h$ errors for momentum and total energy are measured similarly.
If $\mathtt{err}_{\Delta x}$ denotes the error on a mesh with resolution $\Delta x$, then the rate is given by $\ln(\mathtt{err}_{\Delta x}/\mathtt{err}_{\Delta x/2})/\ln{2}$.
When the time step size is sufficiently small, such that the spatial error dominates, we observe second order convergence for $\IQ^1$ and $\IQ^2$ schemes and fourth order convergence for $\IQ^3$ scheme, see Table~\ref{tab:2D_convergence}. 
For odd-order spaces, we obtain the optimal order of convergence.
Since the NIPG method is suboptimal in even-order spaces, a second order convergence for $\IQ^2$ scheme is as expected. 
\begin{table}[htbp!]
\centering
\begin{tabularx}{\linewidth}{@{~}c@{~}|c@{~}|C@{~}|c@{~}|C@{~}|c@{~}|C@{~}|c@{~}}
\toprule
$k$ & $\Delta x=\Delta y$ & $\|\rho_h^{N_T}-\rho(T)\|_{L^2_h}$ & rate & $\|\vec{m}_h^{N_T}-\vec{m}(T)\|_{L^2_h}$ & rate & $\|E_h^{N_T}-E(T)\|_{L^2_h}$ & rate\\
\midrule
1 & $1/2^3$ & $6.397\cdot10^{-2}$ & ---   & $2.144\cdot10^{-1}$ & ---   & $4.392\cdot10^{-1}$ & ---   \\
~ & $1/2^4$ & $1.978\cdot10^{-2}$ & 1.693 & $5.297\cdot10^{-2}$ & 2.017 & $1.069\cdot10^{-1}$ & 2.039 \\
~ & $1/2^5$ & $5.194\cdot10^{-3}$ & 1.929 & $1.288\cdot10^{-2}$ & 2.040 & $2.729\cdot10^{-2}$ & 1.970 \\
\midrule
2 & $1/2^4$ & $9.257\cdot10^{-3}$ & ---   & $2.519\cdot10^{-2}$ & ---   & $4.538\cdot10^{-2}$ & ---   \\
~ & $1/2^5$ & $2.603\cdot10^{-3}$ & 1.830 & $7.005\cdot10^{-3}$ & 1.847 & $1.248\cdot10^{-2}$ & 1.863 \\
~ & $1/2^6$ & $6.847\cdot10^{-4}$ & 1.927 & $1.838\cdot10^{-3}$ & 1.930 & $3.327\cdot10^{-3}$ & 1.907 \\
\midrule
3 & $1/2^1$ & $1.100\cdot10^{-1}$ & ---   & $3.353\cdot10^{-1}$ & ---   & $5.739\cdot10^{-1}$ & ---   \\
~ & $1/2^2$ & $1.408\cdot10^{-2}$ & 2.996 & $3.645\cdot10^{-2}$ & 3.202 & $6.853\cdot10^{-2}$ & 3.066 \\
~ & $1/2^3$ & $9.518\cdot10^{-4}$ & 3.887 & $2.360\cdot10^{-3}$ & 3.949 & $4.663\cdot10^{-3}$ & 3.878 \\
\bottomrule
\end{tabularx}
\caption{Accuracy test: the $\IQ^k$ scheme using a very small time step for a smooth solution, where $k\in\{1,2,3\}$, errors and convergence rates for density, momentum, and total energy. }
\label{tab:2D_convergence}
\end{table}

\subsection{Lax shock tube problem}
The Lax shock tube problem a classical benchmark problem for gas dynamics equations. We choose the computational domain $\Omega = [-5,5]$ and set the simulation end time $T=1.3$. The initial condition is prescribed as follows:
\begin{align*}
\transpose{[\rho_0, u_0, p_0]} = 
\begin{cases}
\transpose{[0.445,\, 0.698,\, 3.528]} & \text{if}~~x\in[-5,0),\\
\transpose{[0.5,\, 0,\, 0.571]} & \text{if}~~x\in[0,5].
\end{cases}
\end{align*}
In addition, the Dirichlet boundary conditions $\transpose{[\rho, u, p]} = \transpose{[0.445,\, 0.698,\, 3.528]}$ on the left end of domain $\Omega$ and $\transpose{[\rho, u, p]} = \transpose{[0.5,\, 0,\, 0.571]}$ on the right end of domain $\Omega$ are supplemented.
\par
We uniformly partition domain $\Omega$ into $512$ cells. For this one-dimensional problem, the $\IQ^1$ scheme is considered. We take the parameter $a=0.125$ in \eqref{eq:hyperbolic_CFL} for adaptive time step size.
%The penalties of DG method for solving ($\mathrm{P}$) are: $\sigma = 2$ on $\Gammah$, $\sigma = 4$ on $\partial{\Omega}$, and $\tilde{\sigma} = 2$.
The Figure~\ref{fig:lax_shock_tub} shows simulation results of Reynolds number $\Rey=100$ and $\Rey=1000$.
The reference solution is generated by a second order finite difference scheme using a fifth order positivity-preserving WENO flux for $\vec{F}^\mathrm{a}$ with a second order approximation for diffusion on a mesh of $64000$ points \cite{zhang2017positivity}.
\begin{figure}[ht!]
\begin{center}
\begin{tabularx}{0.95\linewidth}{@{}c@{~~}c@{~}c@{~}c@{}}
\begin{sideways}{\footnotesize $\hspace{1.5cm} \Rey=100 \quad$}\end{sideways} &
\includegraphics[width=0.3\textwidth]{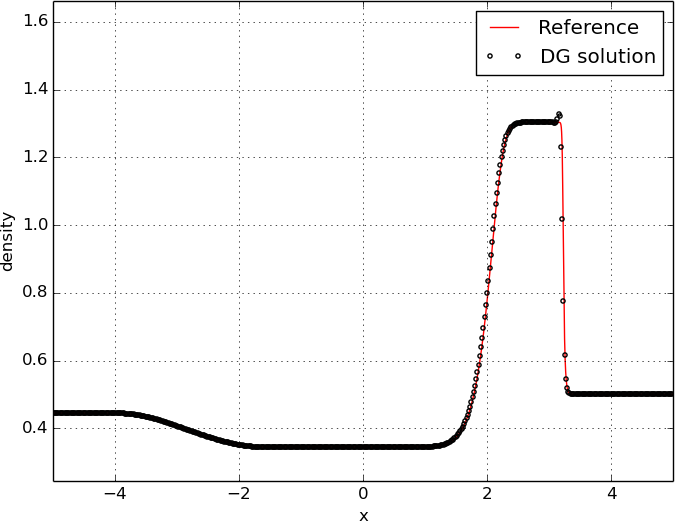} &
\includegraphics[width=0.3\textwidth]{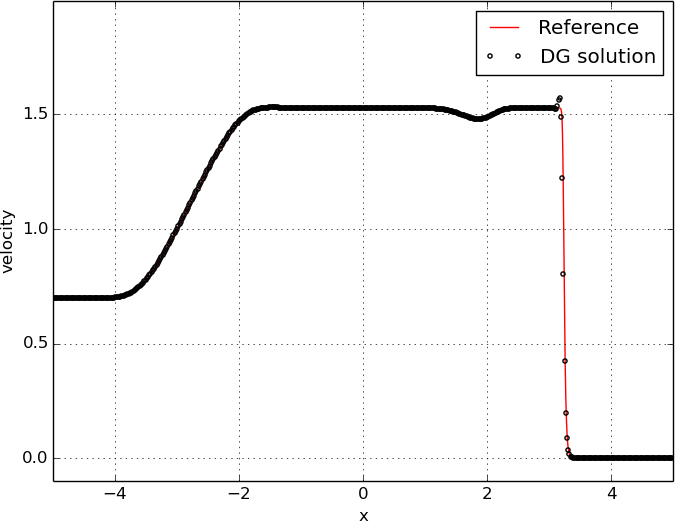} &
\includegraphics[width=0.3\textwidth]{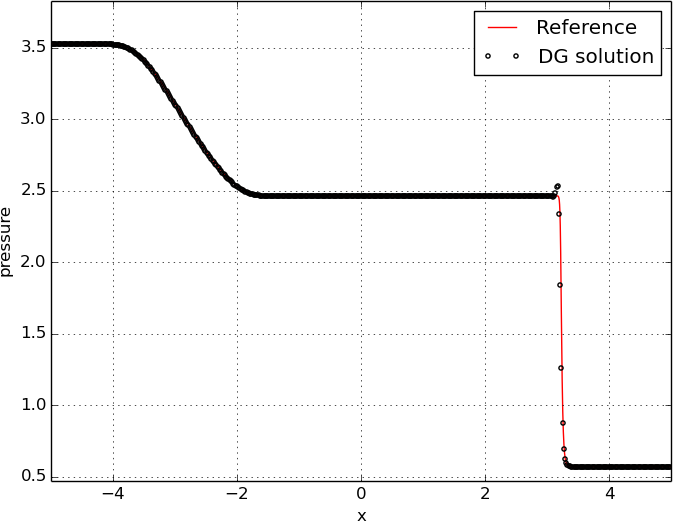} \\
%~ & density & velocity & pressure \\
\begin{sideways}{\footnotesize $\hspace{1.5cm} \Rey=1000 \quad$}\end{sideways} &
\includegraphics[width=0.3\textwidth]{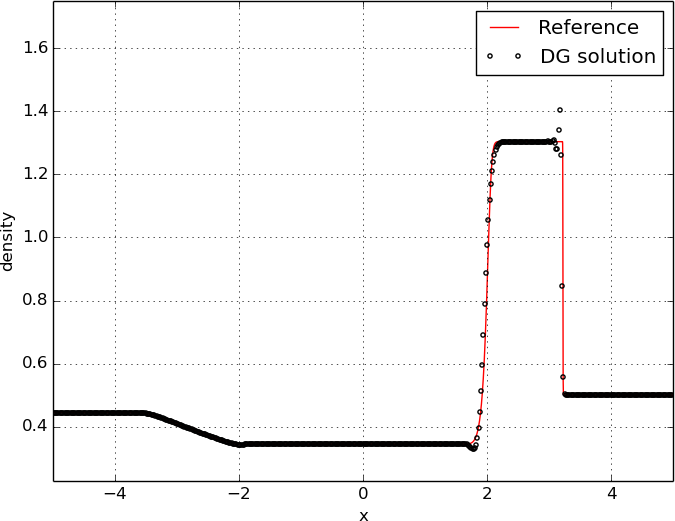} &
\includegraphics[width=0.3\textwidth]{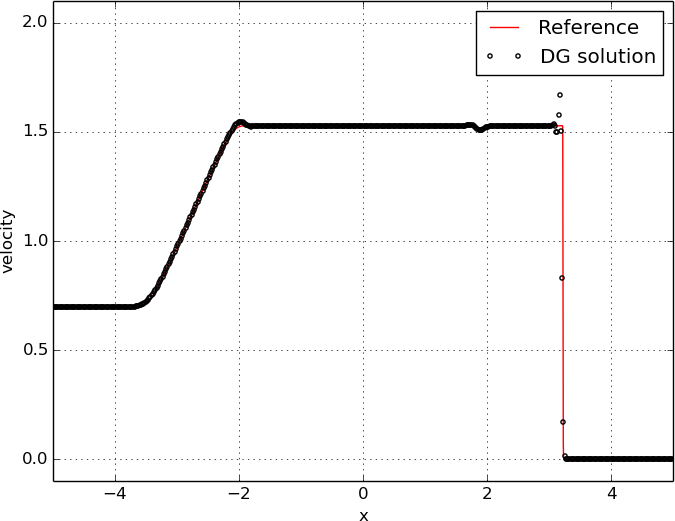} &
\includegraphics[width=0.3\textwidth]{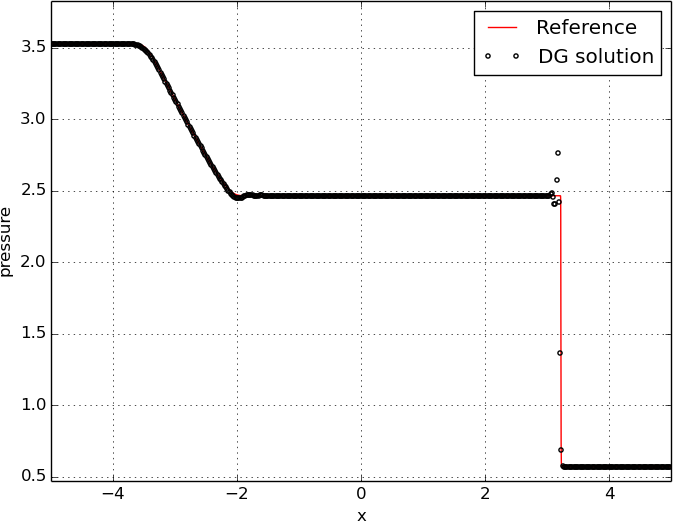} \\
~ & density & velocity & pressure \\
\end{tabularx}
\caption{Lax shock tube: the $\IQ^1$ scheme with only the positivity-preserving limiter on $512$ uniform cells. The snapshots are taken at $T=1.3$. Only cell averages are plotted.}
\label{fig:lax_shock_tub}
\end{center}
\end{figure}

\subsection{Double rarefaction}
This Riemann problem contains low density and low pressure. We choose the computational domain $\Omega = [-1,1]$ and set the simulation end time $T=0.6$. The initial condition is prescribed as follows:
\begin{align*}
\transpose{[\rho_0, u_0, p_0]} = 
\begin{cases}
\transpose{[7,\, -1,\, 0.2]} & \text{if}~~x\in[-1,0),\\
\transpose{[7,\, 1,\, 0.2]} & \text{if}~~x\in[0,1].
\end{cases}
\end{align*}
In addition, the Dirichlet boundary conditions $\transpose{[\rho, u, p]} = \transpose{[7,\, -1,\, 0.2]}$ on the left end of domain $\Omega$ and $\transpose{[\rho, u, p]} = \transpose{[7,\, 1,\, 0.2]}$ on the right end of domain $\Omega$ are supplemented.
\par
We uniformly partition domain $\Omega$ into $512$ cells. For this one-dimensional problem, the $\IQ^1$ scheme is considered. We take the parameter $a=0.125$ in \eqref{eq:hyperbolic_CFL} for adaptive time step size.
%The penalties of DG method for solving ($\mathrm{P}$) are: $\sigma = 2$ on $\Gammah$, $\sigma = 4$ on $\partial{\Omega}$, and $\tilde{\sigma} = 2$.
The Figure~\ref{fig:double_rare} shows simulation results of Reynolds number $\Rey=1000$.
The reference solution is generated by a second order finite difference scheme on a mesh of $32000$ points \cite{zhang2017positivity}.
\begin{figure}[ht!]
\begin{center}
\begin{tabularx}{0.95\linewidth}{@{}c@{~~}c@{~~}c@{}}
\includegraphics[width=0.285\textwidth]{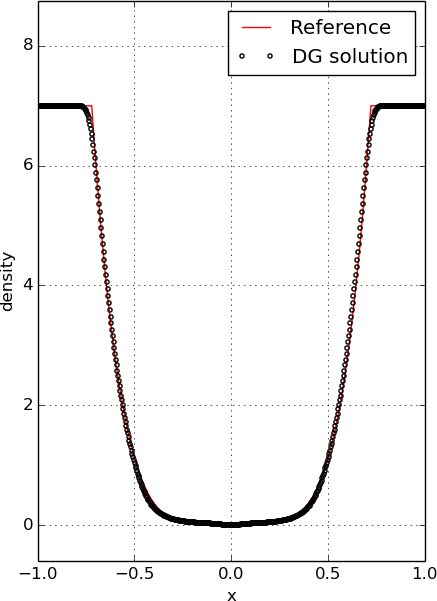} &
\includegraphics[width=0.3\textwidth]{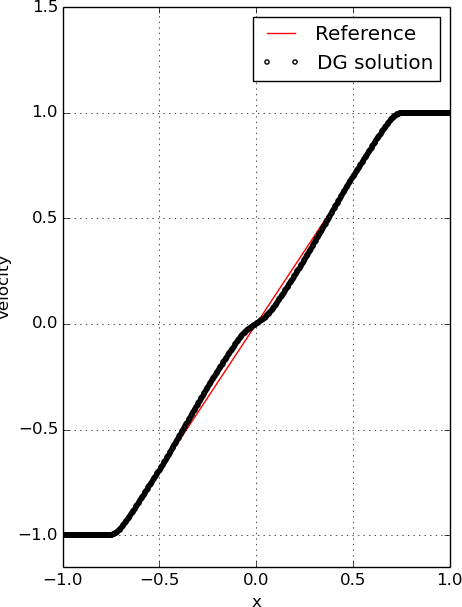}&
\includegraphics[width=0.3\textwidth]{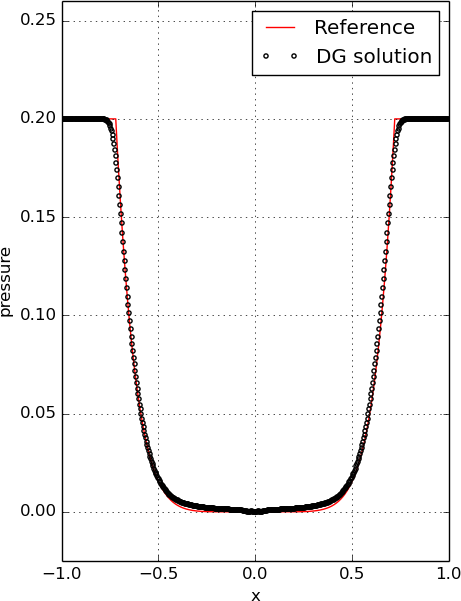}\\
density & velocity & pressure \\
\end{tabularx}
\caption{Double rarefaction: the $\IQ^1$ scheme with only the positivity-preserving limiter on $512$ uniform cells. The snapshots are taken at $T=0.6$. Only cell averages are plotted.}
\label{fig:double_rare}
\end{center}
\end{figure}

%%%%%%%%%%%%%%%%%%%%%%%%%%%%%%%%%%%%%%%%%%%%%%%%%%%%%%%%%%%%%%%%%%%%%%%%%%%%%%%%%%%%%%%%%%%%%%%%%%%%%%%%%%%%%%%%%%%%%%%
\subsection{Sedov blast wave}
%%%%%%%%%%%%%%%%%%%%%%%%%%%%%%%%%%%%%%%%%%%%%%%%%%%%%%%%%%%%%%%%%%%%%%%%%%%%%%%%%%%%%%%%%%%%%%%%%%%%%%%%%%%%%%%%%%%%%%%
The Sedov blast wave involves low density, low pressure, and a strong shock, which is of great utility as a verification test for a positivity-preserving scheme.
\par
We choose the computational domain $\Omega = [0,1.1]^2$ and set the simulation end time $T = 1$. We uniformly partition domain $\Omega$ by square cells with mesh resolution $\Delta x = 1.1/320$. 
The initials are prescribed as piecewise constants: density $\rho_0 = 1$ and velocity $\vec{u}_0 = \vec{0}$, for all points in $\Omega$; the total energy $E_0$ equals to $10^{-12}$ everywhere except the cell at the lower left corner, where $0.244816/{\Delta x}^2$ is used.
The boundary conditions are as follows. In subproblem ($\mathrm{H}$), we utilize reflective boundary condition on the left and bottom edges. The outflow boundary condition is employed on the right and top edges. In subproblem ($\mathrm{P}$), we supplement Neumann-type boundary conditions for both velocity and internal energy.
\par
We take parameter $a=0.5$ in \eqref{eq:hyperbolic_CFL} for $\IQ^1$ scheme and $a=1$ in \eqref{eq:hyperbolic_CFL} for $\IQ^2$ and $\IQ^3$ schemes for adaptive time step size.
The Figure~\ref{fig:sedov} displays snapshots of the density field at time $T=1$ with Reynolds number $\Rey=200$ and $\Rey = 1000$. The results are comparable to those in literature, e.g., \cite{zhang2017positivity}.
\begin{figure}[ht!]
\begin{center}
\begin{tabularx}{\linewidth}{@{}c@{~~}c@{~}c@{~}c@{~}c@{}}
\begin{sideways}{\footnotesize $\hspace{1.85cm} \Rey=200 \quad$}\end{sideways} &
\includegraphics[width=0.3\textwidth]{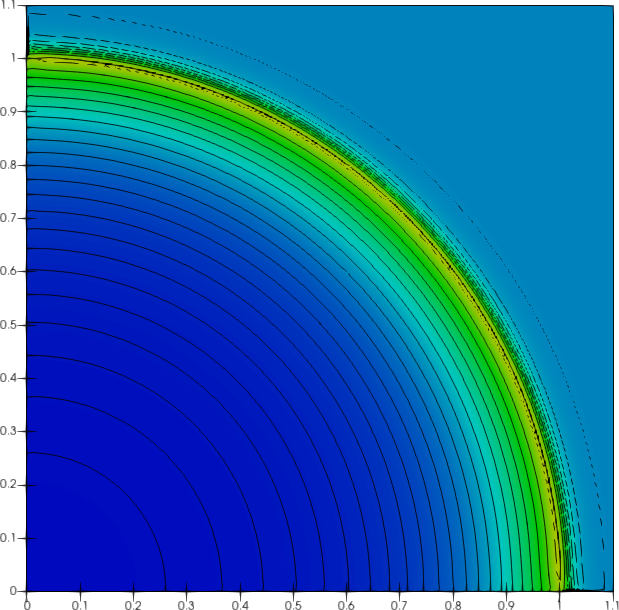} &
\includegraphics[width=0.3\textwidth]{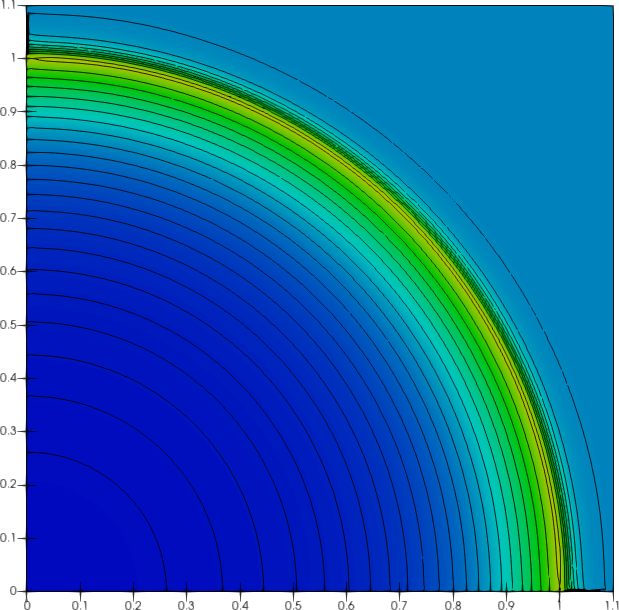} &
\includegraphics[width=0.3\textwidth]{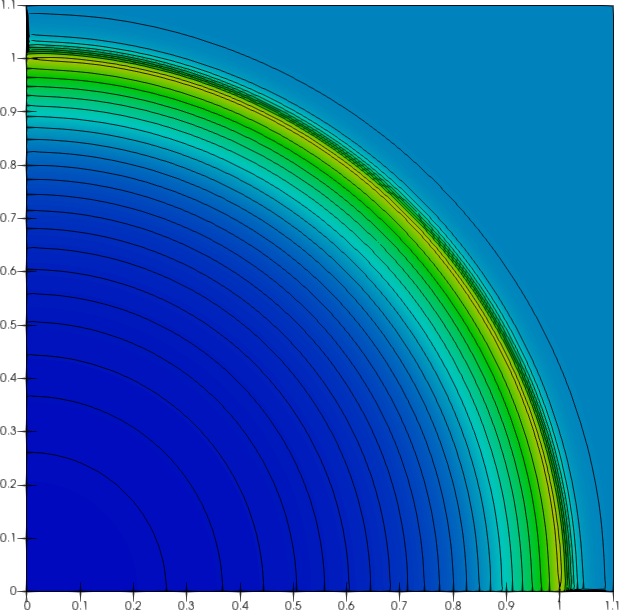} &
\includegraphics[width=0.045\textwidth]{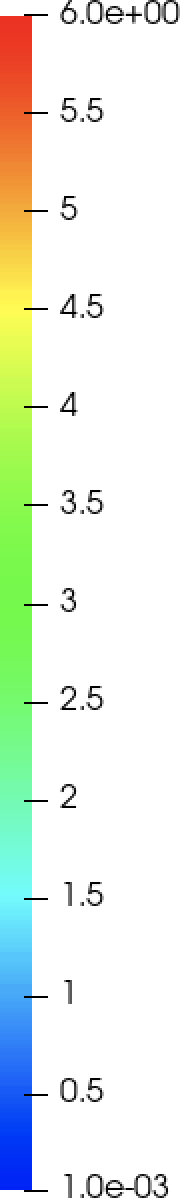} \\
\begin{sideways}{\footnotesize $\hspace{1.8cm} \Rey=1000 \quad$}\end{sideways} &
\includegraphics[width=0.3\textwidth]{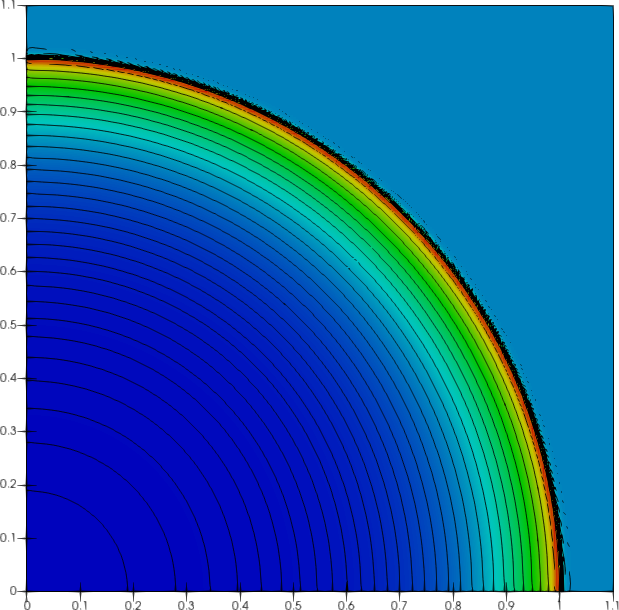} &
\includegraphics[width=0.3\textwidth]{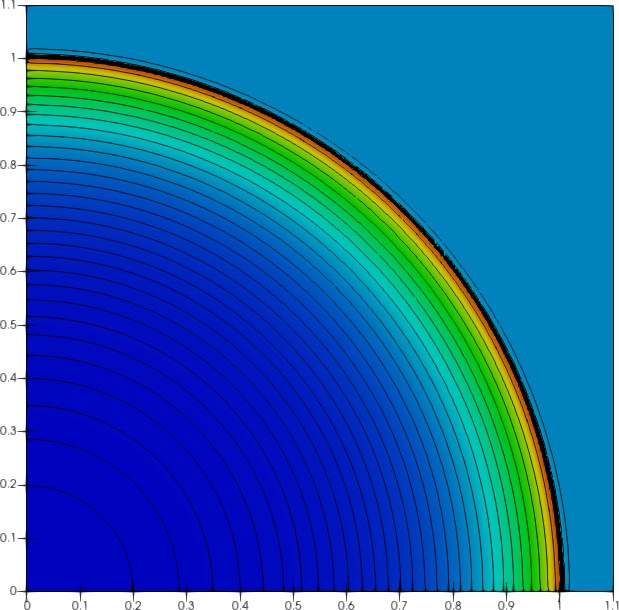} &
\includegraphics[width=0.3\textwidth]{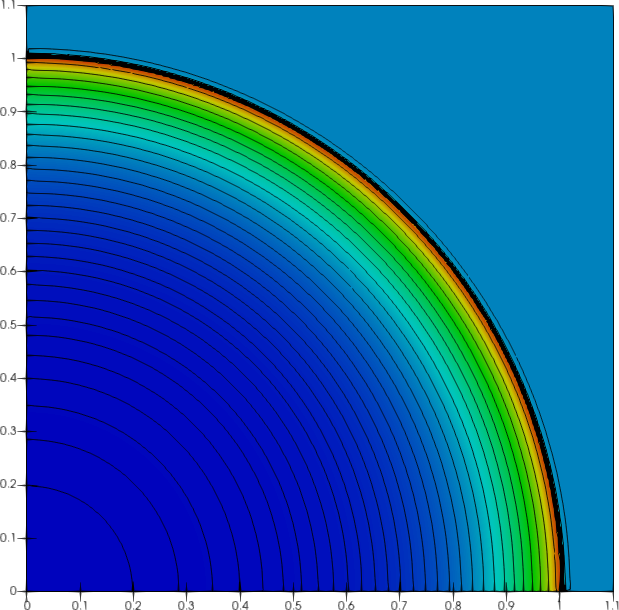} &
\includegraphics[width=0.045\textwidth]{{Figures/sedov/color_bar}.png}\\
\end{tabularx}
\caption{2D Sedov blast wave. From left to right: the $\IQ^1$, $\IQ^2$, and $\IQ^3$ schemes with only the positivity-preserving limiter on a $320\times 320$ uniform mesh. The snapshots of density profile are taken at $T=1$. Plot of density: $50$ exponentially distributed contour lines of density from $0.001$ to $6$.}
\label{fig:sedov}
\end{center}
\end{figure}

%%%%%%%%%%%%%%%%%%%%%%%%%%%%%%%%%%%%%%%%%%%%%%%%%%%%%%%%%%%%%%%%%%%%%%%%%%%%%%%%%%%%%%%%%%%%%%%%%%%%%%%%%%%%%%%%%%%%%%%
\subsection{Shock diffraction}
%%%%%%%%%%%%%%%%%%%%%%%%%%%%%%%%%%%%%%%%%%%%%%%%%%%%%%%%%%%%%%%%%%%%%%%%%%%%%%%%%%%%%%%%%%%%%%%%%%%%%%%%%%%%%%%%%%%%%%%
Let the computational domain $\Omega$ be the union of $[0,1]\times[6,12]$ and $[1,13]\times[0,12]$. We select the simulation end time $T = 2.3$.
The initial condition is a pure right-moving shock of Mach number $5.09$, initially located at $\{x=0.5, 6\leq y\leq 12\}$, moving into undisturbed air ahead of the shock with a density of $1.4$ and a pressure of $1$.
For the hyperbolic subproblem, the left boundary of $\Omega$ is inflow, the right and bottom boundaries of $\Omega$ are outflow, the fluid--solid boundaries $\{y=6, 0\leq x\leq 1\}$ and $\{x=1, 0\leq y\leq 6\}$ are reflective, and the flow values on top boundary are set to describe the exact motion of the Mach $5.09$ shock. %In subproblem ($\mathrm{P}$), associated with the inflow and following shock moving boundaries in hyperbolic subproblem, we supplement Dirichlet boundary conditions. For the rest boundaries, we apply Neumann-type boundary conditions.
\par
We uniformly partition $\Omega$ by square cells with mesh resolution $\Delta x = 1/96$ for $\IQ^1$ scheme and $\Delta x = 1/64$ for $\IQ^2$ and $\IQ^3$ schemes, respectively. 
We take parameter $a=0.5$ in \eqref{eq:hyperbolic_CFL} for $\IQ^1$ scheme and $a=1$ in \eqref{eq:hyperbolic_CFL} for $\IQ^2$ and $\IQ^3$ schemes for adaptive time step size.
The diffraction of high-speed shocks at a sharp corner generates low density and low pressure. We compare two groups of simulations with Reynolds number $\Rey=200$ and $\Rey=1000$. See Figure~\ref{fig:shock_diffraction} for a snapshots of the density field at time $T=2.3$. We only employ the positivity-preserving limiter. No special treatment is taken at the corner.
\begin{figure}[ht!]
\begin{center}
\begin{tabularx}{\linewidth}{@{}c@{~~}c@{~}c@{~}c@{~}c@{}}
\begin{sideways}{\footnotesize $\hspace{1.5cm} \Rey=200 \quad$}\end{sideways} &
\includegraphics[width=0.29\textwidth]{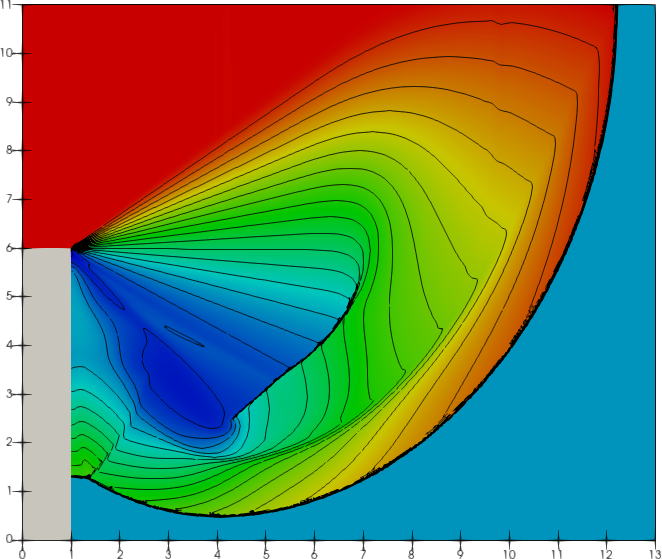} & 
\includegraphics[width=0.29\textwidth]{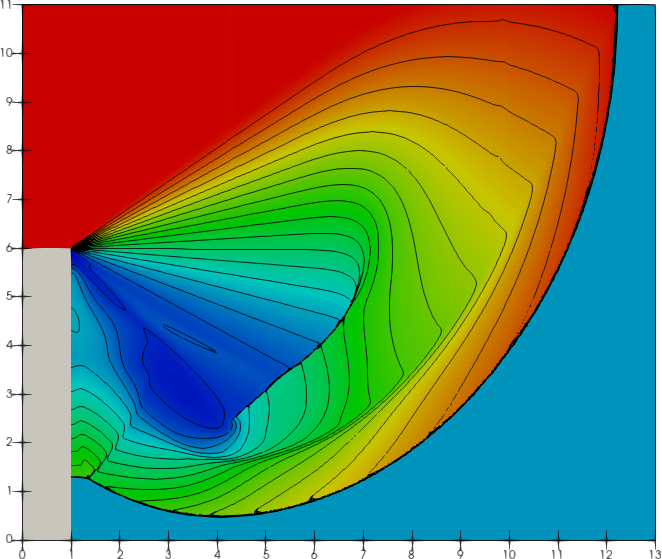} & \includegraphics[width=0.29\textwidth]{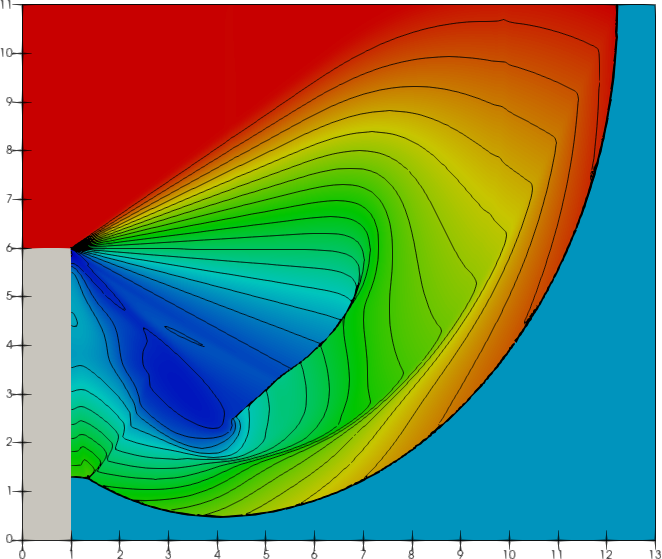} &
\includegraphics[width=0.0775\textwidth]{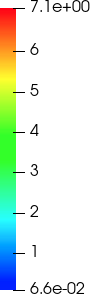} \\
\begin{sideways}{\footnotesize $\hspace{1.5cm} \Rey=1000 \quad$}\end{sideways} &
\includegraphics[width=0.29\textwidth]{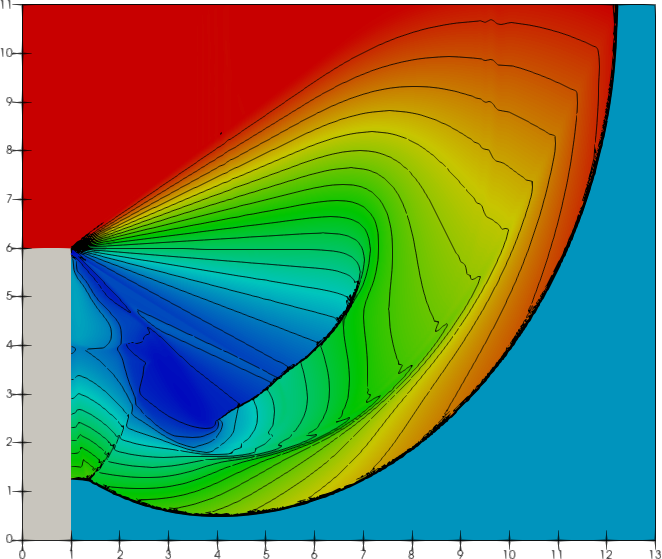} &
\includegraphics[width=0.29\textwidth]{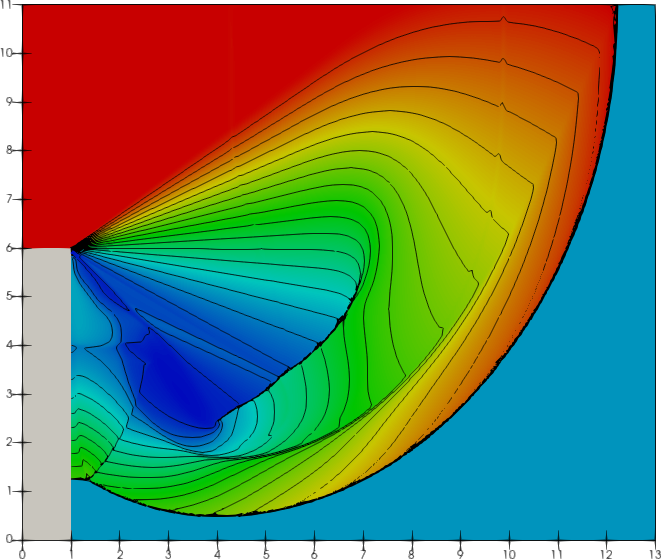} &
\includegraphics[width=0.29\textwidth]{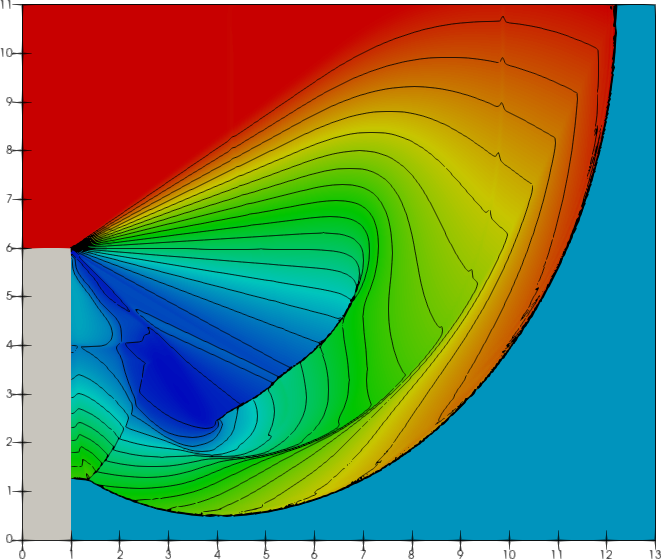} &
\includegraphics[width=0.0775\textwidth]{{Figures/shock_diffraction/color_bar}.png} \\
\end{tabularx}
\caption{Shock diffraction: the $\IQ^1$, $\IQ^2$, and $\IQ^3$ schemes with only the positivity-preserving limiter on a uniform mesh with resolution $\Delta x = 1/96$ for $\IQ^1$ scheme and $\Delta x = 1/64$ for $\IQ^2$ and $\IQ^3$ schemes. The snapshots of density profile are taken at $T=2.3$. Plot of density: $20$ equally space contour lines from $0.066227$ to $7.0668$.}
\label{fig:shock_diffraction}
\end{center}
\end{figure}

%%%%%%%%%%%%%%%%%%%%%%%%%%%%%%%%%%%%%%%%%%%%%%%%%%%%%%%%%%%%%%%%%%%%%%%%%%%%%%%%%%%%%%%%%%%%%%%%%%%%%%%%%%%%%%%%%%%%%%%
\subsection{Double Mach reflection of a Mach $10$ shock}
%%%%%%%%%%%%%%%%%%%%%%%%%%%%%%%%%%%%%%%%%%%%%%%%%%%%%%%%%%%%%%%%%%%%%%%%%%%%%%%%%%%%%%%%%%%%%%%%%%%%%%%%%%%%%%%%%%%%%%%
The double Mach reflection of a Mach $10$ shock is a widely used benchmark test problem \cite{woodward1984numerical}. This experiment studies a planar shock flow in a tube, which contains an oblique wall of thirty degree. In the beginning, the planar shock is perpendicular to the tube surface and move to right. Later, when the shock meets the oblique wall a complicated shock reflection occurs.
Following the numerical setup in \cite{cockburn1998runge}, we tilt the incident shock rather than the solid surface and select the computational domain $\Omega = [0,4]\times[0,1]$. We set the simulation end time $T = 0.2$.
\par
A Mach $10$ shock initially is positioned at point $(\frac{1}{6},0)$ and makes a sixty degree angle with $x$-axis. The line $6x-2\sqrt{3}y-1=0$ denotes the shock location and separates domain $\Omega$ into left and right zones.
For initials, the density equals to $8$, the velocity equals to $\transpose{[4.125\sqrt{3},-4.125]}$, and the pressure equals to $116.5$ in the post-shock region (left zone). 
And the undisturbed air ahead of the shock (right zone) has a density of $1.4$ and a pressure of $1$.
For the hyperbolic subproblem, the left boundary of $\Omega$ is inflow, the right boundary of $\Omega$ is outflow, part of the bottom boundary of $\Omega$ on $\{y=0, \frac{1}{6}\leq x\leq 4\}$ are reflective, and the post-shock condition is imposed at $\{y=0, 0\leq x< \frac{1}{6}\}$. On the boundary with post-shock condition, the density, velocity, and pressure are fixed in time with the initial values to make the reflected shock stick to the bottom wall. 
The flow values on top boundary are set to describe the exact motion of the Mach $10$ shock.
%For the parabolic subproblem, associated with the inflow, post-shock, and following shock moving boundaries in subproblem ($\mathrm{H}$), we supplement Dirichlet boundary conditions. For the rest boundaries, we apply Neumann-type boundary conditions.
\par
We uniformly partition $\Omega$ by square cells with the mesh resolution $\Delta x = 1/480$ for $\IQ^1$ scheme and $\Delta x = 1/240$ for $\IQ^2$ and $\IQ^3$ schemes. 
We take parameter $a=0.5$ in \eqref{eq:hyperbolic_CFL} for $\IQ^1$ scheme and $a=1$ in \eqref{eq:hyperbolic_CFL} for $\IQ^2$ and $\IQ^3$ schemes for adaptive time step size.
We compare two groups of simulations with Reynolds number $\Rey=100$ and $\Rey=1000$. The Figure~\ref{fig:shock_reflection_Re100} and Figure~\ref{fig:shock_reflection_Re1000} provide snapshots of the density fields at time $T=0.2$.
For high Reynolds number simulations, it is clear that the rollup is better-captured by the $\IQ^3$ scheme   than the $\IQ^1$ scheme, see Figure~\ref{fig:shock_reflection_Re1000}.
\begin{figure}[ht!]
\begin{center}
\begin{tabularx}{\linewidth}{@{}c@{~}c@{~}c@{}}
\includegraphics[width=0.55\textwidth]{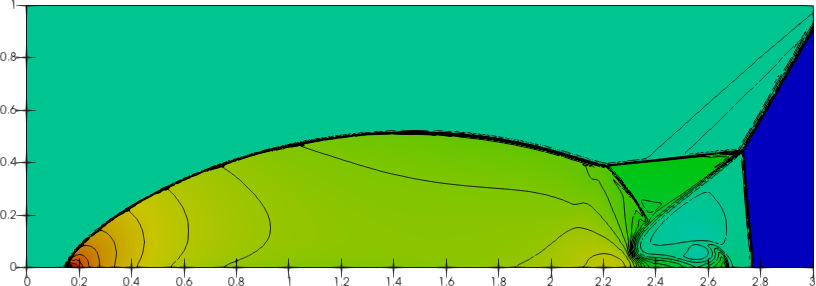} &
\includegraphics[width=0.375\textwidth]{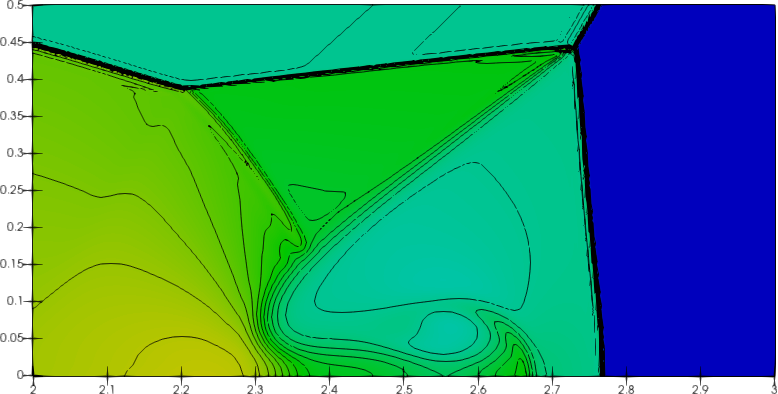} &
\includegraphics[width=0.061\textwidth]{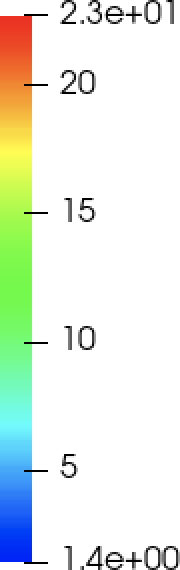} \\
\includegraphics[width=0.55\textwidth]{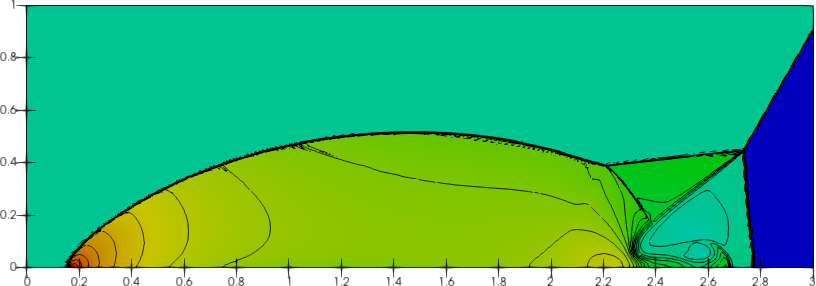} &
\includegraphics[width=0.375\textwidth]{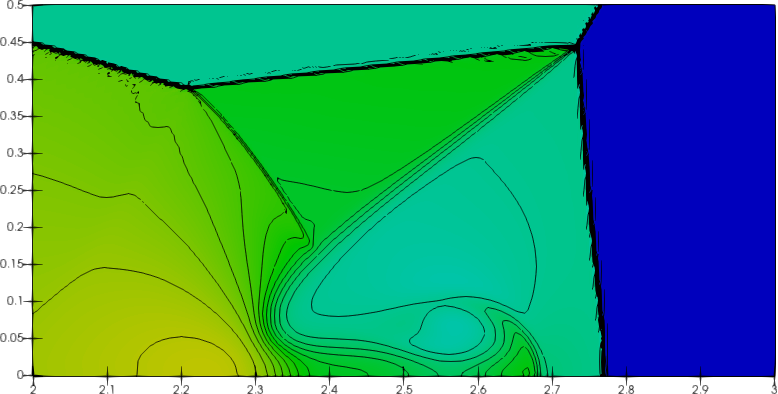} &
\includegraphics[width=0.061\textwidth]{{Figures/shock_reflection/color_bar}.png} \\
\includegraphics[width=0.55\textwidth]{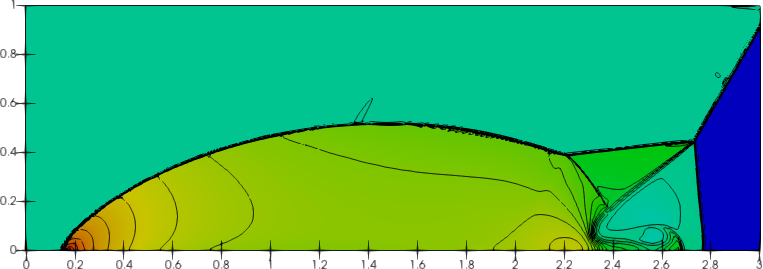} &
\includegraphics[width=0.375\textwidth]{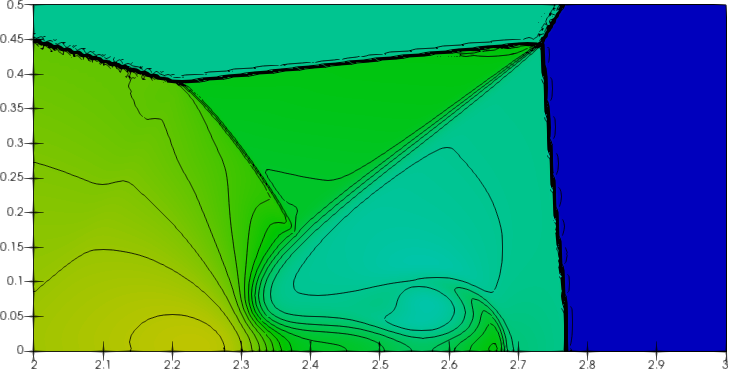} &
\includegraphics[width=0.061\textwidth]{{Figures/shock_reflection/color_bar}.png} \\
\end{tabularx}
\caption{Shock reflection. From top to bottom: simulation results of $\IQ^1$, $\IQ^2$, and $\IQ^3$ schemes for $\Rey=100$ with only the positivity-preserving limiter. The snapshots of density profile are taken at $T=0.2$. Plot of density: 30 equally space contour lines from $1.3965$ to $22.682$.}
\label{fig:shock_reflection_Re100}
\end{center}
\end{figure}
\begin{figure}[ht!]
\begin{center}
\begin{tabularx}{\linewidth}{@{}c@{~}c@{~}c@{}}
\includegraphics[width=0.55\textwidth]{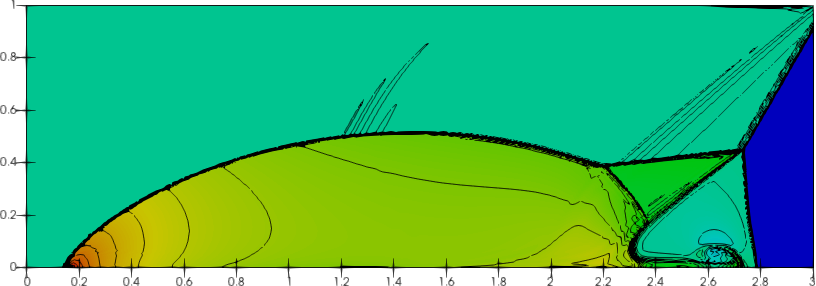} &
\includegraphics[width=0.375\textwidth]{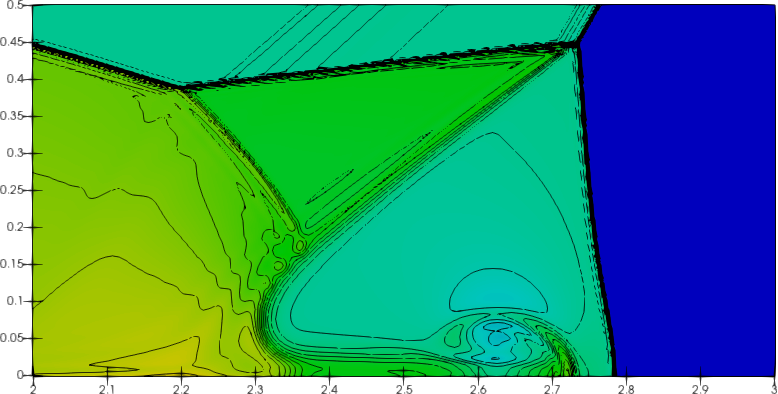} &
\includegraphics[width=0.061\textwidth]{{Figures/shock_reflection/color_bar}.png} \\
\includegraphics[width=0.55\textwidth]{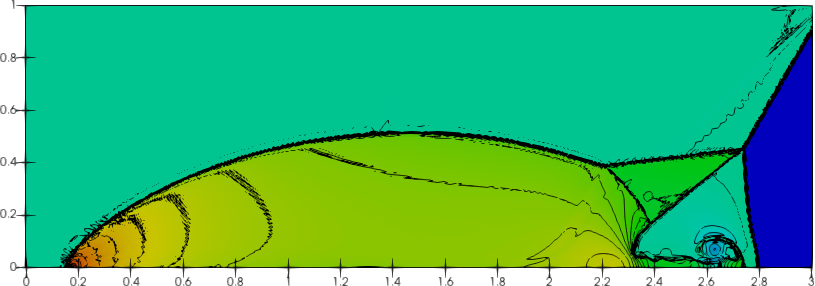} &
\includegraphics[width=0.375\textwidth]{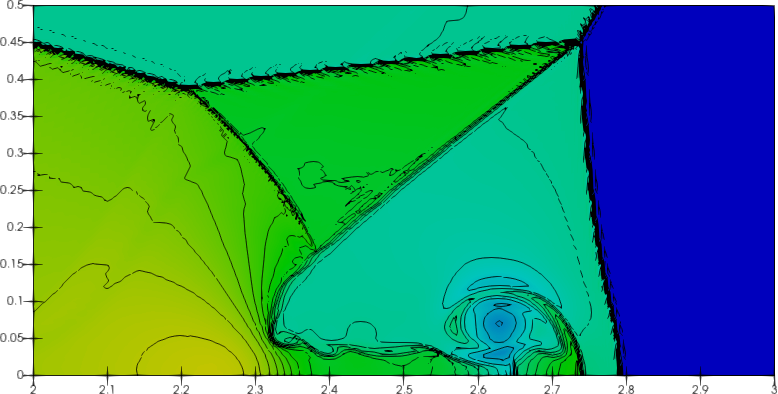} &
\includegraphics[width=0.061\textwidth]{{Figures/shock_reflection/color_bar}.png} \\
\includegraphics[width=0.55\textwidth]{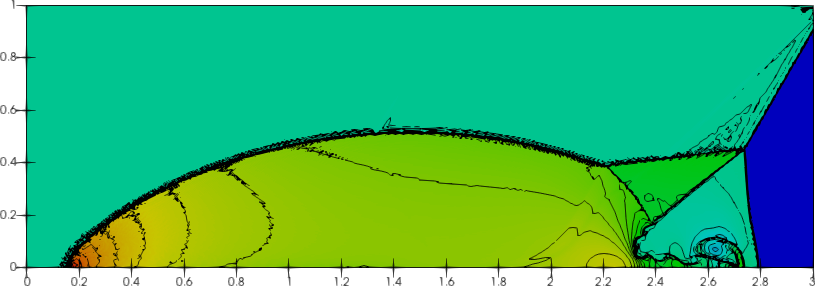} &
\includegraphics[width=0.375\textwidth]{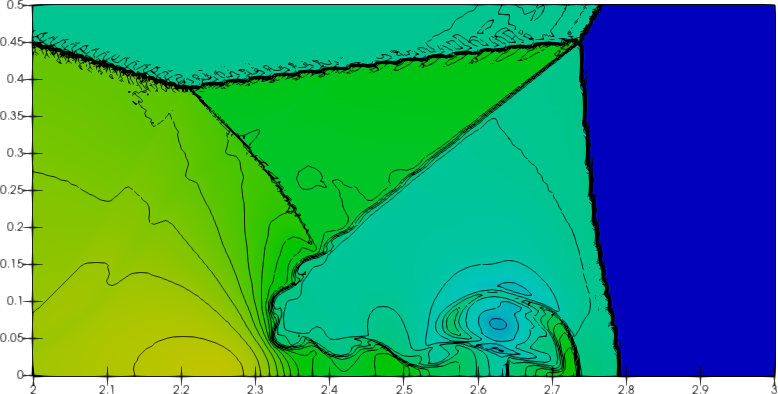} &
\includegraphics[width=0.061\textwidth]{{Figures/shock_reflection/color_bar}.png} \\
\end{tabularx}
\caption{Shock reflection. From top to bottom: simulation results of $\IQ^1$, $\IQ^2$, and $\IQ^3$ schemes for $\Rey=1000$ with only the positivity-preserving limiter. The snapshots of density profile are taken at $T=0.2$. Plot of density: 30 equally space contour lines from $1.3965$ to $22.682$.}
\label{fig:shock_reflection_Re1000}
\end{center}
\end{figure}

%%%%%%%%%%%%%%%%%%%%%%%%%%%%%%%%%%%%%%%%%%%%%%%%%%%%%%%%%%%%%%%%%%%%%%%%%%%%%%%%%%%%%%%%%%%%%%%%%%%%%%%%%%%%%%%%%%%%%%%
\subsection{Mach $10$ shock reflection and diffraction}\label{sec:experiments:ref_def}
%%%%%%%%%%%%%%%%%%%%%%%%%%%%%%%%%%%%%%%%%%%%%%%%%%%%%%%%%%%%%%%%%%%%%%%%%%%%%%%%%%%%%%%%%%%%%%%%%%%%%%%%%%%%%%%%%%%%%%%
This is the same test as in \cite{fan2022positivity}.
Let the computational domain $\Omega$ be the union of $[1,4]\times[-1,0]$ and $[0,4]\times[0,1]$. We select the simulation end time $T = 0.2$.
A Mach $10$ shock initially is positioned at point $(\frac{1}{6},0)$ and makes a sixty degree angle with $x$-axis. The line $6x-2\sqrt{3}y-1=0$ denotes the initial shock location and separates domain $\Omega$ into left zone and right zone.
For initials, the density equals to $8$, the velocity equals to $\transpose{[4.125\sqrt{3},-4.125]}$, and the pressure equals to $116.5$ in the post-shock region (left zone). 
And the undisturbed air ahead of the shock (right zone) has a density of $1.4$ and a pressure of $1$.
\par
For the hyperbolic subproblem, the left boundary of $\Omega$ is inflow, the right and bottom boundaries of $\Omega$ are outflow, part of the fluid--solid boundaries of $\Omega$ on $\{y=0, \frac{1}{6}\leq x\leq 1\}$ and $\{x=1, -1\leq y\leq 1\}$ are reflective, 
and the post-shock condition is imposed at $\{y=0, 0\leq x< \frac{1}{6}\}$. On the boundary with post-shock condition, the density, velocity, and pressure are fixed in time with the initial values to make the reflected shock stick to the solid wall. The flow values on top boundary are set to describe the exact motion of the Mach $10$ shock.
%For the parabolic subproblem, associated with the inflow, post-shock, and following shock moving boundaries in subproblem ($\mathrm{H}$), we supplement Dirichlet boundary conditions. For the rest boundaries, we apply Neumann-type boundary conditions.
\par
We take the parameter $a=0.5$ in \eqref{eq:hyperbolic_CFL} for $\IQ^1$ scheme and $a=1$ in \eqref{eq:hyperbolic_CFL} for $\IQ^2$ and $\IQ^3$ schemes for adaptive time step size. 
Consider three groups of numerical experiments. 
In the first group of tests, we choose $\IQ^1$ scheme and uniformly partition $\Omega$ by square cells with the mesh resolution $\Delta x = 1/480$. We various the Reynolds number in three different levels: $100$, $500$, and $1000$. From Figure~\ref{fig:shock_reflection_diffraction_1}, we see as the Reynolds number increases the rollup becomes stronger.
In the second group of tests, we fix the Reynolds number $\Rey=1000$ and compare the $\IQ^1$, $\IQ^2$, and $\IQ^3$ schemes with mesh resolution $\Delta x = 1/480$, $1/240$, and $1/120$. 
From Figure~\ref{fig:shock_reflection_diffraction_2}, we see even though the degrees of freedom for $\IQ^3$ simulation are significantly less than the $\IQ^1$ simulation, the rollup is well-captured in the $\IQ^3$ case.
% Q1 DOF (u - 12902400); Q2 DOF (u - 7257600); Q3 DOF (u - 3225600).
In the third group of tests, we take $\IQ^3$ scheme and compare simulation results under different mesh resolutions $\Delta x = 1/120$, $1/180$, $1/240$. From Figure~\ref{fig:shock_reflection_diffraction_3}, we see as mesh refinement, our scheme produces satisfactory non-oscillatory solutions when the physical diffusion is accurately resolved, which is consistent with the observations for fully explicit high order accurate schemes in \cite{zhang2017positivity}.
\begin{figure}[ht!]
\begin{center}
\begin{tabularx}{\linewidth}{@{}c@{~}c@{~}c@{~}c@{}}
\includegraphics[width=0.31\textwidth]{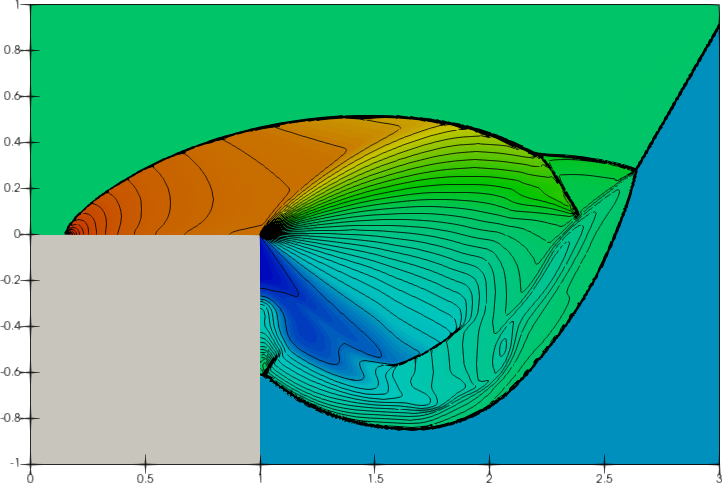} &
\includegraphics[width=0.31\textwidth]{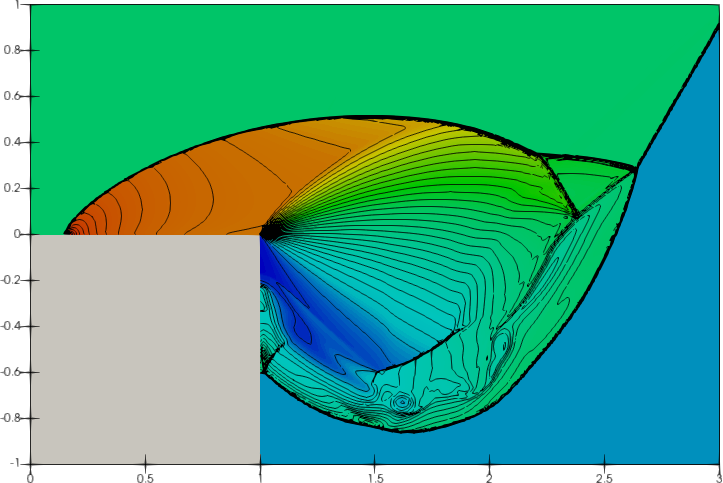} &
\includegraphics[width=0.31\textwidth]{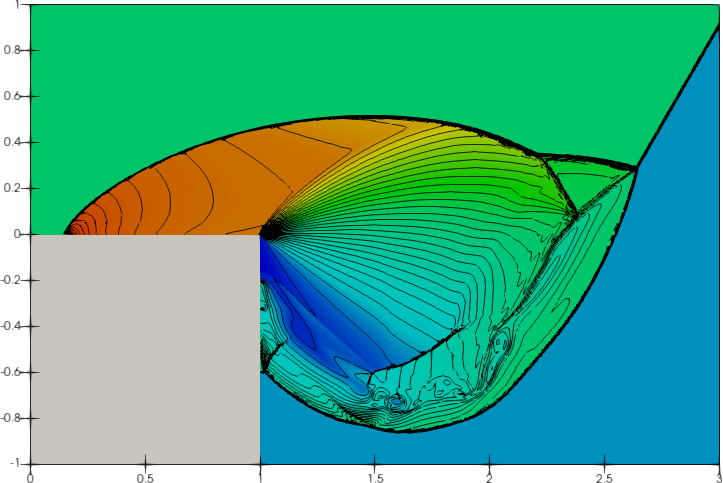} &
\includegraphics[width=0.041\textwidth]{{Figures/shock_reflection_diffraction/color_bar}.png} \\
\end{tabularx}
\caption{Mach $10$ shock reflection and diffraction. The snapshots of density profile are taken at $T=0.2$. Plot of density: 50 equally space contour lines from $0$ to $25$. From left to right: simulation results of Reynolds number $\Rey=100$, $500$, and $1000$ with mesh resolution $\Delta x = 1/480$.}
\label{fig:shock_reflection_diffraction_1}
\end{center}
\end{figure}
\begin{figure}[ht!]
\begin{center}
\begin{tabularx}{\linewidth}{@{}c@{~}c@{~}c@{~}c@{}}
\includegraphics[width=0.31\textwidth]{{Figures/shock_reflection_diffraction/Q1_Re1000_DG_h480_density_contour}.png} &
\includegraphics[width=0.31\textwidth]{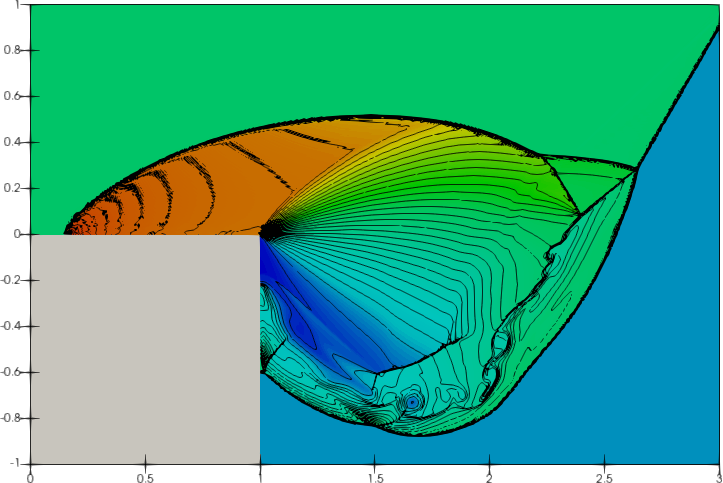} &
\includegraphics[width=0.31\textwidth]{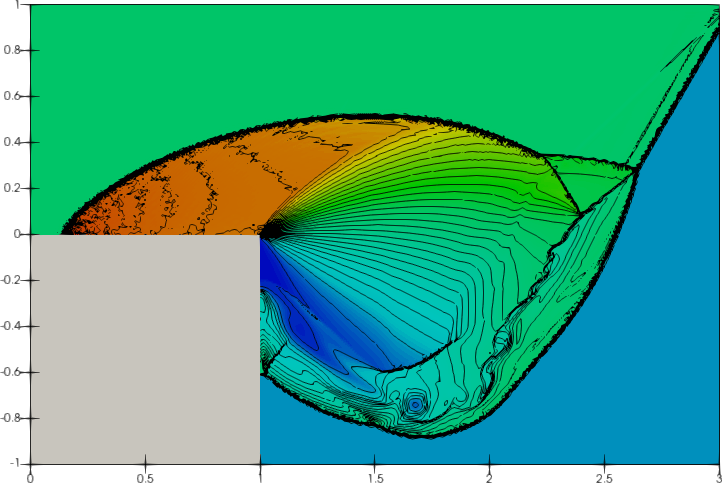} &
\includegraphics[width=0.041\textwidth]{{Figures/shock_reflection_diffraction/color_bar}.png} \\
\end{tabularx}
\caption{Mach $10$ shock reflection and diffraction. The snapshots of density profile are taken at $T=0.2$. Plot of density: 50 equally space contour lines from $0$ to $25$. From left to right: simulation results of $\IQ^1$, $\IQ^2$, and $\IQ^3$ schemes with mesh resolution $\Delta x = 1/480$, $1/240$, and $1/120$.}
\label{fig:shock_reflection_diffraction_2}
\end{center}
\end{figure}
\begin{figure}[ht!]
\begin{center}
\begin{tabularx}{\linewidth}{@{}c@{~}c@{~}c@{}}
\includegraphics[width=0.31\textwidth]{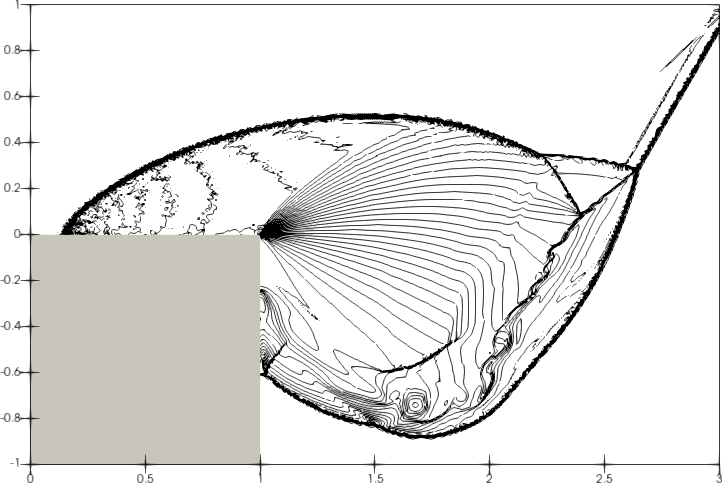} &
\includegraphics[width=0.31\textwidth]{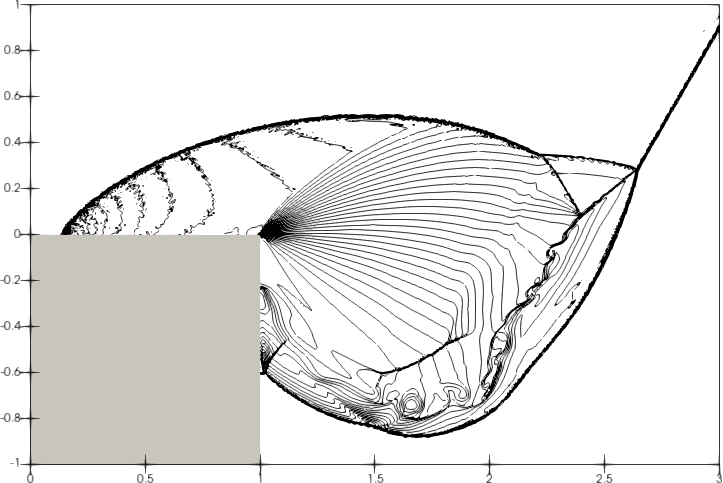} &
\includegraphics[width=0.31\textwidth]{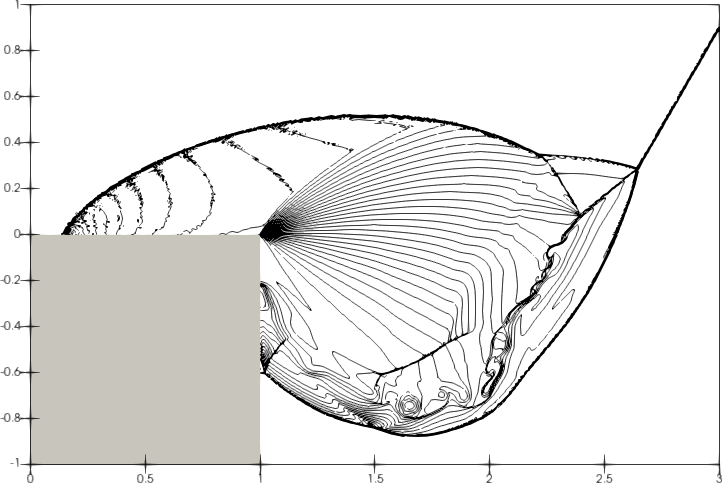} \\
\end{tabularx}
\caption{Mach $10$ shock reflection and diffraction. The snapshots of density profile are taken at $T=0.2$. Plot of density: 50 equally space contour lines from $0$ to $25$. Only contour lines are plotted. From left to right: simulation results of $\IQ^3$ scheme with mesh resolution $\Delta x = 1/120$, $1/180$, and $1/240$.}
\label{fig:shock_reflection_diffraction_3}
\end{center}
\end{figure}

%% file: Content/appendix.tex
%%%%%%%%%%%%%%%%%%%%%%%%%%%%%%%%%%%%%%%%%%%%%%%%%%%%%%%%%%%%%%%%%%%%%%%%%%%%%%%%%%%%%%%%%%%%%%%%%%%%%%%%%%%%%%%%%%%%%%%
\section{The M-matrix structure of the $\IQ^1$ DG scheme for parabolic subproblem}\label{sec:positivity:parabolic} 
%%%%%%%%%%%%%%%%%%%%%%%%%%%%%%%%%%%%%%%%%%%%%%%%%%%%%%%%%%%%%%%%%%%%%%%%%%%%%%%%%%%%%%%%%%%%%%%%%%%%%%%%%%%%%%%%%%%%%%%
 
The non-singular M-matrix is an inverse-positive matrix, which serves as a convenient tool for proving the positivity of internal energy. There are many equivalent definitions or characterizations of M-matrix. A comprehensive review of M-matrix can be found in \cite{plemmons1977m}. Here, we state a sufficient but not necessary condition to verify the nonsingular M-matrix.
\begin{lemma}\label{thm:M_matrix2}
For a real square matrix $\vecc{A}$ with positive diagonal entries and nonpositive off-diagonal entries, it is a nonsingular M-matrix if all the row sums of $\vecc{A}$ are nonnegative and at least one row sum is positive.
\end{lemma}

\paragraph{\bf One-dimensional case}
Assume the computational domain $\Omega=[-L,L]$, where $L>0$, is uniformly partitioned into $\Nel$ intervals (cells) with spacing $\Delta{x}$. Let $-L = x_0 < x_1 < \cdots < x_{\Nel} = L$ denote the grid points. 
On cell $K_i = [x_{i}, x_{i+1}]$, where $i = 0, \cdots, \Nel-1$, the piecewise linear bases are defined as follows:
$\varphi_{i0}(x) = \frac{1}{\Delta{x}}(x_{i+1}-x)$ and 
$\varphi_{i1}(x) = \frac{1}{\Delta{x}}(x-x_{i})$. And if $x\notin K_i$, the $\varphi_{i0}$ and $\varphi_{i1}$ equal to $0$. 
\par
In one dimension, the matrix from IIPG discretization of the Laplace operator evaluated by 2-point Gauss--Lobatto quadrature enjoys an M-matrix structure. This result is well-known in literature, for instance, see \cite{horvath2013discrete}. 
Let us present the matrix $\vecc{A}_\mathcal{D}$ explicitly.
For simplicity, we only show $\vecc{A}_\mathcal{D}$ with respect to pure Neumann boundary condition. Enforcing part or entire Dirichlet boundary does not break the M-matrix structure. 
\begin{align*}
\vecc{A}_\mathcal{D} = 
\begin{pmatrix}
\red{\frac{1}{\Delta{x}}} & -\frac{1}{\Delta{x}} & 0 & 0 &  &  &  &  &  & \\ 
-\frac{1}{2\Delta{x}} & \red{\frac{1+2\tilde{\sigma}}{2\Delta{x}}} & \frac{1-2\tilde{\sigma}}{2\Delta{x}} & -\frac{1}{2\Delta{x}} &  &  &  &  &  & \\ 
 & \ddots & \ddots & \ddots & \ddots &  &  &  &  & \\ 
 &  & -\frac{1}{2\Delta{x}} & \frac{1-2\tilde{\sigma}}{2\Delta{x}} & \red{\frac{1+2\tilde{\sigma}}{2\Delta{x}}} & -\frac{1}{2\Delta{x}} &  &  &  & \\ 
 &  &  &  & -\frac{1}{2\Delta{x}} & \red{\frac{1+2\tilde{\sigma}}{2\Delta{x}}} & \frac{1-2\tilde{\sigma}}{2\Delta{x}} & -\frac{1}{2\Delta{x}} &  & \\ 
 &  &  &  &  & \ddots & \ddots & \ddots & \ddots & \\ 
 &  &  &  &  &  & -\frac{1}{2\Delta{x}} & \frac{1-2\tilde{\sigma}}{2\Delta{x}} & \red{\frac{1+2\tilde{\sigma}}{2\Delta{x}}} & -\frac{1}{2\Delta{x}} \\ 
 &  &  &  &  &  & 0 & 0 & -\frac{1}{\Delta{x}} & \red{\frac{1}{\Delta{x}}}
\end{pmatrix}.
\end{align*}
In above, we mark all diagonal entries in red color.
Obviously, when the penalty parameter $\tilde{\sigma} > 1/2$, the diagonal entries of $\vecc{A}_\mathcal{D}$ are positive. All the off-diagonal entries of $\vecc{A}_\mathcal{D}$ are non-positive. The row sum of $\vecc{A}_\mathcal{D}$ equals zero. %This is easy to understand since the Laplace operator apply on constant is $0$. 
In addition, since the Lagrange bases are numerically orthogonal with respect to the Gauss--Lobatto quadrature, the mass matrix is diagonal with positive diagonal entries $[\vecc{A}_\mathcal{M}]_{ij;ij} = \Delta{x} \hat{\omega}_j \rho_{ij}^\mathrm{P}$. Thus the row sum of matrix $\vecc{A}_{\mathcal{M}} + \frac{\Delta t \lambda}{\Rey}\vecc{A}_{\mathcal{D}}$ is positive.
Above all, by Lemma~\ref{thm:M_matrix2}, the system matrix $\vecc{A}_{\mathcal{M}} + \frac{\Delta t \lambda}{\Rey}\vecc{A}_{\mathcal{D}}$ is a non-singular M-matrix, therefore is monotone.

\paragraph{\bf Two-dimensional case} 
In this part, we show the matrix corresponds to the IIPG discretization of $-\laplace{e}$ with $2^2$-point Gauss--Lobatto quadrature enjoys the M-matrix structure.
To the best knowledge of the authors, this is the first time that an M-matrix structure is reported with respect to IPDG method for the Laplace operator in two dimension. 
\par
Consider the computational domain $\Omega$ is uniformly partitioned into $\Nel$ square cells with side length $\Delta{x}$.
The $\IQ^1$ Lagrange bases on reference element $\hat{K} = [-\frac{1}{2},\frac{1}{2}]^2$ are defined as follows: for $\hat{\vec{x}} = \transpose{[\hat{x}, \hat{y}]} \in \hat{K}$,
\begin{align*}
\hat{\varphi}_0(\hat{\vec{x}}) &= (\frac{1}{2} - \hat{x})(\frac{1}{2} - \hat{y}), &
\hat{\varphi}_1(\hat{\vec{x}}) &= (\frac{1}{2} + \hat{x})(\frac{1}{2} - \hat{y}), \\
\hat{\varphi}_2(\hat{\vec{x}}) &= (\frac{1}{2} - \hat{x})(\frac{1}{2} + \hat{y}), &
\hat{\varphi}_3(\hat{\vec{x}}) &= (\frac{1}{2} + \hat{x})(\frac{1}{2} + \hat{y}).
\end{align*}
Denote the lower left corner of a cell $K_i\in\setE_h$ by $\vec{a}_{i0}$. 
The mapping $\vec{F}_i:~ \hat{K}\rightarrow K_i$ and its inverse $\vec{F}_i^{-1}:~ K_i\rightarrow \hat{K}$ are defined by
\begin{align*}
\vec{F}_i(\hat{\vec{x}}) = \Delta{x}\Big(\hat{\vec{x}} + \frac{1}{2}
\begin{bmatrix}
1 \\ 
1
\end{bmatrix}\Big) + \vec{a}_{i0}
\quad\text{and}\quad
\vec{F}_i^{-1}(\vec{x}) = \frac{1}{\Delta{x}}(\vec{x} - \vec{a}_{i0}) - \frac{1}{2}
\begin{bmatrix}
1 \\ 
1
\end{bmatrix}.
\end{align*}
Then, the bases on cell $K_i$ are $\varphi_{ij} = \hat{\varphi}_j\circ\vec{F}_i^{-1}$, where $j = 0, \cdots, 3$. 
Let $\hat{\grad} = \transpose{[\hat{\partial}_{\hat{x}}, \hat{\partial}_{\hat{y}}]}$ denote the gradient on $\hat{K}$. We list the gradient of the basis functions on the reference element, as follows:
\begin{align*}
\hat{\grad}\hat{\varphi}_0 = \frac{1}{2}
\begin{bmatrix}
-1+2\hat{y} \\ 
-1+2\hat{x}
\end{bmatrix},&&
\hat{\grad}\hat{\varphi}_1 = \frac{1}{2}
\begin{bmatrix}
1-2\hat{y} \\ 
-1-2\hat{x}
\end{bmatrix},&&
\hat{\grad}\hat{\varphi}_2 = \frac{1}{2}
\begin{bmatrix}
-1-2\hat{y} \\ 
1-2\hat{x}
\end{bmatrix},&&
\hat{\grad}\hat{\varphi}_3 = \frac{1}{2}
\begin{bmatrix}
1+2\hat{y} \\ 
1+2\hat{x}
\end{bmatrix}.
\end{align*}
We index the two faces of $\hat{K}$ which are perpendicular to $x$-axis by $\hat{e}_0$ and $\hat{e}_1$ and index the two faces which are perpendicular to $y$-axis by $\hat{e}_2$ and $\hat{e}_3$, namely 
\begin{align*}
\hat{e}_0&=\{\hat{x}=-1/2,\, -1/2\leq\hat{y}\leq1/2\}, & 
\hat{e}_1&=\{\hat{x}=1/2,\, -1/2\leq\hat{y}\leq1/2\},\\ 
\hat{e}_2&=\{\hat{y}=-1/2,\, -1/2\leq\hat{x}\leq1/2\}, & 
\hat{e}_3&=\{\hat{y}=1/2,\, -1/2\leq\hat{x}\leq1/2\}.
\end{align*}
Define shift mappings with respect to the faces of the reference element as follows:
\begin{align*}
\hat{\vec{\vartheta}}_0(\hat{\vec{x}}) &= \hat{\vec{x}} + \transpose{[1,0]}\quad\text{if}~\hat{\vec{x}}\in\hat{e}_0,&
\hat{\vec{\vartheta}}_1(\hat{\vec{x}}) &= \hat{\vec{x}} - \transpose{[1,0]}\quad\text{if}~\hat{\vec{x}}\in\hat{e}_1,\\
\hat{\vec{\vartheta}}_2(\hat{\vec{x}}) &= \hat{\vec{x}} + \transpose{[0,1]}\quad\text{if}~\hat{\vec{x}}\in\hat{e}_2,&
\hat{\vec{\vartheta}}_3(\hat{\vec{x}}) &= \hat{\vec{x}} - \transpose{[0,1]}\quad\text{if}~\hat{\vec{x}}\in\hat{e}_3.
\end{align*}
\par
Let us evaluate entries in matrix $\vecc{A}_\mathcal{D}$. We consider the thermally insulating boundary condition $\grad{e}\cdot\normal=0$ on the entire boundary of domain $\Omega$. Enforcing part or entire Dirichlet boundary does not break the M-matrix structure.
Let matrix $\vecc{D}=\mathrm{diag}(\vecc{D}_0,\cdots,\vecc{D}_{\Nel-1})$ be a block diagonal matrix, where each diagonal subblock $\vecc{D}_{i'}\in\IR^{4\times 4}$ is defined by: 
for any $j'$, $j \in\{0,\cdots, 3\}$, the entry at ${j'}^\mathrm{th}$ row and ${j}^\mathrm{th}$ column of $\vecc{D}_{i'}$ is the Gauss--Lobatto integral of the expression 
\begin{align*}
&\int_{K_{i'}} \grad{\varphi_{i'j}}\cdot\grad{\varphi_{i'j'}} 
-\frac{1}{2}\sum_{m=0}^3 \int_{e_m\in \Gamma_h} \grad{\varphi_{i'j}}\cdot\normal_{K_{i'}}\,\varphi_{i'j'} 
+ \frac{\tilde{\sigma}}{h}\sum_{m=0}^3 \int_{e_m\in \Gamma_h} \varphi_{i'j}\varphi_{i'j'}\\
=&\, \int_{\hat{K}} \hat{\grad}\hat{\varphi}_j \cdot \hat{\grad}\hat{\varphi}_{j'}
-\frac{1}{2}\sum_{m=0}^3 \iota_m \int_{\hat{e}_m} \hat{\grad}{\hat{\varphi}_j}\cdot\hat{\normal}_{\hat{K}}\, \hat{\varphi}_{j'}
+ \frac{\tilde{\sigma}}{\sqrt{2}}\sum_{m=0}^3 \iota_m \int_{\hat{e}_m}\hat{\varphi}_j\hat{\varphi}_{j'}.
\end{align*}
In above, $\iota_m$ is an indicator, which equals to $1$, if the face $e_m$ of element $K_{i'}$ is an interior face, and otherwise equals to $0$.
We mark all the diagonal entries of $\vecc{A}_\mathcal{D}$ in red color. The diagonal subblocks of $\vecc{A}_\mathcal{D}$ are: for $i' = 0,\cdots,\Nel-1$, 
\begin{align}\label{eq:IIPG_mat_diag_block}
\vecc{D}_{i'} =
\begin{pmatrix}
\red{1 + (\iota_0+\iota_2)(\frac{\tilde{\sigma}}{2\sqrt{2}}-\frac{1}{4})} & -\frac{1}{2} + \frac{\iota_0}{4} & -\frac{1}{2} + \frac{\iota_2}{4} & 0 \\ 
-\frac{1}{2} + \frac{\iota_1}{4} & \red{1 + (\iota_1+\iota_2)(\frac{\tilde{\sigma}}{2\sqrt{2}}-\frac{1}{4})} & 0 & -\frac{1}{2} + \frac{\iota_2}{4} \\ 
-\frac{1}{2} + \frac{\iota_3}{4} & 0 & \red{1 + (\iota_0+\iota_3)(\frac{\tilde{\sigma}}{2\sqrt{2}}-\frac{1}{4})} & -\frac{1}{2} + \frac{\iota_0}{4} \\ 
0 & -\frac{1}{2} + \frac{\iota_3}{4} & -\frac{1}{2} + \frac{\iota_1}{4} & \red{1 + (\iota_1+\iota_3)(\frac{\tilde{\sigma}}{2\sqrt{2}}-\frac{1}{4})}
\end{pmatrix}.
\end{align}
\par
Before computing the off-diagonal subblocks of $\vecc{A}_\mathcal{D}$, let us take a look at an example of a square domain $\Omega=[0,L]^2$, where $L>0$. For any patition of the domain $\Omega$ with more than $2\times2$ square cells, we divide all cells into three categories: all faces are interior faces; only one face is a boundary face; only two faces are boundary faces. See the blue, green, and red cells in the schematic Figure~\ref{fig:2D_IIPG_matrix}.
Using \eqref{eq:IIPG_mat_diag_block}, we get if all faces of a cell $K_{i'}$ are interior faces, then the associated diagonal subblock
\begin{align*}
\vecc{D}_{i'} =
\begin{pmatrix}
\red{\frac{1}{2} + \frac{\tilde{\sigma}}{\sqrt{2}}} & -\frac{1}{4} & -\frac{1}{4} & 0 \\ 
-\frac{1}{4} & \red{\frac{1}{2} + \frac{\tilde{\sigma}}{\sqrt{2}}} & 0 & -\frac{1}{4} \\ 
-\frac{1}{4} & 0 & \red{\frac{1}{2} + \frac{\tilde{\sigma}}{\sqrt{2}}} & -\frac{1}{4} \\ 
0 & -\frac{1}{4} & -\frac{1}{4} & \red{\frac{1}{2} + \frac{\tilde{\sigma}}{\sqrt{2}}}
\end{pmatrix}.
\end{align*}
If only one face of a cell $K_{i'}$ is a boundary face, then dependents on the boundary face location, the associated diagonal subblock belongs to the following four cases.
\begin{align*}
e_0\!\subset\!\partial{\Omega}\!:\!~
\vecc{D}_{i'} \!&=
{\small\begin{pmatrix}
\red{\frac{3}{4} + \frac{\tilde{\sigma}}{2\sqrt{2}}} & -\frac{1}{2} & -\frac{1}{4} & 0 \\ 
-\frac{1}{4} & \red{\frac{1}{2} + \frac{\tilde{\sigma}}{\sqrt{2}}} & 0 & -\frac{1}{4}  \\ 
-\frac{1}{4} & 0 & \red{\frac{3}{4} + \frac{\tilde{\sigma}}{2\sqrt{2}}} & -\frac{1}{2} \\ 
0 & -\frac{1}{4} & -\frac{1}{4} & \red{\frac{1}{2} + \frac{\tilde{\sigma}}{\sqrt{2}}}
\end{pmatrix}}, &
e_1\!\subset\!\partial{\Omega}\!:\!~
\vecc{D}_{i'} \!&=
{\small\begin{pmatrix}
\red{\frac{1}{2} + \frac{\tilde{\sigma}}{\sqrt{2}}} & -\frac{1}{4} & -\frac{1}{4} & 0  \\ 
-\frac{1}{2} & \red{\frac{3}{4} + \frac{\tilde{\sigma}}{2\sqrt{2}}} & 0 & -\frac{1}{4} \\ 
-\frac{1}{4} & 0 & \red{\frac{1}{2} + \frac{\tilde{\sigma}}{\sqrt{2}}} & -\frac{1}{4}  \\ 
0 & -\frac{1}{4} & -\frac{1}{2} & \red{\frac{3}{4} + \frac{\tilde{\sigma}}{2\sqrt{2}}}
\end{pmatrix}},\\
e_2\!\subset\!\partial{\Omega}\!:\!~
\vecc{D}_{i'} \!&=
{\small\begin{pmatrix}
\red{\frac{3}{4} + \frac{\tilde{\sigma}}{2\sqrt{2}}} & -\frac{1}{4} & -\frac{1}{2} & 0 \\ 
-\frac{1}{4} & \red{\frac{3}{4} + \frac{\tilde{\sigma}}{2\sqrt{2}}} & 0 & -\frac{1}{2} \\ 
-\frac{1}{4} & 0 & \red{\frac{1}{2} + \frac{\tilde{\sigma}}{\sqrt{2}}} & -\frac{1}{4}  \\ 
0 & -\frac{1}{4} & -\frac{1}{4} & \red{\frac{1}{2} + \frac{\tilde{\sigma}}{\sqrt{2}}}
\end{pmatrix}}, &
e_3\!\subset\!\partial{\Omega}\!:\!~
\vecc{D}_{i'} \!&=
{\small\begin{pmatrix}
\red{\frac{1}{2} + \frac{\tilde{\sigma}}{\sqrt{2}}} & -\frac{1}{4} & -\frac{1}{4} & 0  \\ 
-\frac{1}{4} & \red{\frac{1}{2} + \frac{\tilde{\sigma}}{\sqrt{2}}} & 0 & -\frac{1}{4}  \\ 
-\frac{1}{2} & 0 & \red{\frac{3}{4} + \frac{\tilde{\sigma}}{2\sqrt{2}}} & -\frac{1}{4} \\ 
0 & -\frac{1}{2} & -\frac{1}{4} & \red{\frac{3}{4} + \frac{\tilde{\sigma}}{2\sqrt{2}}}
\end{pmatrix}}.
\end{align*}
If only two faces of a cell $K_{i'}$ are boundary faces, then dependents on the boundary face location, the associated diagonal subblock belongs to the following four cases.
\begin{align*}
e_0,e_2\!\subset\!\partial{\Omega}\!:\!~
\vecc{D}_{i'} \!&=
{\small\begin{pmatrix}
\red{1} & -\frac{1}{2} & -\frac{1}{2} & 0 \\ 
-\frac{1}{4} & \red{\frac{3}{4} + \frac{\tilde{\sigma}}{2\sqrt{2}}} & 0 & -\frac{1}{2} \\ 
-\frac{1}{4} & 0 & \red{\frac{3}{4} + \frac{\tilde{\sigma}}{2\sqrt{2}}} & -\frac{1}{2} \\ 
0 & -\frac{1}{4} & -\frac{1}{4} & \red{\frac{1}{2} + \frac{\tilde{\sigma}}{\sqrt{2}}}
\end{pmatrix}}, &
e_1,e_2\!\subset\!\partial{\Omega}\!:\!~
\vecc{D}_{i'} \!&=
{\small\begin{pmatrix}
\red{\frac{3}{4} + \frac{\tilde{\sigma}}{2\sqrt{2}}} & -\frac{1}{4} & -\frac{1}{2} & 0 \\ 
-\frac{1}{2} & \red{1} & 0 & -\frac{1}{2} \\ 
-\frac{1}{4} & 0 & \red{\frac{1}{2} + \frac{\tilde{\sigma}}{\sqrt{2}}} & -\frac{1}{4}  \\ 
0 & -\frac{1}{4} & -\frac{1}{2} & \red{\frac{3}{4} + \frac{\tilde{\sigma}}{2\sqrt{2}}}
\end{pmatrix}},\\
e_0,e_3\!\subset\!\partial{\Omega}\!:\!~
\vecc{D}_{i'} \!&=
{\small\begin{pmatrix}
\red{\frac{3}{4} + \frac{\tilde{\sigma}}{2\sqrt{2}}} & -\frac{1}{2} & -\frac{1}{4} & 0 \\ 
-\frac{1}{4} & \red{\frac{1}{2} + \frac{\tilde{\sigma}}{\sqrt{2}}} & 0 & -\frac{1}{4}  \\ 
-\frac{1}{2} & 0 & \red{1} & -\frac{1}{2} \\ 
0 & -\frac{1}{2} & -\frac{1}{4} & \red{\frac{3}{4} + \frac{\tilde{\sigma}}{2\sqrt{2}}}
\end{pmatrix}}, & 
e_1,e_3\!\subset\!\partial{\Omega}\!:\!~
\vecc{D}_{i'} \!&=
{\small\begin{pmatrix}
\red{\frac{1}{2} + \frac{\tilde{\sigma}}{\sqrt{2}}} & -\frac{1}{4} & -\frac{1}{4} & 0  \\ 
-\frac{1}{2} & \red{\frac{3}{4} + \frac{\tilde{\sigma}}{2\sqrt{2}}} & 0 & -\frac{1}{4} \\ 
-\frac{1}{2} & 0 & \red{\frac{3}{4} + \frac{\tilde{\sigma}}{2\sqrt{2}}} & -\frac{1}{4} \\ 
0 & -\frac{1}{2} & -\frac{1}{2} & \red{1}
\end{pmatrix}}.
\end{align*}
\par
Let matrix $\vecc{F} = \vecc{A}_\mathcal{D} - \vecc{D}$, namely $\vecc{F}$ contains all the off-diagonal subblocks of $\vecc{A}_\mathcal{D}$, where each off-diagonal subblock is associated with integrals on a cell face. To be more accurate, each off-diagonal subblock $\vecc{F}_{i'i}^m\in\IR^{4\times 4}$, where $i'\neq i$ and $K_{i'}\cap K_{i} = e_m$ with $m \in\{0,\cdots,3\}$, is defined by:
for any $j'$, $j \in\{0,\cdots, 3\}$, the entry on ${j'}^\mathrm{th}$ row and ${j}^\mathrm{th}$ column of $\vecc{F}_{i'i}^m$ is the Gauss--Lobatto integral of the expression
\begin{align*}
-\frac{1}{2}\int_{e_m} \grad{\varphi_{ij}}\cdot\normal_{K_{i'}}\,\varphi_{i'j'} 
-\frac{\tilde{\sigma}}{h}\int_{e_m} \varphi_{ij}\varphi_{i'j'} 
= -\frac{1}{2}\int_{\hat{e}_m}\hat{\grad}\hat{\varphi}_j\circ\hat{\vec{\vartheta}}_m\cdot\hat{\normal}_{\hat{K}}\,\hat{\varphi}_{j'}
-\frac{\tilde{\sigma}}{\sqrt{2}}\int_{\hat{e}_m}\hat{\varphi}_j\circ\hat{\vec{\vartheta}}_m\hat{\varphi}_{j'}.
\end{align*}
Therefore, the matrix $\vecc{F}$ only contains the following four types of non-zero off-diagonal subblocks, namely when $i'\neq i$ and $K_{i'}\cap K_{i} \neq \emptyset$, 
\begin{align*}
\vecc{F}_{i'i}^0 &=
\begin{pmatrix}
-\frac{1}{4} & \frac{1}{4}-\frac{\tilde{\sigma}}{2\sqrt{2}} & 0 & 0\\ 
0 & 0 & 0 & 0\\ 
0 & 0 & -\frac{1}{4} & \frac{1}{4}-\frac{\tilde{\sigma}}{2\sqrt{2}}\\ 
0 & 0 & 0 & 0\\
\end{pmatrix}, &
\vecc{F}_{i'i}^1 &=
\begin{pmatrix}
0 & 0 & 0 & 0\\ 
\frac{1}{4}-\frac{\tilde{\sigma}}{2\sqrt{2}} & -\frac{1}{4} & 0 & 0\\ 
0 & 0 & 0 & 0\\ 
0 & 0 & \frac{1}{4}-\frac{\tilde{\sigma}}{2\sqrt{2}} & -\frac{1}{4}
\end{pmatrix},\\
\vecc{F}_{i'i}^2 &=
\begin{pmatrix}
-\frac{1}{4} & 0 & \frac{1}{4}-\frac{\tilde{\sigma}}{2\sqrt{2}} & 0\\ 
0 & -\frac{1}{4} & 0 & \frac{1}{4}-\frac{\tilde{\sigma}}{2\sqrt{2}}\\ 
0 & 0 & 0 & 0\\ 
0 & 0 & 0 & 0
\end{pmatrix}, &
\vecc{F}_{i'i}^3 &=
\begin{pmatrix}
0 & 0 & 0 & 0\\ 
0 & 0 & 0 & 0\\  
\frac{1}{4}-\frac{\tilde{\sigma}}{2\sqrt{2}} & 0 & -\frac{1}{4} & 0\\ 
0 & \frac{1}{4}-\frac{\tilde{\sigma}}{2\sqrt{2}} & 0 & -\frac{1}{4}
\end{pmatrix}.
\end{align*}
Obviously, when the penalty parameter $\tilde{\sigma}>\frac{\sqrt{2}}{2}$, the diagonal entries of $\vecc{A}_\mathcal{D}$ are positive. All the off-diagonal entries of $\vecc{A}_\mathcal{D}$ are non-positive. The row sum of $\vecc{A}_\mathcal{D}$ equals zero.
In addition, since the Lagrange bases are numerically orthogonal with respect to the Gauss--Lobatto quadrature, the mass matrix is diagonal with positive diagonal entries $[\vecc{A}_\mathcal{M}]_{ij;ij} = \Delta{x}^2 \hat{\omega}_j \rho_{ij}^\mathrm{P}$. Thus the row sum of matrix $\vecc{A}_{\mathcal{M}} + \frac{\Delta t \lambda}{\Rey}\vecc{A}_\mathcal{D}$ is positive.
Above all, by Lemma~\ref{thm:M_matrix2}, the system matrix $\vecc{A}_{\mathcal{M}} + \frac{\Delta t \lambda}{\Rey}\vecc{A}_{\mathcal{D}}$ is a non-singular M-matrix, therefore is monotone.
Here, we highlight our system matrix holds the M-matrix structure unconditionally.
\begin{figure}[ht!]
\begin{center}
\begin{tabularx}{0.85\linewidth}{@{}C@{}C@{~~}C@{}}
\includegraphics[width=0.2\textwidth]{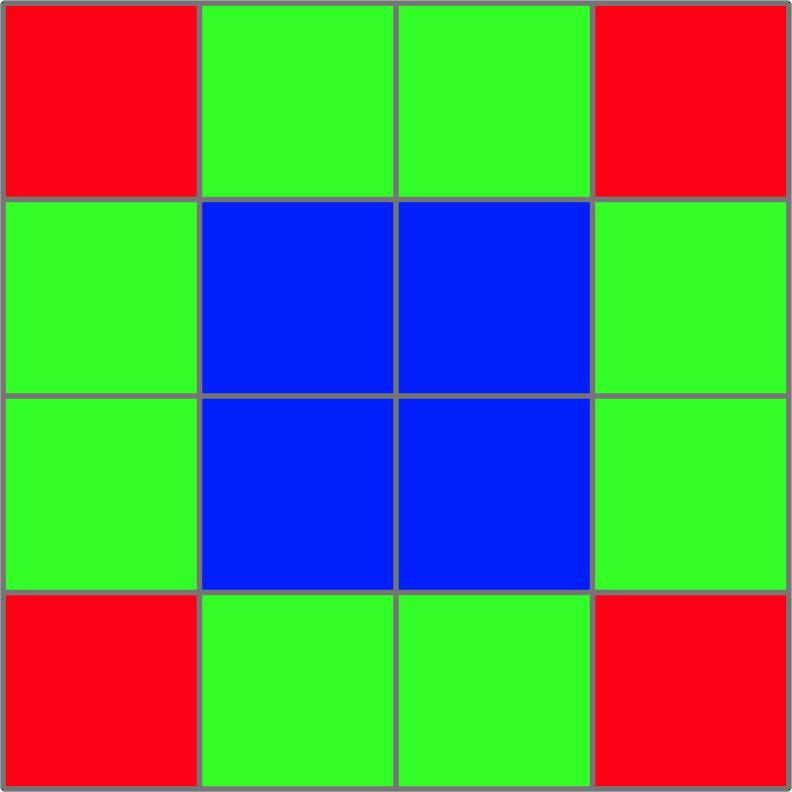} &
\includegraphics[width=0.2\textwidth]{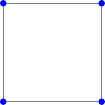} &
\includegraphics[width=0.25\textwidth]{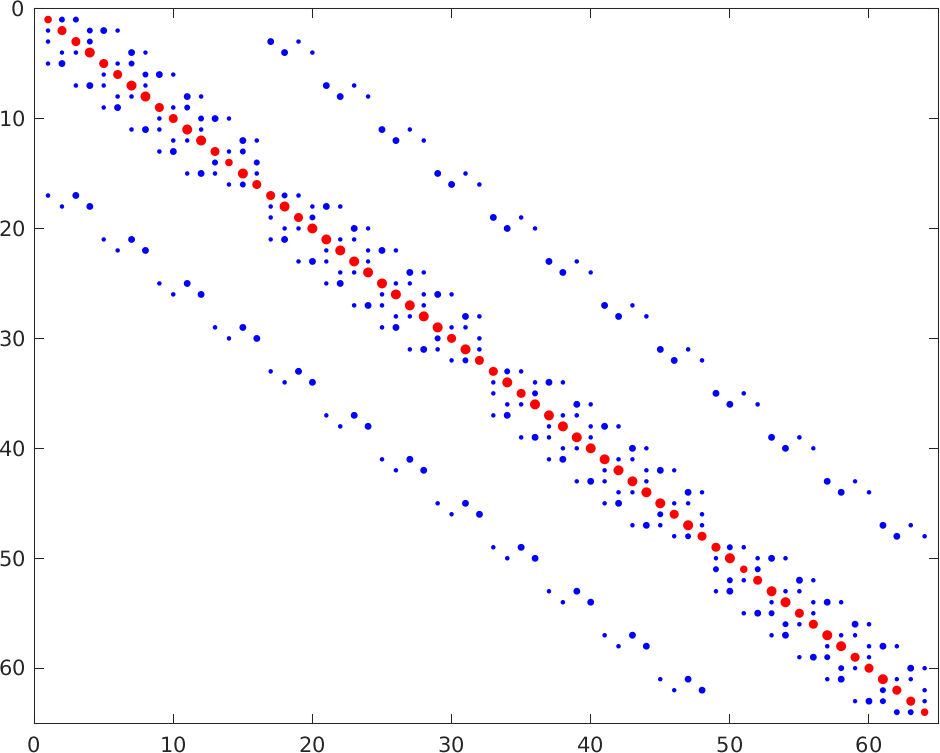}\\
\end{tabularx}
\caption{A schematic graph of the domain partition, quadrature, and the M-matrix structure of $\vecc{A}_\mathcal{D}$. Left: a $4\times4$ mesh of domain $[0,1]^2$. The cells with zero, one, and two boundary faces are marked in blue, green, and red. Middle: $2^2$-point Gauss--Lobatto quadrature used in $\IQ^1$ scheme for computing integrals in parabolic subproblem. Right: sparsity pattern of $\vecc{A}_\mathcal{D}$ associated with a $4\times4$ mesh of the domain $[0,1]^2$. The positive and negative entries are plotted by red and blue dots.}
\label{fig:2D_IIPG_matrix}
\end{center}
\end{figure}

%% file: article.bbl
\begin{thebibliography}{10}
\bibitem{yan2014maximum}
J.~Yan, Maximum principle satisfying direct discontinuous {G}alerkin method and
  its variation for convection diffusion equations, Mathematics of Computation.
\end{thebibliography}
